\documentclass{article}
\usepackage[margin = 1 in]{geometry}
\usepackage{amsmath, amssymb, verbatim, xcolor, amsthm, bbm, mathtools, authblk}
\usepackage[colorlinks = true]{hyperref}
\numberwithin{equation}{section}

\newcommand{\A}{\mathcal{A}}
\newcommand{\B}{\mathfrak{B}}
\newcommand{\cC}{\mathcal{C}}
\newcommand{\D}{\mathcal{D}}
\newcommand{\E}{\mathbb{E}}
\newcommand{\cE}{\mathcal{E}}
\newcommand{\F}{\mathcal{F}}

\newcommand{\cL}{\mathcal{L}}
\renewcommand{\P}{\mathbb{P}}
\newcommand{\R}{\mathbb{R}}
\newcommand{\cR}{\mathcal{R}}
\renewcommand{\S}{\mathcal{S}}
\newcommand{\T}{\mathsf{T}}
\newcommand{\V}{\mathcal{V}}
\newcommand{\tr}{\mathrm{tr}}
\newcommand{\wt}{\widetilde}
\newcommand{\wh}{\widehat}
\newcommand{\wb}{\overline}
\newcommand{\spn}{\text{span}}
\newcommand{\1}{\mathbbm{1}}

\newtheorem{theorem}{Theorem}[section]
\newtheorem{lemma}[theorem]{Lemma}
\newtheorem{prop}[theorem]{Proposition}
\newtheorem{cor}[theorem]{Corollary}
\newtheorem{claim}[theorem]{Claim}
\newtheorem{assump}[theorem]{Assumption}
\theoremstyle{remark}
\newtheorem*{remark}{Remark}

\allowdisplaybreaks

\begin{document}
\title{Langevin Dynamics of Non-reciprocally Coupled Spherical Spin Glasses}
\author[1]{Antonio Auffinger}
\affil[1]{Northwestern University, Department of Mathematics, Evanston, IL, USA.

\texttt{auffing@math.northwestern.edu}}

\author[2]{Ratul Biswas}
\affil[2]{NSF-Simons National Institute for Theory and Mathematics in Biology, Chicago, IL, USA.

\texttt{ratul.biswas@northwestern.edu}}

\date{}
\maketitle
\abstract{We analyze the limiting dynamics of two copies of a disordered mean-field spherical $p$-spin system coupled non-reciprocally to one another and initialized independently of the disorder. We prove that, on every finite time interval, as the system size tends to infinity, the empirical correlation and integrated response functions converge almost surely and in $L^q$ for $q \geq 1$, uniformly in time, to the unique non-random solution of a system of coupled integro-differential equations. For even $p$-spin models, we establish stationarity and exponential decay of the correlation and response functions in the high-temperature regime. In the special case of the spherical Sherrington--Kirkpatrick model, we derive explicit representations of the limiting correlation and integrated response functions. These representations allow us to characterize the long-time dynamics throughout the full range of temperatures. In particular, we prove the occurrence of aging in the low-temperature regime and fluctuation-dissipation behavior in the high-temperature regime, and provide a separate characterization of the dynamics at the critical temperature.}
\section{Introduction}

Non-reciprocal interactions arise naturally in biological systems. For instance, in predator-prey systems, the interaction between species benefits one population and harms the other \cite{All}, and in neural networks the influence of one neuron on another need not be symmetric \cite{Par}. Such interactions also play a role in collective motion and synchronization, where they can generate sustained oscillations, traveling waves, and other time-dependent phases including chaos \cite{Fru, Som}.

A particular natural class of non-reciprocal systems consists of multiple interacting species, with symmetric interactions among members of each species and antisymmetric interactions between members of different species. This structure has been previously known in the context of excitatory-inhibitory neural networks, for example in \cite{Seu}, where such interactions were shown to give rise to dynamics ranging from stable fixed points to limit cycles. Models with such non-reciprocal interactions have recently been studied in the statistical physics literature as well, for example, \cite{Avn} where two non-reciprocally coupled Ising models were studied. In another direction, Lorenzana et al \cite{Lor} analyzed the Langevin dynamics of two identical copies of the spherical Sherrington-Kirkpatrick spin glass model with a non-reciprocal coupling. Their analysis showed evidence that depending on the network structure spin glass systems could retain glassy-dynamics phenomena such as aging, now accompanied by oscillations. Contrary to their findings, previous studies had predicted that any amount of non-reciprocity tends to destroy glassiness. In this paper, we rigorously analyze the dynamics introduced in \cite{Lor} and extend its analysis to the mixed $p$-spin model.

For $x = (x_1, \ldots, x_N)\in \R^N$, let $H_N$ denote the Hamiltonian of a spin glass with mixed $p$-spin interactions, i.e.,\begin{align*}
    H_N(x) = \sum_{p=1}^m H_{N,p}(x) = \sum_{p=1}^m\frac{a_p}{p!}\sum_{1 \leq i_1, \ldots, i_p \leq N}J_{i_1 \ldots i_p}x_{i_1}\cdots x_{i_p}, 
\end{align*}
where $m\geq 2$, $a_m \neq 0$, and the disorder parameters $J_{i_1\ldots i_p} = J_{\{i_1\ldots i_p\}}$ are independent (modulo the permutation of indices) centered normal random variables with variance $N^{-p+1} \prod_k l_k!$, with $(l_1, l_2, \ldots)$ denoting the multiplicities of each element in the multiset $\{i_1, \ldots, i_p\}$. 
Alternatively, if we set $\xi(r) = \sum_{p=1}^m a_p^2 r^p/p!$,  $H_N$ is the centered Gaussian process defined on $\mathbb R^N$ with covariance given by
\[
\mathbb E H_N (x) H_N (y) = N \xi \left(\frac{ x \cdot y }{N}\right),
\]
where $\cdot$ denotes the standard inner product in $\mathbb R^N$.

In this paper we study the Langevin dynamics of two identical copies of $H_N$ coupled with a non-reciprocal parameter. Precisely, let $f:\R \to \R_+$ be a differentiable function such that $f'$ is locally Lipschitz and there exist constants $\kappa <\infty$ and $\rho,\delta> 0$ such that 
\begin{align}\label{cond}
    \sup_{t\geq 0}\frac{|f'(t)|}{(1+t)^{\kappa}}< \infty, \qquad \text{and}\qquad \inf_{t\geq 0} \Big(f'(t) - \rho t^{m/2 + \delta - 1}\Big) > -\infty.
\end{align}

Let $x(t) = (x^1(t), x^2(t)) \in \R^{2N}$, where $x^a(t) = (x_1^a(t), \ldots, x^a_N(t))\in \R^N$ is the spin configuration of the species $a \in\{1,2\}$ at time $t$. At an inverse temperature $\beta >0$ and coupling strength $\alpha >0$, we consider the Langevin dynamics given by \begin{align}\label{lan}
    dx(t) = \Big[ G^\circ(x(t)) - f'\Big(\frac{\|x(t)\|^2}{2N}\Big)x(t)\Big] dt + dB(t),
\end{align}
where 
\[G^\circ(x(t))\coloneqq G(x(t) ) + \alpha (\Theta\otimes I_N) x(t),\] with $\otimes$ denoting the Kronecker product of matrices, $I_N$ the identity matrix of size $N$, \begin{align}
    G(x(t)):  = \begin{pmatrix}
        G^1(x(t))\\ G^2(x(t))
    \end{pmatrix}\coloneqq \beta \begin{pmatrix} \nabla H_N (x^1(t))\\ \nabla H_N (x^2(t))\end{pmatrix}, \qquad \Theta = \begin{pmatrix} 0 &1 \\-1 & 0\end{pmatrix},
    \label{nonrec}
    \end{align}
 and $B(t)$ a standard Brownian motion in $\R^{2N}$. The term containing $f$ in \eqref{lan} is a Lagrange multiplier that implements a `soft' spherical constraint on the spin configuration $x(t)$. We assume that the initial condition $x(0)$, drawn according to a distribution $\mu_{2N}$ on $\R^{2N}$, the disorder $J\coloneqq \{J_{i_1, \ldots, i_p}: 1\leq p\leq m, 1\leq i_1, \ldots, i_p\leq N\}$ in $H_N$ and the Brownian motion $B(t)$ are all independent of each other.

Note that when $\alpha=0$, the drift in \eqref{lan} is of gradient form, with the two copies sharing a common confining potential and following the same stochastic gradient descent dynamics. For $\alpha>0$, the additional forces $+\alpha x^2$ and $-\alpha x^1$ favor alignment of the first species with the second and anti-alignment of the second with the first, creating now an asymmetric imbalance.

Define the following quantities for $a, b\in\{1,2\}$: \begin{align*}
    &C^{ab}_N(s,t) = \frac{1}{N} x^a(s)\cdot x^b(t),&&\chi^{ab}_N(s,t) = \frac{1}{N} x^a(s)\cdot B^b(t)\\
    &A^{ab}_N(s,t)= \frac{1}{N}G^a(x(s))\cdot x^b(t) ,&&F^{ab}_N(s,t) = \frac{1}{N}G^a(x(s))\cdot B^b(t).
    \end{align*}
Denoting by $M^{ab}$ the $(a,b)$th entry of the matrix $M$, we can assemble the above quantities in four $2 \times 2 $ matrices $C_N(s,t)$, $\chi_N(s,t)$, $A_N(s,t)$ and $F_N(s,t)$, where 
    \begin{align*}
    C_N(s,t) & = 
        \begin{pmatrix}
            C_N^{11}(s,t) & C_N^{12}(s,t)\\ C_N^{21}(s,t) & C_N^{22}(s,t) 
        \end{pmatrix}
    \end{align*}
    and the others are defined similarly. Finally, we denote $K_N(t) =C_N(t,t)$. 

Let $T < \infty$ be fixed, and $\cC([0,T])$ denote the set of continuous real-valued functions on $[0,T]$, equipped with the supremum norm.  Let $\cE_{2N} = \R^{2N}\times \R^{d(N,m)}\times \cC([0,T])^{2N}$, where $d(N,m)$ is the dimension of the space of the disorder $J$, and $\R^{2N}$ and $\R^{d(N,m)}$ are equipped with the $L^2$-norm. We endow $\cE_{2N}$ with the product probability measure $\P = \mu_{2N}\otimes \gamma_{d(N,m)}\otimes \rho_{2N}$ where $\mu_{2N}$ is the distribution of $x(0)$, $\gamma_{d(N,m)}$ is the Gaussian law for the disorder and $\rho_{2N}$ is the distribution of the Brownian motion. For $(x(0), J, B) \in \cE_{2N}$, we write \begin{align*}
    \|(x(0), J, B)\|_{\cE_{2N}}^2 =  \sum_{p=1}^m N^{p-1}\sum_{1\leq i_1, \ldots, i_p \leq N} J_{i_1, \ldots, i_p}^2 + \sum_{a = 1, 2}\sum_{i=1}^N \Big((x_i^a(0))^2 +\sup_{0\leq t\leq T}(B_i^a(t))^2\Big).
\end{align*}

We make the following assumptions:

\begin{assump}\label{ass1}
    We assume that $\mu_{2N}$ is such that there exist positive constants $\kappa', \delta'> 0$ independent of $N$ such that if $V$ is a Lipschitz function on $\cE_{2N}$ with Lipschitz constant $L_V$, then for all $t > 0$ \begin{align}\label{hyp_conc}
     \P (|V - \E V| \geq t) \leq \frac{1}{\kappa'}e^{-\kappa'(t/L_V)^{\delta'}}.
 \end{align}
\end{assump}
    
\begin{assump}\label{ass2}
     We assume that there exists a $2\times 2$ positive semidefinite matrix $K(0)$ with finite entries such that \begin{align}\label{initial_cond}
     \lim_{N\to \infty}\|\E K_N(0) - K(0)\|= 0
 \end{align} where $\|\cdot\|$ denotes the induced Euclidean operator norm. We also assume that for each $t>0$, $\P(\|K_N(0) - K(0)\| > t)$ decays exponentially in $N$ and that \begin{align*}
     \lim_{t\to\infty}\limsup_{N\to \infty}\frac{1}{N}\log \P(\|K_N(0)\|\geq t) = -\infty.
 \end{align*}
\end{assump}

\begin{remark}
Exponential tightness at speed $N$ allows the initial-data cutoff to be chosen with any prescribed exponential decay rate. This stronger tail requirement is not imposed in \cite{Ben}, which assumes exponential concentration at each fixed deviation threshold and uniform moment bounds separately. Its role here is to justify localization estimates with arbitrarily large exponential rates. Together with Assumption \ref{ass1}, exponential tightness also implies $\sup_N\mathbb E\|K_N(0)\|^k<\infty$ for every fixed $k\geq 1$, so a separate uniform-moment assumption
is not required.
\end{remark}

We are now ready to present our first result, where we characterize the limiting empirical correlation function $C(s,t)$ and the integrated response function $\chi(s,t)$ for the Langevin dynamics on the soft spherical mixed $p$-spin model \eqref{lan} in terms of the unique solution to a system of non-linear integro-differential equations which are analogous to the Cugliandolo-Kurchan equations introduced in \cite{Cug} and rigorously proved the case $\alpha=0$ for the soft-spherical SK model in \cite{Ben2} and the mixed $p$-spin model in \cite{Ben}.
 
\begin{theorem}\label{thm1}
    Let $T<\infty$ be fixed and let $f$ satisfy \eqref{cond} and suppose that Assumptions \ref{ass1} and \ref{ass2} hold. Then there exist deterministic $2\times 2$ matrices $C(s,t)$ and $\chi(s,t)$ such that, as $N\to \infty$, \begin{align*}
        \|C_N(s,t) - C(s,t)\| \to 0 \qquad \text{and} \qquad \|\chi_N(s,t)-\chi(s,t)\|\to 0
    \end{align*} uniformly on $[0,T]^2$, almost surely and in $L^q$ for any $1\leq q< \infty$ with respect to $x(0),\ J$ and $B$. 
    
    Further, we have that $\chi(s,t) = \int_0^t R(s,u) du$, where $(C, R, K)$ is the unique solution to the following system of equations on $s>t$, \begin{align}
        \partial_sC(s,t) & = - f'\Big(\frac{\tr(K(s))}{2}\Big)C(s,t)  + \beta^2\int_0^t\xi'(C(s,u))R(t,u)^Tdu + \beta^2\int_0^s\Big(\xi''(C(s,u))\odot R(s,u)\Big)C(u,t)du \notag \\
    & \qquad + \alpha \Theta C(s,t),\label{l2.1}\\
        K'(s) & =  \alpha[\Theta K(s) -K(s) \Theta] - 2f'\Big(\frac{\tr(K(s))}{2}\Big)K(s) + \beta^2\int_0^s \Big[\xi'(C(s,u))R(s,u)^T + R(s,u)\xi'(C(u,s))\Big]du\notag \\
            & \qquad + I_2 + \beta^2\int_0^s\Big[\Big(\xi''(C(s,u))\odot R(s,u)\Big)C(u,s) + C(s,u)\Big(\xi''(C(u,s))\odot R(s,u)^T\Big)\Big]du,\label{l2.2}\\
       \partial_s R(s,t) & =  - f'\Big(\frac{\tr(K(s))}{2}\Big)R(s,t)+ \alpha \Theta R(s,t) +  \beta^2\int_t^s\Big(\xi''(C(s,u))\odot R(s,u)\Big)R(u,t) du,\label{l2.3}
       \end{align}
     subject to the conditions \begin{equation} \label{l2.4}
        \begin{split}
            R(s,t) & = 0, R(s,s) = I_2  \text { for all }t >s \geq 0, \\
            K(s) & = C(s,s)  \text { for all }s \geq 0,\\
            C(s,t) & = C(t,s)^T  \text { for all }s, t \geq 0.
        \end{split}
    \end{equation}
    where the matrix $K(0)$ is given in \eqref{initial_cond}.
\end{theorem}

\begin{remark}
We use the convention that the functions $\xi'$ and $\xi''$ act entrywise on the matrices and that $\odot$ denotes the Hadamard product of matrices.
\end{remark}

We now turn to even mixtures, where $\xi(r) = \sum_{p=1}^n a_{2p}^2 r^{2p}/(2p)!$ for some integer $n \geq 1$ and $a_{2n} \neq 0$. Our next result shows that the solutions to the system \eqref{l2.1}--\eqref{l2.3} for the soft spherical case approach the solution to that on the hard sphere as we change the potential function driving the Langevin dynamics.

\begin{theorem}\label{thm2}
    Fix $T,\beta < \infty$ and $\xi$ be an even mixture. For $\gamma \geq 0$, let $f_\gamma(r) = \frac{\gamma}{2}(r-1)^2 + f(r)$, where $f$ is a polynomial satisfying \eqref{cond} and $f'(1) = 1/2$. If $(C_\gamma, R_\gamma, K_\gamma)$ is a solution to \eqref{l2.1}--\eqref{l2.3} satisfying \eqref{l2.4} and $K(0) = I_2$, then as $\gamma \to \infty$, $(C_\gamma, R_\gamma)\to (C, R)$ uniformly in $[0,T]^2$ and $K_\gamma \to I_2$ uniformly on $[0,T]$, where $(C,R)$ is such that for $s\geq t$, \begin{align}
        \partial_sC(s,t) & = (- \mu(s) + \alpha \Theta)C(s,t)  + \beta^2\int_0^t\xi'(C(s,u))R(t,u)^Tdu + \beta^2\int_0^s\Big(\xi''(C(s,u))\odot R(s,u)\Big)C(u,t)du \label{hard_C} \\
    \partial_s R(s,t) & =  (- \mu(s)+ \alpha \Theta )R(s,t) +  \beta^2\int_t^s\Big(\xi''(C(s,u))\odot R(s,u)\Big)R(u,t) du, \label{hard_R}
    \end{align}
    with \begin{align}\label{hard_mu}
        \mu(s) = \frac{1}{2} + \frac{\beta^2}{2}\int_0^s \tr\Big[\xi'(C(s,u))R(s,u)^T +\Big(\xi''(C(s,u))\odot R(s,u)\Big)C(u,s) \Big]du
    \end{align}
    and subject to the conditions \begin{equation} \label{l2.4'}
        \begin{split}
            R(s,t) & = 0  \text { for all }t >s \geq 0, \\
            C(s,s) & = R(s,s) = I_2  \text { for all }s \geq 0,\\
            C(s,t) & = C(t,s)^T  \text { for all }s, t \geq 0.
        \end{split}
    \end{equation}
\end{theorem}

In the next result, we describe the behavior of the limiting equations at large times. For even mixtures, we establish the existence of a stationary regime at sufficiently high temperature: for each fixed time difference, the correlation and response converge as the waiting time tends to infinity to limits depending only on that difference. These limits decay exponentially as the time difference increases and are uniquely characterized by a stationary equation.

\begin{theorem}\label{thm3}
    Let $(C,R)$ denote the unique solution to \eqref{hard_C}-\eqref{hard_R}, subject to the conditions \eqref{l2.4'}. There exists $\beta_* > 0$ depending on $\xi$ such that for all $\beta \leq \beta_*$ and all $\tau \geq 0$, the stationary limits \begin{align*}
    C_*(\tau) \coloneqq \lim_{t \to \infty} C(t+\tau, t) \qquad \text{and} \qquad R_*(\tau) \coloneqq \lim_{t\to \infty} R(t+\tau, t)
    \end{align*}
    exist. Further, $(C_*, R_*)$ is the unique solution to the system
    \begin{align}
        C'_*(t) & = (-\mu_* + \alpha\Theta)C_*(t) + \beta^2 \int_0^\infty \Big(\xi''(C_*(u))\odot R_*(u)\Big)C_*(t-u)du + \beta^2 \int_t^\infty \xi'(C_*(u))R_*(u-t)^T du, \label{C*}\\
        R_*'(t) & = (-\mu_* + \alpha\Theta)R_*(t) + \beta^2 \int_0^t \Big(\xi''(C_*(u))\odot R_*(u)\Big)R_*(t-u)du, \label{R*}
    \end{align}
    where \begin{align}\label{mu*}
        \mu_* = \frac{1}{2} + \frac{\beta^2}{2} \int_0^\infty \tr\Big[\xi'(C_*(u))R_*(u)^T + \Big(\xi''(C_*(u))\odot R_*(u)\Big)C_*(u)^T\Big]du,
    \end{align}
    subject to the conditions \begin{align}\label{l2.4''}
            C_*(0) = R_*(0) = I_2 \qquad \text{and} \qquad C_*(-t) = C_*(t)^T \text{ and } R_*(-t) = 0\text{ for all } t>0
    \end{align}
    in the class of functions satisfying \begin{align*}
        \|C_*(t)\| \leq c_0 e^{-\delta_*t/2} \text{ and } \|R_*(t)\| \leq c_0 e^{-\delta_*t} \text{ for all } t\geq 0
    \end{align*}
    for some $c_0>1$ and $\delta_*<1/4$.
\end{theorem}

We now comment on previous rigorous results related to aging. Ben Arous, Dembo and Guionnet \cite{Ben,Ben2} established convergence to the Cugliandolo--Kurchan equations and low-temperature aging for a single soft-spherical model, the case $\alpha = 0$ here. Aging was also investigated for several dynamics in the Random Energy Model \cite{BBG, BBC, Gay}. Building on the results of \cite{Ben}, Dembo, Guionnet and Mazza \cite{Dem} established the hard spherical limit and, at sufficiently high temperature, convergence to a time-translation-invariant solution satisfying the fluctuation--dissipation relation. More recently, Dembo and Subag \cite{Dem2, Dem3} derived limiting equations for disorder-dependent initial conditions, including Gibbs initialization in replica-symmetric and, under suitable hypotheses, one-step replica-symmetry-breaking regimes.

The remainder of the paper is organized as follows. In Section \ref{sec2}, we derive uniform moment bounds and establish concentration of the functions $C,\chi,A,F$. These results are then used in Section \ref{sec3} to derive the limiting dynamics of the soft spherical model. In Section \ref{sec4}, we establish differentiability and uniqueness of the limiting functions, thereby completing the proof of Theorem \ref{thm1}. In Section \ref{sec5}, we extend the soft spherical results to even mixtures and prove Theorem \ref{thm2}. Section \ref{sec6} is devoted to the proof of Theorem \ref{thm3}, concerning the high-temperature aging phenomenon. In Section \ref{sec7}, we specialize to the SK Hamiltonian and show that our results rigorously recover the transition at $\beta=1/2$ between an aging phase and a time-translation-invariant regime satisfying the fluctuation-dissipation relation after a phase correction, as previously identified in \cite{Lor}. We also characterize the behavior at the critical temperature, refining the results of \cite{Ben2} for the exact spherical model. Throughout Sections \ref{sec2}–\ref{sec5}, we set $\beta=1$ and absorb its dependence into the sequence $(a_p)_p$. The dependence on $\beta$ can be restored by multiplying each occurrence of $\xi'$ and $\xi''$ by $\beta^2$. 

\textit{Statement on AI Use:} Large Language Models were used to assist in developing and refining the proofs of the lemmas in the Appendix, deriving the asymptotic expressions for aging in Section \ref{sec7}, and improving the overall presentation of the manuscript.

\textit{Acknowledgements: } RB would like to thank Giulia Garcia Lorenzana for helpful discussions during the initial part of the project. This research was supported by the NSF-Simons National Institute for Theory and Mathematics in Biology (NITMB) Fellowship supported via grants from the NSF (DMS-2235451) and Simons Foundation (MPS-NITMB-00005320).

\section{Uniform moment bounds and concentration}\label{sec2}

\begin{prop}\label{PropUniq}
For any $N$ and almost any $x(0), J$ and $B$, there exists a unique strong solution to \eqref{lan}. For almost all $x(0)$ and $J$, this solution is unique in law, and the law of the solution is a probability measure on $\cC([0,T])^{2N}$ which we denote by $\P^{2N}_{x(0),J}$. Further, defining \begin{align*}
    \|J\|_\infty^N\coloneqq \max_{1\leq p\leq m} N^{p/2-1}\sup_{\|u^1\|, \ldots, \|u^p\|\leq 1} \Big|\sum_{1\leq i_1,\ldots, i_p \leq N} J_{i_1, \ldots, i_p} u^1_{i_1}\cdots u^p_{i_p}\Big|
\end{align*}
and $S_N(t) = \|x(t)\|^2/(2N)$, we have for $\delta$ as in \eqref{cond} a constant $\eta< \infty$ such that for all $N$, $u> 0$, $J$ and $x(0)$ that 
\begin{align}\label{p1.1}
    \P^{2N}_{x(0),J}\Big(\sup_{0\leq t\leq T}S_N(t) \geq S_N(0) + \eta T (1+\|J\|_\infty^N)^{m/(2\delta) + 1} + u\Big) \leq e^{-uN}
\end{align}
and for any $u > 0$, there exists some $\eta = \eta(u, T)<\infty$ such that
\begin{align}\label{p1.2}
    \P\Big(\sup_{0\leq t\leq T}S_N(t) \geq  \eta\Big) \leq e^{-u N}.
\end{align}
\end{prop}

\begin{proof}
    The argument is very similar to that of \cite[Proposition 2.1]{Ben}, however, we detail the proof here for completeness. For $r > 0$ fixed, let $\phi_r$ be a bounded, globally Lipschitz function on $\R^{2N}$ such that $\phi_r(x) = x$ whenever $\|x \| \leq \sqrt{2Nr}$. We define \begin{align*}
        b_r(x(t)) = G^\circ(\phi_r(x(t))) - f'\Big(r \wedge \frac{\|x(t)\|^2}{2N}\Big)x(t).
    \end{align*}
    Since $f'$ is locally Lipschitz and $\|J\|_\infty^N$ is almost surely finite, $b_r(x(t))$ is globally Lipschitz and hence the existence and uniqueness of a square-integrable strong solution $x_r(t)$ for \begin{align*}
        dx(t) = b_r(x(t))dt + dB(t)
    \end{align*}
    follows (see \cite[Chapter 5, Theorem 2.9]{Kar}). We define $x_r(t)$ on the same probability space, and consider the stopping times \begin{align*}
        \tau_r = \inf\{t\geq0: \|x_r(t)\| \geq \sqrt{2Nr}\}.
    \end{align*} Then, $\tau_r$ is non-decreasing, and $x_r(t)$ is a strong solution for \eqref{lan} in $[0, \tau_r]$. Thus if we can show that $\sum_{r=1}^\infty \P(\tau_r \leq T) <\infty$, by Borel-Cantelli Lemma it will follow that $x(t) = \lim_{r\to \infty}x_r(t)$ is a strong solution of \eqref{lan} in $[0,T]$.

    To that end, we fix $r$ and define $S_N^r(t) = \|x_r(t \wedge \tau_r)\|^2/{2N}$, and $Z(s) = N^{-1}\int_0^{s \wedge \tau_r} x_r(t \wedge \tau_r) \cdot dB(t)$. An application of It\^o's formula to $S_N^r(t)$ yields that \begin{align}\label{p1.3}
        S_N^r(s\wedge \tau_r) & \leq  S_N^r(0) + \sqrt{2}\|J\|_\infty^N\sum_{p=1}^m \frac{2^{p/2}|a_p|}{(p-1)!}\int_0^{s\wedge\tau_r}S_N^r(t)^{p/2}dt - 2\int_0^{s\wedge\tau_r} f'(S_N^r(t))S_N^r(t)dt + Z(s)+ s\wedge\tau_r.
    \end{align}

By \eqref{cond}, $t^{1-m/2}f'(t) \to \infty$ and so we have that there exists a constant $\eta(\|J\|_\infty^N)$ which is almost surely finite and does not depend on $r$ such that \begin{align*}
        S_N^r(s\wedge\tau_r) \leq S_N^r(0) +\eta (\|J\|_\infty^N)s + Z(s).
    \end{align*}
    The quadratic variation of $Z(s)$ is $(2/N)\int_0^{s\wedge\tau_r}S_N^r(t) dt \leq 2sr/N$, so applying Doob's inequality for the exponential martingale \begin{align*}
        \cL_\lambda(s) = e^{\lambda Z(s) - (\lambda^2/N)\int_0^{s\wedge \tau_r} S_N^r(t)dt}
    \end{align*}
    we get for any $u>0$, \begin{align}\label{p1.4}
        \P\Big(\sup_{s\leq T} \Big[Z(s) - \int_0^{s\wedge \tau_r} S_N^r(t)dt\Big] \geq u\Big) \leq \P\Big(\sup_{s\leq T} \cL_N(s) \geq e^{uN}\Big) \leq e^{-uN}.
    \end{align}
    Thus with probability $1-e^{-uN}$ we have for all $s\leq T$, \begin{align*}
        S_N^r(s\wedge \tau_r) \leq S_N^r(0) + \eta(\|J\|_\infty^N) T  + u + \int_0^{s\wedge \tau_r}S_N^r(t)dt
    \end{align*}
    and so by Gr\"onwall's inequality we have \begin{align*}
        \sup_{s\leq T} S_N^r(s\wedge \tau_r) \leq \big(S_N^r(0) + \eta(\|J\|_\infty^N) T + u\big)e^T
    \end{align*}
    so setting $r$ sufficiently large such that $u(r) = e^{-T}r/2 - \eta(\|J\|_\infty^N) T - S_N^r(0) > 0$ we get that $\sup_{s\leq T} S_N^r(s\wedge \tau_r) < r/2$ and hence $\P(\tau_r \leq T) \leq e^{-u(r)N}$, which implies $\sum_{r=1}^\infty \P(\tau_r \leq T) <\infty$.

    We now turn to the proof of \eqref{p1.1}. By Claim \ref{claim_a1}, assumption \eqref{cond} implies that for any $\rho_0 > 0$, there exists $\eta< \infty$ such that for all $u, t\geq 0$, \begin{align*}
        2tf'(t) - u\sum_{p=1}^m\frac{|a_p|(2t)^{p/2}}{(p-1)!} - 1 \geq \rho_0 t - \eta(1+u)^{m/(2\delta) + 1}. 
    \end{align*}
 Hence from \eqref{p1.3} we have that \begin{align*}
            S_N^r(s\wedge \tau_r) \leq S_N^r(0) -\int_0^{s\wedge \tau_r}\Big[\rho_0 S_N^r(t) - \eta(1+ \|J\|_\infty^N)^{m/(2\delta) + 1}\Big]dt + Z(s)
        \end{align*}
    From \eqref{p1.4} we get that with probability $\geq 1-e^{-uN}$, \begin{align*}
        \sup_{t \leq s\wedge \tau_r} Z(t) \leq \int_0^{s\wedge \tau_r} S_N^r(t)dt + u
    \end{align*}
    for all $s\leq T$. Setting $\rho_0 = 2$ we have \begin{align*}
        S_N^r(s\wedge \tau_r) & \leq S_N^r(0) + u - \int_0^{s\wedge \tau_r}S_N^r(t)dt + \eta(1+\|J\|_\infty^N)^{m/(2\delta) + 1}(s\wedge \tau_r)\\
        & \leq S_N(0) + u  + \eta  (1+\|J\|_\infty^N)^{m/(2\delta) + 1}(s\wedge \tau_r).
    \end{align*}
   Taking $r \to \infty$ and noting that $s\leq T$, we arrive at \eqref{p1.1}.

    Finally, fix $u>0$. Since $S_N(0)=\tr (K_N(0))/2\leq\|K_N(0)\|$, the exponential tightness in Assumption \ref{ass2} allows us to choose $\eta_1<\infty$ such that $\P(S_N(0)\geq \eta_1)\leq e^{-2uN}$ for every $N\geq1$. Similarly, \cite[(B.7)]{Ben} allows us to choose $\eta_2<\infty$ such that $\P(\|J\|_\infty^N\geq \eta_2)\leq e^{-2uN}$ for every $N\geq 1$. Averaging \eqref{p1.1} over the initial condition and disorder yields
    \begin{align*}
        \P\Big( \sup_{t\leq T}S_N(t)\geq S_N(0)+\eta_0T(1+\|J\|_\infty^N)^{m/(2\delta) + 1}+2u\Big)\leq e^{-2uN},
    \end{align*}
    and therefore, a union bound gives
    \begin{align*}
    \P\Big(\sup_{t\leq T}S_N(t)\geq\eta\Big) &\leq \P(S_N(0)\geq \eta_1) +\P(\|J\|_\infty^N\geq \eta_2)+e^{-2uN} \leq 3e^{-2uN}\leq e^{-uN},
    \end{align*}
    which proves \eqref{p1.2}.
    \end{proof}

\begin{remark}
The final Gr\"onwall step in \cite[Proposition 2.1]{Ben} overlooks the negative coefficient in \cite[(2.9)]{Ben}, so the conclusion of uniform-in-time pathwise bound there is not justified.
\end{remark}

\begin{prop}\label{p2}
    Let $U_N$ be any of the functions $A_N^{ab}$, $F_N^{ab}$, $C_N^{ab}$, and $\chi_N^{ab}$, where $a, b \in \{1,2\}$. Then, for any fixed $k,\ T <\infty$, the sequence $(\E [\sup_{s,t\leq T}|U_N(s,t)|^k])_N$ is uniformly bounded in $N$. Moreover, for any fixed $T$, the sequence of random functions $(U_N(s,t))_N$ is a.s. pre-compact in $C([0,T]^2)$ endowed with the uniform topology, and the sequence $(\E U_N(s,t))_N$ is also pre-compact in this topology.
\end{prop}

\begin{proof}
The proof follows a similar line of argument as that presented in \cite[Proposition 2.3]{Ben}, however we include some details here. Let $S_N(t) = \|x(t)\|^2/(2N)$ as above, and $P_N(t) = \|B(t)\|^2/(2N)$ and $Q_N(t) = \|G(x(t))\|^2/(2N)$. We denote $\|S_N\|_\infty = \sup_{0\leq t\leq T} S_N(t)$ and similarly for the functions $P_N$ and $Q_N$.

We claim that the following bound holds for each $k < \infty$: \begin{align}\label{p2.1}
    \sup_N \E [(\|J\|_\infty^N)^k] + \sup_N\E [\|S_N\|_\infty^k] + \sup_N\E [\|P_N\|_\infty^k] +\sup_N\E [\|Q_N\|_\infty^k] < \infty.
\end{align}
The uniform moment bounds for $\|J\|_\infty^N$ were provided in \cite[(2.12)]{Ben}. Together with the uniform boundedness of $( [\E\|K_N(0)\|^k])_N$, which implies the uniform boundedness of $(\E[ S_N(0)^k])_N$, and \eqref{p1.1} we have that for each $k \geq 1$, \begin{align}\label{p2.3}
    \sup_N \E \Big[\sup_{0\leq t\leq T}S_N(t)^k\Big] <\infty.
\end{align}
Next, a uniform moment bound for $\|P_N\|_\infty$ is readily obtained from the estimate \cite[(2.16)]{Ben}. Finally, using \cite[(2.11)]{Ben}, we obtain that for some $\eta>0$ independent of $N$, \begin{align}\label{p2.5}
    Q_N(t) & \leq  \eta(\|J\|_\infty^N)^2 \Big(1 + 2^{m-1} S_N(t)^{m-1}\Big),
\end{align}
from which the uniform moment bounds on $\|Q_N\|_\infty$ follow from \cite[(2.12)]{Ben} and \eqref{p2.3}.

From \eqref{p1.2}, \cite[(2.16), (B.7)]{Ben}, and \eqref{p2.5} we have that for each $r>0$, there is some $\eta(r)$ such that for all $N\geq 1$, \begin{align}\label{p2.6}
    \P( \|J\|_\infty^N + \|S_N\|_\infty + \|P_N\|_\infty + \|Q_N\|_\infty \geq \eta(r)) \leq e^{-rN}.
\end{align}
Note that 
\begin{align}\label{p2.61}
    |C_N^{ab}(s,t)|& \leq \frac{1}{2}\Big(C_N^{aa}(s,s) + C_N^{bb}(t,t)\Big) \leq S_N(s) + S_N(t)
    \end{align}
    and similarly, \begin{align*}
    |F_N^{ab}(s,t)| \leq Q_N(s) + P_N(t), \qquad |\chi_N^{ab}(s,t)|  \leq S_N(s) +  P_N(t), \qquad  |A_N^{ab}(s,t)| \leq Q_N(s) + S_N(t).
\end{align*}
Therefore from \eqref{p2.1} and \eqref{p2.6}, for $U_N \in \{ A_N^{ab},  F_N^{ab}, \chi_N^{ab}, C_N^{ab}: a, b = 1, 2\}$  the uniform in $N$ bounds for the moments of $\|U_N\|_\infty \coloneqq \sup_{s, t \in [0,T]} |U_N(s,t)|$ and the exponential in $N$ bounds on the tail probabilities of $\|U_N\|_\infty$ follow immediately.

We now establish equicontinuity of the sequence $(U_N)_N$ on $[0,T]^2$. Once this is established, from the uniform bounds just obtained and the Arzelà-Ascoli theorem, the a.s.\ pre-compactness of the sequence $(U_N)_N$ and the pre-compactness of $(\E U_N)_N$ shall follow.

Let $v,w \in \{x^a, B^a, G^a(x): a = 1, 2\}$. We first show that for every $\varepsilon > 0$ and $\eta > 0$ sufficiently small, there is some $L_{\varepsilon, \eta}$ such that $L_{\varepsilon, \eta}\to \infty$ as $\eta \to 0$ and for all $N$ \begin{align}\label{p2.7}
    \P\Big(\sup_{|t-t'| \leq \eta} \|v(t) - v(t')\|^2\geq N\varepsilon\Big) \leq e^{-L_{\varepsilon, \eta}N}, \qquad \sup_{|t-t'|\leq \eta}\E [\|v(t) - v(t')\|^2] \leq \frac{N}{L_{\varepsilon, \eta}}.
\end{align}
The proof of \eqref{p2.7} follows an argument similar to that in \cite[Proposition 2.3]{Ben}. Standard Brownian modulus estimates give the assertion for $v=B^a$. The integrated form of \eqref{lan}, the polynomial growth bound on $f'$ in \eqref{cond} and estimates \eqref{p2.1} and \eqref{p2.6} then allows us to deduce \eqref{p2.7} for $v= x^a$. Finally, using the polynomial-gradient estimate of \cite[Lemma 2.2]{Ben} we deduce \eqref{p2.7} for $v= G^a(x)$. Now, we express $U_N(s,t) = v(s)\cdot w(t)/N$ for some $v,w \in \{x^a, B^a, G^a(x): a = 1, 2\}$, and apply Cauchy-Schwarz to obtain \begin{align}\label{p2.75}
    |U_N(s,t) - U_N(s',t')| & \leq \frac{1}{N}\Big|[v(s) - v(s')]\cdot w(t)\Big| + \frac{1}{N}\Big|v(s')\cdot[w(t) - w(t')]\Big| \notag \\
    & \leq \Big\|\frac{v(s) - v(s')}{\sqrt{N}}\Big\|\Big\|\frac{w(t)}{\sqrt{N}}\Big\| +\Big\|\frac{w(t) - w(t')}{\sqrt{N}}\Big\|\Big\|\frac{v(s')}{\sqrt{N}}\Big\|.
\end{align}
By \eqref{p2.1}, \eqref{p2.6} and  the first estimate in \eqref{p2.7}, it follows from \eqref{p2.75} that the random moduli of continuity of $U_N$, \begin{align*}
    \sup_{|s-s'| + |t-t'| < \eta} |U_N(s,t) - U_N(s',t')|
\end{align*}
have exponentially summable tails. The Borel–Cantelli lemma yields almost-sure equicontinuity of $(U_N)_N$. The equicontinuity of $(\E U_N)_N$ follows from the second estimate in \eqref{p2.7} and \eqref{p2.75}.
\end{proof}

\begin{prop}\label{p3}
    Suppose that the concentration of measure hypothesis \eqref{hyp_conc} holds. Let $U_N$ be any of the functions $A_N^{ab}$, $F_N^{ab}$, $C_N^{ab}$, and $\chi_N^{ab}$, where $a, b \in \{1,2\}$. Then, for any $u > 0$, \begin{align*}
        \sum_N \P\Big(\sup_{s, t \leq T} |U_N(s,t) - \E U_N(s,t)| \geq u\Big) <\infty, 
    \end{align*} and as a consequence of the uniform moment bounds on $\|U_N\|_\infty$,
    \begin{align}\label{p3.1}
        \lim_{N\to \infty}\sup_{s, t\in [0,T]} \E[|U_N(s,t) - \E U_N(s,t)|^2] =0.
        \end{align} 
\end{prop}

\begin{proof}
    Similar to \cite[Proposition 2.4]{Ben}, the idea is to use the concentration of measure property \cite[Lemma 2.5]{Ben}, which holds under the assumption of hypothesis \eqref{hyp_conc}. To that extent, we let $S_N, P_N, Q_N$ be as above and for $r \geq 0$, define\begin{align*}
        \cL_{N,r} = \{(x(0), J, B) \in \cE_{2N}: \|J\|_\infty^N + \|S_N\|_\infty + \|P_N\|_\infty + \|Q_N\|_\infty \leq r\}.
    \end{align*}
    For the rest of this proof we shall fix $s, t\in [0,T]$ and use $\eta_i, i\geq 1$ to denote constants that depend only on $r$ and $T$. To lighten notation, we shall also write $U_N$ for $U_N(s,t)$.
    
    For $r \geq 0$ sufficiently large, on the event $\cL_{N,r}$, which by \eqref{p2.6}, occurs with probability at least $1 - e^{-\eta_1N}$, we have a pointwise bound for $|U_N|$. Proposition \ref{p2} guarantees the uniform in $N$ bounds for the second moment of $U_N$, so to use \cite[Lemma 2.5]{Ben}, it suffices to show that on $\cL_{N,r}$, $U_N$ is Lipschitz with constant $\eta_2/\sqrt{N}$. More precisely, if $x$ and $\wt x$ are two strong solutions of \eqref{lan} with initial conditions $(x(0), J, B)$ and $(\wt x(0), \wt J, \wt B)$ respectively we want to show that if $(x(0), J, B)$ and $(\wt x(0), \wt J, \wt B)$ are in $\cL_{N,r}$, then \begin{align}\label{p3.2}
        \|U_N - \wt U_N\|_\infty \leq \frac{\eta_2}{\sqrt{N}}\|(x(0), J, B) - (\wt x(0), \wt J, \wt B)\|_{\cE_{2N}}.
    \end{align}

    Assuming \eqref{p3.2}, the remainder of the proof follows as in \cite[Proposition 2.4]{Ben}. For each fixed $s,t\in[0,T]$, \cite[Lemma 2.5]{Ben} gives an exponentially summable bound on the tail probabilities of $|U_N(s,t)-\mathbb{E}U_N(s,t)|$ with constants uniform in $s$ and $t$. Taking a union bound over a  sufficiently fine finite grid of $[0,T]^2$, and utilizing the modulus of continuity bounds for $U_N$ and $\E U_N$ established in Proposition  \ref{p2} extends the summable bound on the tail probabilities of $|U_N(s,t)-\mathbb{E}U_N(s,t)|$ from the finite grid to the supremum over $[0,T]^2$. \eqref{p3.1} then follows from the Borel-Cantelli lemma and the uniform moment bounds of $\|U_N\|_\infty$ from Proposition \ref{p2}.

     We proceed to show \eqref{p3.2} now. To that end, we first establish a Lipschitz property for the strong solution when the initial condition is in $\cL_{N,r}$. More precisely, when $(x(0), J, B)$ and $(\wt x(0), \wt J, \wt B)$ are in $\cL_{N,r}$ and $x, \wt x$ are the strong solutions with corresponding initial conditions, we show that \begin{align}\label{p3.3}
         \sup_{0\leq t\leq T} \|x(t) - \wt x(t)\| \leq \eta_3 \|(x(0), J, B) - (\wt x(0), \wt J, \wt B)\|_{\cE_{2N}}.
     \end{align}
    To prove \eqref{p3.3}, we use \eqref{lan} to write \begin{align*}
    p_N(t)\coloneqq \frac{1}{N}\|x(t) - \wt x(t)\|^2& = \frac{1}{N}[x(t) -\wt x(t)]\cdot \Big[(x(0) - \wt x(0)) + (B(t) - \wt B(t)) +\int_0^t [G(x(u))-G(\wt x(u))]du\\
    & \qquad  + \int_0^t [G(\wt x(u)) - \wt G(\wt x(u))]du -\int _0^t\Big(f'\Big(\frac{\|x(u)\|^2}{2N}\Big) -f'\Big(\frac{\|\wt x(u)\|^2}{2N}\Big)\Big)x(u)du\\
    & \qquad - \int_0^tf'\Big(\frac{\|\wt x(u)\|^2}{2N}\Big)(x(u) - \wt x(u))du+\alpha (\Theta\otimes I_N) \int_0^t [x(u) - \wt x(u)]du\Big]
    \end{align*} 
    where $\wt G(\cdot)$ is defined similarly as $G(\cdot)$ but for the disorder $\wt J$. We now proceed as in \cite[Lemma 2.6]{Ben} to obtain that on $\cL_{N,r}$, for some constant $\eta_4 = \eta(r, T, \alpha)$ that is independent of $N$ and any $\lambda >0$,\begin{align*}
        p_N(t) \leq \eta_4 \Big(\lambda p_N(t) +  \frac{1}{\lambda N}\|(x(0), J, B) - (\wt x(0), \wt J, \wt B))\|_{\cE_{2N}}^2 +\frac{1}{\lambda}\int_0^t p_N(u)du\Big).
    \end{align*}
    \eqref{p3.3} now follows by choosing $\lambda$ small enough and applying Gr\"onwall's inequality. In order to show that $U_N$ is Lipschitz, we express $U_N(s,t) = v(s)\cdot w(t)/N$, where $v,w\in \{x^a, B^a, G^a(x): a = 1, 2\}$, and obtain that \begin{align}\label{p3.4}
    |U_N(s,t) - \wt U_N(s,t)|
    & \leq \Big\|\frac{v(s) - \wt v(s)}{\sqrt{N}}\Big\|\Big\|\frac{w(t)}{\sqrt{N}}\Big\| +\Big\|\frac{w(t) - \wt w(t)}{\sqrt{N}}\Big\|\Big\|\frac{\wt v(s)}{\sqrt{N}}\Big\|.
\end{align}
Thus, \eqref{p3.2} is established for $U_N \in \{C_N^{ab}, \chi_N^{ab}: a, b\in 1, 2\}$, by choosing $v= x^a$ and $w = x^b$ or $B^b$ in \eqref{p3.4} and using \eqref{p3.3} and the definition of $\cL_{N,r}$. To handle the case $U_N \in \{A_N^{ab}, F_N^{ab}: a, b\in 1, 2\}$ we use \cite[Lemma 2.2]{Ben}, \eqref{p3.3} and the fact that on $\cL_{N,r}$ one has $\|J\|_\infty^N + \|S_N\|_\infty \leq r$, to obtain the following estimates by Cauchy-Schwarz (see \cite[Lemma 2.7]{Ben} for details): \begin{align*}
    \|G(x(t)) -\wt G(\wt x(t))\| \leq \eta_5\|(x(0), J, B) - (\wt x(0), \wt J, \wt B)\|_{\cE_{2N}}, \qquad \|G(x(t))\|  \leq \eta_6\sqrt{N} \|J\|_\infty^N (1 + r^{m-1}) \leq \sqrt{N}\eta_7.
\end{align*} 
These estimates, when applied to \eqref{p3.4} with $v = G^a(x)$ and $w = x^b$ or $B^b$, yield \eqref{p3.2} for $U_N \in \{A_N^{ab}, F_N^{ab}: a, b\in 1, 2\}$. 
\end{proof}

For our purposes, the following corollary of Propositions \ref{p2} and \ref{p3} will be useful. 

\begin{cor} \label{c1}
    If $\Psi: \R^l \to \R$ is locally Lipschitz with $|\Psi(z)| \leq \eta(1+ \|z\|^k)$ for some $\eta, l, k < \infty$ and $X_N \in \R^l$ is a random vector such that for every $1\leq j \leq l$, the $j$th coordinate of $X_N$ is $U_N(s_j, t_j)$ for $U_N \in \{A_N^{ab}, F_N^{ab}, C_N^{ab}, \chi_N^{ab}:a, b \in \{1,2\}\}$ and some $s_j, t_j \in [0,T]$, then, \begin{align*}
        \lim_{N\to \infty} \sup_{s_j, t_j} |\E[\Psi(X_N)] - \Psi(\E[X_N])| = 0. 
    \end{align*}
\end{cor}

Corollary \ref{c1} can be proved in the exact same way from Propositions \ref{p2} and \ref{p3} as \cite[Corollary 2.8]{Ben} was established from \cite[Propositions 2.3 and 2.4]{Ben}, so we skip the details.

\section{Limiting dynamics for the soft spherical model}\label{sec3}

Henceforth, unless otherwise noted, we will use $\E[M]$ to denote the expectation taken entrywise of the matrix $M$ over the randomness in $x(0), J, B$. 

The main objective of this section is to establish the following proposition:

\begin{prop}\label{p6}
    Fix $T < \infty$ and let $\psi(r) = \xi'(r) + r\xi''(r)$ act entrywise on matrices. For a $2\times 2$ matrix M, define \begin{align*}
        \D[M] = \xi'(M)\Theta - \xi''(M)\odot[M\Theta].
    \end{align*}
    Then any limit point of the sequence $((\E C_N, \E \chi_N, \E A_N, \E F_N))_N$ with respect to the uniform convergence on $[0,T]^2$ satisfies the following  equations \begin{align}
     C(s,t) & = C(s,0) + \chi(s,t) + \int_0^t \Big[A(u,s)^T -\alpha C(s,u)\Theta- f'\Big(\frac{\tr (K(u))}{2}\Big) C(s,u) \Big]du, \label{p6.1}\\
    \chi(s,t) & =  (s\wedge t)I_2 +\int_0^s \Big[F(u,t) +\alpha\Theta\chi(u,t)-f'\Big(\frac{\tr (K(u))}{2}\Big) \chi(u,t)\Big]du, \label{p6.2} \\
     A(s,t) & = \int_0^{s\vee t}\Big[f' \Big(\frac{\tr (K(u))}{2}\Big)\psi( C(s,u))-\xi''(C(s,u))\odot A(u,s)^T\Big] C(u,t)du - \alpha \int_0^{s \vee t} \D[C(s,u)]C(u,t) du \notag \\
    & \qquad + \xi'(C(s,s\vee t))C(s\vee t,t) - \xi'(C(s,0))C(0,t) - \int_0^{s\vee t} \xi'(C(s,u))A(u,t)du,\label{p6.3}\\
    F(s,t) & = \int_0^s \Big[f' \Big(\frac{\tr(K(u))}{2}\Big)\psi(C(s,u))-\xi''(C(s,u))\odot A(u,s)^T\Big]\chi(u,u\wedge t)du -\alpha \int_0^s \D[C(s,u)]\chi(u,t)du \notag \\
    & \qquad  + \xi'(K(s))\chi(s,s\wedge t) -\int_0^{t\wedge s} \xi'(C(s,u))du- \int_0^s\xi'(C(s,u))F(u,u\wedge t)du, \label{p6.4}
    \end{align}
    in $(\cC([0,T]^2)^{2\times 2})^4$ subject to the following conditions: \begin{equation} \label{p6.cond}
    \begin{split}
        C(s,t) & = C(t,s)^T \text{ for all } t,s \geq 0,\\
        K(s) & = C(s,s) \text{ for all } s \geq 0,\\
        \chi(s,0) & = 0,\ \chi(s,t) = \chi(s,s) \text{ for all } t\geq s\geq 0,\\
        F(s,0) & = 0,\ F(s,t) = F(s,s) \text{ for all } t\geq s\geq 0.
        \end{split}
    \end{equation}
\end{prop}

\begin{proof}

We use the following notation throughout this section. For two sequences of $2\times 2$ non-random matrices $(M_N)_N, (\wt M_N)_N$, we use $M_N\simeq \wt M_N$ to denote that $\|M_N(s,t) - \wt M_N(s,t)\| \to 0$ as $N\to \infty$ uniformly over $s, t \in [0,T]^2$. 

We claim that that the following hold:
\begin{align}
     \E C_N(s,t) & \simeq \E C_N(s,0) + \E \chi_N(s,t) + \int_0^t \Big[\E A_N(u,s)^T -\alpha \E C_N(s,u)\Theta - f'\Big(\frac{\tr( \E K_N(u))}{2}\Big) \E C_N(s,u)\Big]du, \label{p6.5}\\
    \E \chi_N(s,t) & \simeq  (s\wedge t)I_2 + \int_0^s \Big[\E F_N(u,t) +\alpha\Theta\E \chi_N(u,t) -f'\Big(\frac{\tr(\E K_N(u))}{2}\Big) \E \chi_N(u,t)\Big]du, \label{p6.6} \\
    \E A_N(s,t) & \simeq \xi'(\E C_N(s,s\vee t))\E C_N(s\vee t, t) - \xi'(\E C_N(s,0))\E C_N(0,t) - \int_0^{s\vee t} \xi'(\E C_N(s,u))\E A_N(u,t) du \notag \\
    & \qquad + \int_0^{s\vee t}\Big[f' \Big(\frac{\tr (\E K_N(u))}{2}\Big)\psi(\E C_N(s,u))-\xi''(\E C_N(s,u))\odot \E A_N(u,s)^T\Big]\E C_N(u,t)du \notag \\
    & \qquad - \alpha \int_0^{s\vee t} \D[\E C_N(s,u)] \E C_N(u,t)du, \label{p6.7}\\
   \E F_N(s,t) & \simeq \xi' (\E K_N(s))\E \chi_N(s,s\wedge t) -\int_0^{t\wedge s} \xi' (\E C_N(s,u))du- \int_0^s\xi'(\E C_N(s,u))\E F_N(u,u\wedge t)du \notag\\
    & \qquad  + \int_0^s \Big[f' \Big(\frac{\tr(\E K_N(u))}{2}\Big)\psi(\E C_N(s,u))-\xi''(\E C_N(s,u))\odot \E A_N(u,s)^T\Big]\E\chi_N(u,u\wedge t)du \notag \\
    & \qquad -\alpha \int_0^s \D[\E C_N(s,u)] \E \chi_N(u,t)du.  \label{p6.8}
    \end{align}

Assuming the claim, we establish \eqref{p6.1}-\eqref{p6.4} as follows. Since $B(0) = 0$, we have that $\E F_N(s,0) = 0$ for all $N \geq 1$ and $s \geq 0$. Now, letting $\mathcal{G}_s$ be the $\sigma$-algebra generated by $J, x(0)$ and $B(u), u\leq s$, we have for $a, b\in \{1,2\}$ and $t \geq s$, \begin{align*}
    \E F_N^{ab}(s,t) = \E F_N^{ab}(s,s) + \frac{1}{N}\E[G^a(x(s))\cdot \E [B^b(t) - B^b(s)|\mathcal{G}_s]] = \E F_N^{ab}(s,s),
\end{align*}
where the last step utilizes the independence of $B(t) - B(s)$ and $\mathcal{G}_s$. Hence $\E F_N(s,t) = \E F_N(s,s)$ for all $t \geq s$. A similar argument yields that $\E \chi_N(s,0) = 0$ and $\E \chi_N(s,t) = \E \chi_N(s,s)$ for all $N \geq 1$ and $t \geq s \geq 0$. By definition, it follows that $\E C_N(s,t) = \E C_N(t, s)^T$ for all $s,t$.

Let $\Phi(C, \chi, A, F): (\cC([0,T]^2)^{2\times 2})^4\to (\cC([0,T]^2)^{2\times 2})^4$ be the function that maps $(C,\chi,A, F)$ to the difference between the two sides of \eqref{p6.1}-\eqref{p6.4}. Since the boundary and symmetry conditions \eqref{p6.cond} are met by $(\E C_N, \E \chi_N, \E A_N, \E F_N)$, it follows by comparing \eqref{p6.1}-\eqref{p6.4} and \eqref{p6.5}-\eqref{p6.8} that 
\[
\Phi(\E C_N, \E \chi_N, \E A_N, \E F_N) \to 0
\]
as $N\to \infty$ uniformly on $[0,T]^2$. Since $\Phi$ is continuous with respect to the topology of uniform convergence, $\Phi$ evaluated at any limit point of the sequence $(\E C_N, \E \chi_N, \E A_N, \E F_N)_N$ is equal to zero, and the limit point also satisfies the boundary and symmetry conditions \eqref{p6.cond}. 

Thus all that is left is to prove \eqref{p6.5}-\eqref{p6.8}, which we do below.
\end{proof}

\begin{proof}[Proof of \eqref{p6.5} and \eqref{p6.6}] We first write down the integral form of \eqref{lan}:
\begin{align}\label{lan_int}
    x(s) = x(0) + \int_0^s \Big[G(x(u)) + \alpha (\Theta \otimes I_N) x(u)\Big] du  - \int_0^s f'\Big(\frac{\|x(u)\|^2}{2N}\Big)x(u) du + B(s).
\end{align}
We shall write $\bar 1 = 2$ and $\bar 2 = 1$. For $a, b \in \{1,2\}$, we take the inner product of $x^a(s)$ from the above expression with $B^b(t)$ to obtain \begin{align*}
    \frac{1}{N} x^a(s)\cdot B^b(t) & = \frac{1}{N}x^a(0)\cdot B^b(t) + \frac{1}{N}\int_0^s G^a(x(u))\cdot B^b(t) du  + \frac{\alpha}{N}\Theta^{a\bar a}\int_0^s x^{\bar a}(u)\cdot B^b(t)du \\
    & \qquad - \frac{1}{N}\int_0^s f'\Big(\frac{\|x(u)\|^2}{2N}\Big) x^a(u)\cdot B^b(t) du + \frac{1}{N}B^a(s)\cdot B^b(t)
\end{align*}
Assembling this entrywise computation into a matrix and taking averages yields 
\begin{align}\label{p6.9}
        \E \chi_N(s,t) & = \E \chi_N(0,t) + \int_0^s \Big[\E F_N(u,t) +\alpha\Theta\E \chi_N(u,t)\Big]du - \int_0^s\E \Big[f'\Big(\frac{\tr(K_N(u))}{2}\Big) \chi_N(u,t)\Big]du +(s\wedge t)I_2.
    \end{align}
Noting the polynomial growth of $f$ guaranteed by \eqref{cond}, we have from Corollary \ref{c1} that \begin{align}\label{p6.10}
    \E \Big[f'\Big(\frac{\tr(K_N(u))}{2}\Big) \chi_N(u,t)\Big] \simeq f'\Big(\frac{\tr(\E K_N(u))}{2}\Big) \E \chi_N(u,t)
\end{align}
so using this to estimate for the third term of \eqref{p6.9} yields \eqref{p6.6}.

\eqref{p6.5} is obtained in a similar manner. Indeed, by taking the inner product of $x^a(s)$ from \eqref{lan_int} with $x^b(t)$, assembling the entrywise computation into a matrix and taking averages  
    we get \begin{align*}
        \E C_N(s,t) & = \E C_N(0,t) + \int_0^s \Big(\E A_N(u,t) +\alpha\Theta \E C_N(u,t) - \E \Big[f'\Big(\frac{\tr (K_N(u))}{2}\Big) C_N(u,t)\Big]\Big)du + \E \chi_N(t,s)^T\\
        & \simeq \E C_N(0,t) + \int_0^s \Big(\E A_N(u,t) +\alpha\Theta \E C_N(u,t) - f'\Big(\frac{\tr (\E K_N(u))}{2}\Big) \E C_N(u,t)\Big) du + \E \chi_N(t,s)^T,
    \end{align*}
    where the final line utilizes Corollary \ref{c1} in the same spirit as \eqref{p6.10}. Taking a transposition of the above, swapping $s$ and $t$, and utilizing the symmetry condition $C_N(s,t) = C_N(t,s)^T$ we arrive at \eqref{p6.5}.
    \end{proof}
Consider the covariance kernel $k$ defined as follows. For each continuous path $x \in \cC([0,T])^{2N}$ and $s, t \in [0,T]$, we denote \begin{align}\label{defn:k}
    k(s,t) \coloneqq \E _J[G(x(s))(G(x(t)))^T],
\end{align}
where $\E_J$ denotes the expectation only in $J$. Following the computation in \cite[Lemma 3.2]{Ben}, we obtain that
\begin{align}\label{eqn:k}
    k(s,t) & = \frac{1}{N}(\xi''(C_N(s,t))\otimes 1_N)\odot \begin{pmatrix}
        x^1(t) (x^1(s))^T& x^2(t)(x^1(s))^T\\
        x^1(t) (x^2(s))^T & x^2(t)(x^2(s))^T
    \end{pmatrix} + \xi'(C_N(s,t))\otimes I_N
\end{align}
where $1_N$ is the square matrix of size $N$ with all entries equal to 1.

In order to establish \eqref{p6.7} and \eqref{p6.8}, we introduce stochastic integrals involving the covariance kernel $k$, which are to be understood according to the polynomial-kernel convention of \cite{Ben}. More precisely, in computing the stochastic integrals introduced below, we expand each of the polynomials $\xi, \xi'$ appearing in \eqref{eqn:k} into monomials of the form $x^{a_1}_{i_1}(s)\cdots x^{a_l}_{i_l}(s) x^{b_1}_{j_1}(t) \cdots x^{b_{l'}}_{j_{l'}}(t)$ for some non-random integers $a_k, b_k \in \{1,2\}$, $i_k, j_k \in [N]$ and $l, l' \geq 0$, place all factors evaluated at the fixed endpoint $s$ outside of the integral, and integrate the remaining factors in the usual It\^o sense.

Equipped with this convention, we denote, for $0 < \tau \leq T$, $y \in L^2(\{1, \ldots,2N\}, [0,\tau])$ \begin{align}\label{ker_k1}
    (k_\tau y)(s)\coloneqq \int_0^\tau k(s,t) y(t) dt,
\end{align}
for $s\leq \tau$, and extend this definition to stochastic integrals as follows. For $t > 0$, let $\F_t$ be the filtration generated by $\{x(u): 0\leq u\leq t\}$. We shall write \begin{align}\label{ker_k2}
    (k_\tau \circ dy)(s) = \int_0^\tau k(s,t) dy(t),
\end{align}
where $y$ is a continuous semi-martingale with respect to $\F_t$, such that for each $j$, $y_j$ is the sum of a continuous, square-integrable martingale and a continuous, square-integrable, finite variation process, both adapted to $\F_t.$ For two semi-martingales $y$ and $z$, if $X_N(s,t)\coloneqq y(s)\cdot z(t)/N$, we write \begin{align*}
    d_u X_N(u,t)= \frac{1}{N}dy(u)\cdot z(t), \qquad d_uX_N(s,u)\coloneqq \frac{1}{N}y(s)\cdot dz(u).
\end{align*}

We shall require the following Gaussian change of measure result for the proof of \eqref{p6.7} and \eqref{p6.8}.

\begin{lemma}\label{l1}
    For $\tau > 0$ fixed, let $V(t) = \E [G(x(t))| \F_\tau]$ and $W(t) = \E[B(t)|\F_\tau]$ for $t\leq \tau$. Then under $\gamma_{d(N,m)}\otimes \P^{2N}_{x(0), J}$,  there exists a version of $V, W$ such that \begin{align}
        U(t) & \coloneqq  x(t)-x(0) + \int_0^t \Big[f'\Big(\frac{\tr (K_N(u))}{2}\Big) - \alpha (\Theta \otimes I_N)\Big] x(u)du\label{defn:u}\\
        W(t) & = U(t) - \int_0^t V(u) du, \label{eqn:w}
        \end{align}
         are continuous $\F_t$-semimartingales, admitting a decomposition into a square-integrable continuous martingale and a continuous adapted finite-variation process whose total variation is square-integrable. Further for any $t \leq \tau$, \begin{align} 
        (k_\tau\circ dW)(t) & = V(t),\label{p7.1}\\
        (k_\tau \circ dU)(t) & =V(t) + (k_\tau V)(t)\label{p7.2}.
    \end{align} Also, we can choose a version of \begin{align*}
        \Gamma(s,t)\coloneqq \E [(G(x(s)) - V(s))(G(x(t)) - V(t))^T|\F_\tau]
    \end{align*}
    such that for $s, t\leq \tau$, \begin{align}\label{p7.3}
        \int_0^\tau k(s,u)\Gamma(u,t)du + \Gamma (s,t)= k(s,t).
    \end{align}
\end{lemma}

\begin{proof}
    From \eqref{lan}, we have that \begin{align}\label{p6.14}
            U(t) = \int_0^t G(x(u)) du + B(t).
        \end{align} 
        Continuity and integrability properties for the semi-martingales $U(t)$ and $W(t)$ follow straightforwardly from \eqref{defn:u} and \eqref{eqn:w}. The remaining parts of the lemma can be obtained by an argument similar to that presented in \cite[Lemma 3.3]{Ben} once we note that \begin{align*}
        \cL_\tau^N\coloneqq \exp\Big(\int_0^\tau G(x(t))\cdot dU(t) - \frac{1}{2}\int_0^\tau \|G(x(t))\|^2 dt\Big)
    \end{align*}
    defines the Radon-Nikodym derivative associated with the Girsanov transform, i.e., \begin{align*}
        \P_{x(0), J}^{2N}|_{\F_\tau} = \cL_\tau^N  \P_{x(0), 0}^{2N}|_{\F_\tau}.
    \end{align*}
\end{proof}

\begin{proof}[Proof of \eqref{p6.7}]  Defining, for $s,t\leq \tau$, \begin{align*}
    \wh A_N(s,t) = \frac{1}{N}\begin{pmatrix}
        V^1(s) \cdot x^1(t) & V^1(s) \cdot x^2(t)\\ V^2(s) \cdot x^1(t) & V^2(s)\cdot x^2(t)
    \end{pmatrix},
\end{align*}
so we have that $\wh A_N(s,t) = \E [A_N(s,t)|\F_\tau]$ and from \eqref{ker_k1}-\eqref{defn:u}, and \eqref{eqn:k} that for $a, b\in \{ 1,2\}$
\begin{align*}
    \frac{1}{N}(k_\tau V)^a(s)\cdot x^b(t) & = \int_0^\tau \Big[\Big(\xi''(C_N(s,u))\odot \wh A_N(u,s)^T\Big)C_N(u,t) + \xi'(C_N(s,u))\wh A_N(u,t)\Big]^{ab}du,\\
    \frac{1}{N}(k_\tau \circ dU)^a(s)\cdot x^b(t) & = \int_0^\tau f'\Big(\frac{\tr (K_N(u))}{2}\Big) \Big[\psi(C_N(s,u))C_N(u,t)\Big]^{ab}du - \alpha \int_0^\tau \Big[\D[C_N(s,u)]C_N(u,t)\Big]^{ab}du\\
    & \qquad + \int_0^\tau \Big[\Big(\xi''(C_N(s,u))\odot d_u C_N(s,u)\Big)C_N(u,t) + \xi'(C_N(s,u))d_uC_N(u,t)\Big]^{ab}du.
\end{align*}
Thus from \eqref{p7.2} we have that \begin{align}\label{p6.11}
    & \wh A_N(s,t) + \int_0^\tau \Big[\Big(\xi''(C_N(s,u))\odot \wh A_N(u,s)^T\Big)C_N(u,t) + \xi'(C_N(s,u))\wh A_N(u,t)\Big]du \notag \\
    & = \int_0^\tau \Big[\Big(\xi''(C_N(s,u))\odot d_u C_N(s,u)\Big)C_N(u,t) + \xi'(C_N(s,u))d_uC_N(u,t)\Big]du \notag \\
    & \qquad + \int_0^\tau f'\Big(\frac{\tr (K_N(u))}{2}\Big) \psi(C_N(s,u))C_N(u,t)du - \alpha \int_0^\tau \D[C_N(s,u)]C_N(u,t)du.
\end{align}
Using It\^o's lemma \cite[Lemma A.1]{Ben} we have \begin{align}\label{p6.111}
    & \int_0^\tau \Big[\Big(\xi''(C_N(s,u))\odot d_u C_N(s,u)\Big)C_N(u,t) + \xi'(C_N(s,u))d_uC_N(u,t)\Big]^{ab} \notag \\
    & = [\xi'(C_N(s,\tau))C_N(\tau, t)]^{ab} - [\xi'(C_N(s,0))C_N(0,t)]^{ab} - \frac{1}{N}C_N^{ab}(s,t)\int_0^\tau\sum_{l=1,2}\xi''(C_N^{al}(s,u))du \notag \\
    & \qquad - \frac{1}{2N}C_N^{aa}(s,s)\int_0^\tau \Big[\xi'''(C_N(s,u))C_N(u,t)\Big]^{ab} du
\end{align}
 Observe that $\wh A_N$ satisfies the same uniform moment bounds as in Proposition \ref{p2} and the same self-averaging property as in Proposition \ref{p3}. Thus, using Corollary \ref{c1} for the locally Lipschitz functions $\xi'(z_1)z_2$, $\xi''(z_1)z_2z_3$, $\xi'''(z_1)z_2z_3$, $\xi''(z_1)z_2$ and $f'(z_1)\psi(z_2)z_3$,$z_i \in \R$, where we allow for $(\wh A_N^{ab}(s,t))_{a,b=1,2}$ to also be the inputs to these functions as needed, and the uniform convergence with respect to the points $(s_j, t_j)\in [0,T]^2$ we have by \eqref{p6.11} and \eqref{p6.111} that
\begin{align*}
    & \E A_N(s,t) + \int_0^\tau \Big[\Big(\xi''(\E C_N(s,u))\odot \E A_N(u,s)^T\Big)\E C_N(u,t) + \xi'(\E C_N(s,u))\E A_N(u,t)\Big]du\\
    & \simeq \xi'(\E C_N(s,\tau))\E C_N(\tau, t) - \xi'(\E C_N(s,0))\E C_N(0,t) + \int_0^\tau f' \Big(\frac{\tr (\E K_N(u))}{2}\Big)\psi(\E C_N(s,u))\E C_N(u,t)du\\
    & \qquad - \alpha \int_0^\tau \D[\E C_N(s,u)]\E C_N(u,t)du.
\end{align*}
Setting $\tau = s\vee t$ and rearranging terms, we obtain \eqref{p6.7}.
\end{proof}

\begin{proof}[Proof of \eqref{p6.8}]
For brevity, we shall write $G(u)$ in place of $G(x(u))$. For $a, b \in \{1, 2\}$, let us define \begin{align*}
     \Gamma^{ab}(s,t) = \E [(G^a(s) - V^a(s))(G^b(t) - V^b(t))^T|\F_\tau].
 \end{align*}
Again, with the understanding that $\bar1 = 2$ and $\bar 2 = 1$,  we can write for $\tau =s\geq t$ that \begin{align*}
    \E F_N^{ab}(s,t) & = \frac{1}{N}\E [G^a(s)\cdot B^b(t)]  = \frac{1}{N}\E \Big[\E \Big[G^a(s)\cdot \Big(U^b(t) - \int_0^t G^b(u)du \Big)\Big|\F_s\Big]\Big]\\
    & = \frac{1}{N}\E \Big[ \E [G^a(s)|\F_s]\cdot\Big(U^b(t)  - \int_0^t V^b(u)du\Big)\Big] -\frac{1}{N}\E\Big[ \int_0^t \tr(\Gamma^{ba}(u,s))du\Big]
        \end{align*}
        where, in the last line, we used the fact that $\E [G^a(u) - V^a(u)|\F_s] = 0$ for $u\leq s$. Note that $V(t) = \E [G(t)|\F_s]$ and $W(t) = \E[B(t)|\F_s]$ are both $\F_s$-measurable, hence the kernel of $k_s$ is $\F_s$-measurable as well. Thus, using \eqref{eqn:w} and \eqref{p7.1}, we continue from the above display to write
        \begin{align}
            & \E F_N^{ab}(s,t) + \frac{1}{N}\E \int_0^t \tr(\Gamma^{ab}(s,u))du  =
            \frac{1}{N} \E \Big[ V^a(s)\cdot\Big(U^b(t)  - \int_0^t V^b(u)du\Big)\Big] = \frac{1}{N} \E \Big[ V^a(s)\cdot W^b(t)\Big]\notag \\
            & = \frac{1}{N}\E \Big[(k_s\circ dW)^a(s) \cdot W ^b(t)\Big]=\frac{1}{N}\E \Big[(k_s\circ dW)^a(s) \cdot B^b(t)\Big] = \frac{1}{N}\E \Big[\E [(k_s\circ dB)^a(s)|\F_s] \cdot B^b(t)\Big]\label{p6.12}.
        \end{align}
    Now, using \eqref{lan} we have \begin{align}
            & \frac{1}{N}(k_s\circ dB )^a(s)\cdot B^b(t) \notag \\
            & = \frac{1}{N}(k_s\circ dx)^a(s)\cdot B^b(t) + \frac{1}{N}\Big(k_s \Big[f'\Big(\frac{\tr(K_N(\cdot))}{2}\Big)x(\cdot)\Big]\Big)^a(s)\cdot B^b(t)- \frac{1}{N}(k_sG)^a(s)\cdot B^b(t) \notag \\
            & \qquad - \frac{\alpha}{N}\notag (k_s[(\Theta \otimes I_N)x(\cdot)])^a(s)\cdot B^b(t)\\
            & = \sum_{l=1,2}\int_0^s \Big[\xi''(C_N^{al}(s,u))\chi_N^{lb}(u,t) d_u C_N^{al}(s,u) + \xi'(C_N^{al}(s,u))d_u \chi_N^{lb}(u,t)\Big]-\alpha \int_0^s \Big[\D[C_N(s,u)]\chi_N(u,t)\Big]^{ab}du \notag \\
    & \qquad + \sum_{l=1,2}\int_0^s \Big[\Big(f'\Big(\frac{\tr(K_N(u))}{2}\Big)\psi(C_N^{al}(s,u))-\xi''(C_N^{al}(s,u))A_N^{la}(u,s)\Big) \chi_N^{lb}(u,t) - \xi'(C_N^{al}(s,u))F_N^{lb}(u,t)\Big]du.\label{p6.13}
            \end{align}

    Again, we use It\^o's lemma \cite[Lemma A.1]{Ben} to replace the stochastic integrals in \eqref{p6.13} by \begin{align*}
        & \xi'(C_N^{al}(s,s))\chi_N^{lb}(s,t) - \xi'(C_N^{al}(s,0))\chi_N^{lb}(0,t)- \frac{1}{2N}C_N^{aa}(s,s) \int_0^s\xi'''(C_N^{al}(s,u))\chi_N^{lb}(u,t)du \\
        & \qquad - \frac{1}{N}\chi_N^{ab}(s,t)\int_0^s\xi''(C_N^{al}(s,u))du. 
    \end{align*}

    Hence from \eqref{p6.12} we have that \begin{align}\label{p6.15}
        & \E F_N^{ab}(s,t) + \frac{1}{N}\E \int_0^t \tr(\Gamma^{ab}(s,u))du\notag\\
        &  = \E \Big[\xi'(C_N(s,s))\chi_N(s,t) - \xi'(C_N(s,0))\chi_N(0,t)\Big]^{ab}+ \E \int_0^s f'\Big(\frac{\tr(K_N(u))}{2}\Big)\Big[\psi(C_N(s,u))\chi_N(u,t)\Big]^{ab}du\notag\\
    & \qquad -\E  \int_0^s\Big[\Big(\xi''(C_N(s,u))\odot \wh A_N(u,s)^T\Big) \chi_N(u,t) + \xi'(C_N(s,u))F_N(u,t)\Big]^{ab}du  + \E D_N^{ab}(s,t)\notag \\
    & \qquad -\alpha \E \int_0^s \Big[\D[C_N(s,u)]\chi_N(u,t)\Big]^{ab}du -\frac{1}{2N}\E C_N^{aa}(s,s)\int_0^s\Big[\xi'''(C_N(s,u))\chi_N(u,t)\Big]^{ab} du \notag \\
    & \qquad - \frac{1}{N}\E \chi_N^{ab}(s,t)\int_0^s\sum_{l=1,2}\xi''(C_N^{al}(s,u))du
    \end{align}

where \begin{align*}
    D_N^{ab}(s,t)\coloneqq \frac{1}{N}\sum_{l=1,2}\int_0^s \xi'(C_N^{al}(s,u))[G^l(u) - V^l(u)]\cdot B^b(t)du.
\end{align*}
From \eqref{eqn:w} and \eqref{p6.14} we have that \begin{align*}
    B(t) = W(t) - \int_0^t [G(u)-V(u)]du, 
\end{align*}
and thus, \begin{align*}
D_N^{ab}(s,t) & = \frac{1}{N}\sum_{l=1,2}\int_0^s \xi'(C_N^{al}(s,u))[G^l(u) - V^l(u)]\cdot W^b(t)du \\
& \qquad - \frac{1}{N}\sum_{l=1,2}\int_0^s\int_0^t\xi'(C_N^{al}(s,u))[G^l(u) - V^l(u)]\cdot[G^b(v) - V^b(v)]dvdu.
\end{align*}
Taking expectation of the above display and using that $\E[G^l(u) - V^l(u)|\F_s] = 0$ for $u \leq s$, we have that \begin{align*}
    \E D_N^{ab}(s,t) & = -\frac{1}{N}\sum_{l=1,2}\E\int_0^s \xi'(C_N^{al}(s,u)) \int_0^t\tr(\Gamma^{lb}(u,v))dudv.
\end{align*}

We now claim that \begin{align}\label{p6.16}
        \int_0^t \E E_N^{ab}(s,v)dv\simeq 0,
    \end{align}

    where \begin{align*}
        E_N^{ab}(s,v)\coloneqq \xi'(C_N^{ab}(s,v)) - \frac{1}{N}\tr(\Gamma^{ab}(s,v)) - \frac{1}{N}\sum_{l=1,2}\int_0^s \xi'(C^{al}_N(s,u)) \tr(\Gamma^{lb}(u,v))du.
    \end{align*}
    
Assuming the claim, we establish \eqref{p6.8} from \eqref{p6.15} as follows. We observe that each term except $\E D_N^{ab}(s,t)$ appearing on the right side of \eqref{p6.15} is locally Lipschitz in the entries of the matrices $\wh A_N, F_N, \chi_N$ and $C_N$ for which we have established uniform moment bounds and concentration in Propositions \ref{p2} and \ref{p3} respectively. Thus, we apply Corollary \ref{c1} to these terms as we did in the final part of the proof of \eqref{p6.7}, and use \eqref{p6.16} to deduce that \begin{align*}
    \E F_N(s,t) & \simeq \xi' (\E C_N(s,s))\E \chi_N(s,t) - \xi'(\E C_N(s,0))\E \chi_N(0,t) -\int_0^t \xi' (\E C_N(s,u))du\\
    & \qquad  + \int_0^s f' \Big(\frac{\tr(\E K_N(u))}{2}\Big)\psi(\E C_N(s,u))\E \chi_N(u,t)du-\alpha \int_0^s \D[\E C_N(s,u)]\E \chi_N(u,t)du\\
    & \qquad -\int_0^s \Big[\Big(\xi''(\E C_N(s,u))\odot \E A_N(u,s)^T\Big)\E\chi_N(u,t) + \xi'(\E C_N(s,u))\E F_N(u,t)\Big]du,
\end{align*}
from which \eqref{p6.8} follows since $\E \chi_N(0,t) = 0$.

It remains to establish \eqref{p6.16}, which we do now. Let \begin{align*}
    k^{ab}(s,t)\coloneqq \E_J[G^a(s)(G^b(t))^T].
\end{align*} Then, from \eqref{eqn:k} we have that \begin{align*}
    k^{ab}(s,t) & = \frac{1}{N}\xi''(C_N^{ab}(s,t))x^b(t) (x^a(s))^T + \xi'(C_N^{ab}(s,t))I_N. 
\end{align*}
Together with this expression of $k^{ab}(s,t)$ and \eqref{p7.3} we have
\begin{align*}
   0 & = \frac{1}{N} \tr \Big( k^{ab}(s,t) - \Gamma^{ab}(s,t) - \sum_{l=1,2}\int_0^sk^{al}(s,u)\Gamma^{lb}(u,t)du\Big)\\
   & =\frac{1}{N} \xi''(C_N^{ab}(s,t))C_N^{ab}(s,t) + \xi'(C_N^{ab}(s,t)) - \frac{1}{N} \tr(\Gamma^{ab}(s,t)) -\frac{1}{N}\sum_{l=1,2}\int_0^s\tr[\xi'(C_N^{al}(s,u))\Gamma^{lb}(u,t)]du\\
   & \qquad - \frac{1}{N^2}\sum_{l=1,2}\int_0^s\xi''(C_N^{al}(s,u))\tr[x^l(u)(x^a(s))^T \Gamma^{lb}(u,t)]du
\end{align*}
Rearranging terms in the above display, and using that $x(u)$ and $x(s)$ are $\F_s$-measurable, we have \begin{align*}
    & E_N^{ab}(s,t) + \frac{1}{N}\xi''(C_N^{ab}(s,t))C_N^{ab}(s,t)\\
    & = \frac{1}{N^2}\sum_{l=1,2}\int_0^s\xi''(C_N^{al}(s,u))\tr[x^l(u)(x^a(s))^T\E[(G^l(u)-V^l(u))(G^b(t) - V^b(t))^T|\F_s]]du\\
    & = \sum_{l=1,2}\int_0^s \xi''(C_N^{al}(s,u))\E[(A_N^{la}(u,s) - \wh A_N^{la}(u,s))(A_N^{bl}(t,u) - \wh A_N^{bl}(t,u))|\F_s]du
\end{align*}
Since $C_N^{al}(s,u)$ is $\F_s$-measurable as well, we have \begin{align*}
    \E E_N^{ab}(s,t) & = \sum_{l=1,2}\int_0^s\E[\xi''(C_N^{al}(s,u))(A_N^{la}(u,s) - \wh A_N^{la}(u,s))(A_N^{bl}(t,u) - \wh A_N^{bl}(t,u))]du \\
    & \qquad - \frac{1}{N}\E[\xi''(C_N^{ab}(s,t))C_N^{ab}(s,t)]
\end{align*}
We invoke the uniform moment bounds and the concentration property established in Propositions \ref{p2} and \ref{p3} to claim that uniformly in $s\geq t$, all terms in the above display vanish as $N \to \infty$, and hence $\E E_N^{ab}(s,t) \simeq 0$. Utilizing the uniformity of the convergence, \eqref{p6.16} follows.
\end{proof}

\section{Differentiability and uniqueness of the limiting dynamics}\label{sec4}

\begin{lemma}[Differentiability]\label{lem}
    For $T < \infty$ fixed let $(C, \chi, A, F)$ be a solution of the system \eqref{p6.1}-\eqref{p6.4} in $(\cC([0,T]^2)^{2\times 2})^4$ subject to the conditions \eqref{p6.cond}. Then $\chi(s,t) = \int_0^t R(s,u) du$, where $C, R$ and $K(s) = C(s,s)$ satisfy the set of equations \eqref{l2.1}-\eqref{l2.3} and the conditions \eqref{l2.4}.
   
\end{lemma}

\begin{proof}
Consider the following operator on $\cC([0,T])^{2\times 2}$:
 \begin{align*}
     [k_CM](s) \coloneqq -\int_0^s \xi'(C(s,u))M(u)du.
 \end{align*}

Then from \eqref{p6.4} we have that \begin{align*}
    F(s,t) & = [k_CF(\cdot, t)](s) + r(s,t)
\end{align*}
where \begin{align*}
    r(s,t) & = \xi'(K(s))\chi(s,t) + \int_0^s \Big[f' \Big(\frac{\tr(K(u))}{2}\Big)\psi(C(s,u))-\xi''(C(s,u))\odot A(u,s)^T\Big]\chi(u,t)du\\
    & \qquad -\int_0^{s\wedge t} \xi'(C(s,u))du - \alpha \int_0^s\D[C(s,u)]\chi(u,t)du.
\end{align*}
It will be useful to rewrite \begin{align}\label{l2.5}
    r(s,t) = \xi'(K(s))\chi(s,t) + \int_0^s\kappa_1(s,u)\chi(u,t)du - \int_0^{s\wedge t} \xi'(C(s,u))du
\end{align}
where $\kappa_1$ is uniformly bounded $2 \times 2$ matrix and only depends on the functions $C$ and $A$.

Owing to the continuity and uniform boundedness of $\xi'(C(s,u))$ on $[0,T]^2$, it follows from Picard iterations (splitting $[0,T]$ into sufficiently small intervals to guarantee the convergence of the series $\sum_{n\geq 0} k_C^n$) that \begin{align*}
    F(s,t) = \sum_{n\geq 0}[k_C^n r(\cdot, t)](s) = r(s,t) + \int_0^s \kappa_C(s,u)r(u,t)du
\end{align*}
for some uniformly bounded kernel $\kappa_C$. Plugging this in \eqref{p6.2} and applying Fubini's theorem we have that \begin{align*}
    \chi(s,t) & = (s\wedge t)I_2 + \int_0^s\Big(r(u,t) + \int_0^u \kappa_C(u, v) r(v,t)dv\Big)du + \int_0^s \kappa_2(u) \chi(u, t)du\\
    & =  (s\wedge t)I_2 + \int_0^s \kappa_3(s,u)r(u,t)du + \int_0^s \kappa_2(u) \chi(u, t)du,
\end{align*}
where $\kappa_2$ and $\kappa_3$ are uniformly bounded $2 \times 2$ matrices that depend only on $C$. Plugging in \eqref{l2.5} in the display above and using Fubini's theorem again, we have
\begin{align}\label{eqn:6.-1}
    \chi(s,t) & = (s\wedge t)I_2 + \int_0^s \kappa_4(s,u)\chi(u,t)du + \int_0^s\kappa_5(s,u)\Big(\int_0^{u\wedge t}\xi'(C(u,v))dv\Big)du,
\end{align}
where $\kappa_4$ and $\kappa_5$ are uniformly bounded $2 \times 2$ matrices that depend only on $C$ and $A$. Thus, considering the operator \begin{align*}
    [k_{C,A}M](s) = \int_0^s \kappa_4(s,u)M(u)du,
\end{align*}
we can write \eqref{eqn:6.-1} as \begin{align*}
    \chi(s,t) = [k_{C,A}\chi(\cdot, t)](s) + (s\wedge t) I_2 + \int_0^s\kappa_5(s,u)\Big(\int_0^{u\wedge t}\xi'(C(u,v))dv\Big)du,
\end{align*}
and therefore, by Picard's iterations and Fubini's theorem, we have \begin{align*}
    \chi(s,t) & =  (s\wedge t)I_2 + \int_0^s\Big[(u\wedge t) \kappa_6(s,u) + \kappa_7(s,u)\Big(\int_0^{u\wedge t}\xi'(C(u,v))dv\Big)\Big]du 
\end{align*}
for $\kappa_6$ and $\kappa_7$ uniformly bounded $2 \times 2$ matrices that depend only on $C$ and $A$. For $s\geq t$, , since $s\wedge t = t$, the function $\chi(s,t)$ is continuously differentiable in $t$ in the regime $s\geq t$, and $\partial_t\chi(s,t)$ is given by \begin{align*}
    R(s,t) = I_2 + \int_t^s\Big[\kappa_6(s,u) + \kappa_7(s,u)\xi'(C(u,t))\Big]du,
\end{align*}
with $R(s,s) = I_2$ for all $s\geq 0$. Since $\chi(s,t) = \chi(s,s)$ whenever $s<t$, we have that $R(s,t) = \partial_2\chi(s,t) = 0$ for $t> s$. Note that $R$ is uniformly bounded because  $\kappa_6, \kappa_7$ and $C$ also enjoy the same property.

We let $s>t$. Differentiating \eqref{p6.1} in $t$ and using that $R(s,t) = \partial_2\chi(s,t)$ we have that \begin{align}\label{l2.6}
    \partial_2C(s,t) & = R(s,t) + A(t,s)^T -\alpha C(s,t)\Theta- f'\Big(\frac{\tr (K(t))}{2}\Big) C(s,t).
\end{align}
Taking the transpose of the above display and swapping $s$ and $t$, we have \begin{align}\label{l2.7}
    \partial_1C(s,t) & = R(t,s)^T + A(s,t) + \alpha \Theta C(s,t) - f'\Big(\frac{\tr(K(s))}{2}\Big)C(s,t)
\end{align}
We plug in the above two expressions in the identity\begin{align*}
   & \xi'(C(s,s)) C(s, t) - \xi'(C(s,0))C(0,t)\\
   & = \int_0^s\Big[\xi'(C(s,u))\partial_1C(u,t) + \Big(\xi''(C(s,u))\odot\partial_2C(s,u)\Big)C(u,t)\Big]du
\end{align*}
and use \eqref{p6.3} to obtain that \begin{align}\label{l2.8}
    A(s,t) & = \int_0^s\Big[\xi'(C(s,u))R(t,u)^T + \Big(\xi''(C(s,u))\odot R(s,u)\Big)C(u,t)\Big]du.
\end{align}
Putting this back in \eqref{l2.7} and using that $R(t,s)^T = 0$ for $s>t$, we have \begin{align*}
    \partial_1C(s,t) & = - f'\Big(\frac{\tr(K(s))}{2}\Big)C(s,t)  + \int_0^t\xi'(C(s,u))R(t,u)^Tdu + \int_0^s\Big(\xi''(C(s,u))\odot R(s,u)\Big)C(u,t)du\\
    & \qquad + \alpha \Theta C(s,t),
\end{align*}
which is \eqref{l2.1} for $s>t$ and $\beta = 1$.

To establish \eqref{l2.2}, we note that for any $u > 0$, \begin{align*}
    K(s) - K(s-u) & = C(s,s) - C(s-u, s) + C(s-u,s) - C(s-u, s-u)\\
    & = \int_{s-u}^s [\partial_1C(v,s) + \partial_2C(s-u, v)]dv.
\end{align*}
Noting that $R(s,v) = 0$ for $v > s$, we plug \eqref{l2.6} and \eqref{l2.7} in the above equation. Then, using the fact that $A$, $C$, are continuous we divide the expression obtained by $u$ and send $u\to 0$. Thus, we obtain \begin{align*}
   K'(s) & = I_2 + A(s,s) + A(s,s)^T + \alpha[\Theta K(s) -K(s) \Theta] - 2f'\Big(\frac{\tr(K(s))}{2}\Big)K(s),
\end{align*}
where we have also used the fact that $R(s,s) = I_2$. Plugging in the expression for $A(s,s)$ from \eqref{l2.8}, we obtain \eqref{l2.2}:\begin{align*}
    K'(s) & = I_2 + \alpha[\Theta K(s) -K(s) \Theta] - 2f'\Big(\frac{\tr(K(s))}{2}\Big)K(s) + \int_0^s \Big[\xi'(C(s,u))R(s,u)^T + R(s,u)\xi'(C(u,s))\Big]du\\
    & \qquad + \int_0^s\Big[\Big(\xi''(C(s,u))\odot R(s,u)\Big)C(u,s) + C(s,u)\Big(\xi''(C(u,s))\odot R(s,u)^T\Big)\Big]du.
\end{align*}

Let $s>t$. To obtain \eqref{l2.3}, we differentiate \eqref{p6.2} in $s$ to get \begin{align}\label{l2.9}
    \partial_1 \chi(s,t) &  = F(s,t) + \alpha \Theta \chi(s,t) - f'\Big(\frac{\tr(K(s))}{2}\Big)\chi(s,t) + \1_{s<t}I_2.
\end{align}
We plug in \eqref{l2.6} and \eqref{l2.9} in the identity \begin{align*}
    \xi'(K(s))\chi(s,t)  & = \xi'(C(s,s)) \chi(s,t) - \xi'(C(s,0))\chi(0,t)\\
    & = \int_0^s\Big[\xi'(C(s,u))\partial_1\chi(u,t) + \Big(\xi''(C(s,u))\odot\partial_2 C(s,u)\Big)\chi(u,t)\Big]du
\end{align*}
and obtain from \eqref{p6.4} that 
\begin{align*}
    F(s,t) & = \int_0^s \Big(\xi''(C(s,u))\odot R(s,u)\Big)\chi(u,t) du .
\end{align*}
Let \begin{align*}
    E(s,t) & = \int_0^s\Big(\xi''(C(s,u))\odot R(s,u)\Big)R(u,t) du.
\end{align*}
Note that for $s>t$, the integrals in the display above can be restricted to $[t,s]$ owing to the fact that $R(u,t) = 0$ for $u<t$.

By Fubini's theorem we therefore have, $F(s,t) = \int_0^t E(s,u)du$ for all $s\geq t$. Further, since $F(s,t) = F(s,s)$ for $t > s$, it follows that $F(s,t) = \int_0^{s\wedge t} E(s,u)du$ for all $s, t\in [0,T]$. Hence using that $\chi(s,t) = \int_0^t R(s,u)du$, we have from \eqref{p6.2} that for $s>t$ \begin{align*}
    \int_0^tR(s,u)du &  = tI_2 + \int_0^s \int_0^{u\wedge t} E(u,v)dvdu + \int_0^s\Big[\alpha \Theta \chi(u,t) - f'\Big(\frac{\tr(K(u))}{2}\Big)\chi(u,t)\Big]du\\
    & = tI_2 + \int_0^t \int_v^s E(u,v)dudv + \int_0^t\int_0^s\Big[\alpha \Theta R(u,v) - f'\Big(\frac{\tr(K(u))}{2}\Big)R(u,v)\Big]dudv
\end{align*}
Taking a derivative in $t$ yields
\begin{align*}
    R(s,t) & = I_2 + \int_t^s E(u,t)du + \int_0^s\Big[\alpha \Theta R(u,t) - f'\Big(\frac{\tr(K(u))}{2}\Big)R(u,t)\Big]du.
\end{align*}
Taking the derivative in $s$ now yields \begin{align*}
    \partial_1R(s,t) & = E(s,t) + \alpha \Theta R(s,t) - f'\Big(\frac{\tr(K(s))}{2}\Big)R(s,t)
\end{align*}
and replacing the expression of $E(s,t)$ in the display above yields \eqref{l2.3}.
\end{proof}

\begin{prop}[Uniqueness]\label{prop_uniq}
    Let $T < \infty$ be fixed and $\Pi_T = \{(s,t) \in [0,T]^2: t\leq s\}$. Given $K(0)$, there exists at most one solution $(C, R, K) \in (\cC(\Pi_T)^{2\times 2})^2 \times \cC([0,T])^{2\times 2}$ satisfying the equations \eqref{l2.1}-\eqref{l2.3} and the conditions \eqref{l2.4}.
\end{prop}

\begin{proof}
    Let $(C,R,K)$ and $(\wt C,\wt R,\wt K)$ be two solutions of \eqref{l2.1}--\eqref{l2.3} satisfying \eqref{l2.4}. The proof proceeds in three steps. First, we derive a Lipschitz estimate for  $R$ in terms of $C$ and $K$. Second, we use this estimate to obtain integro-differential inequalities for the differences of $C - \wt C$ and $K-\wt K$. Finally, these inequalities are closed by a Gr\"onwall-type argument, which forces all three differences to vanish identically, and thus the two solutions coincide. 
    
      We start with the first step, where we establish the Lipschitz property of $R$ in $C$ and $K$. More precisely, we show that there exists some finite constant $\eta> 0$ that may depend on $T$, the model parameters, and the bounds on the two solutions under comparison such that for all $(s,t) \in \Pi_T$, \begin{align}\label{diff_r}
        \|\Delta R(s,t)\| \leq \eta\Big(\int_t^s \int_u^s \|\Delta C(v,u)\| dvdu + \int_t^s \|\Delta K(u)\| du\Big),
    \end{align}
    where $\Delta R(s,t)\coloneqq R(s,t) - \wt R(s,t)$, and $\Delta C$ and $\Delta K$ are defined analogously. We shall use the same notation for the differences of other matrices and operators appearing later in the proof. In what follows, the value of $\eta$ may change from line to line.
    
    We start by considering \begin{align}\label{defn_H}
       H(s,t) \coloneqq e^{-X(s)} R(s,t) e^{X(t)}, \text{ where } X(t) \coloneqq \int_0^t \Big[-f'\Big(\frac{\tr (K(u))}{2}\Big) I_2 + \alpha \Theta\Big] du. 
   \end{align}
   Then, from \eqref{l2.3} we have that $H(s,t)$ satisfies \begin{align}\label{eqn:diff_h}
       \partial_s H(s,t) & = \int_t^s \A[H(s,u)] H(u,t) du,
   \end{align}
where $\A$ is the operator  \begin{align}\label{defn_a}
\A[H(s,t)] & = e^{-X(s)}\Big(\xi''(C(s,t))\odot [e^{X(s)}H(s,t) e^{-X(t)}]\Big)e^{X(t)} \notag \\
& = e^{-\alpha s \Theta}\Big(\xi''(C(s,t))\odot [e^{\alpha s\Theta}H(s,t) e^{-\alpha t \Theta}]\Big)e^{\alpha t \Theta}.
\end{align}
We integrate this equation, noting that $H(t,t) = I_2$, and use Fubini to obtain
\begin{align}\label{eqn:int_H}
    H(s,t)
    = I_2 + \int_t^s \int_u^s \A[H(v,u)]H(u,t)dvdu.
\end{align}
In Lemma \ref{lem:H} we derive an explicit representation of $H(s,t)$, and show that $H$ satisfies a Lipschitz property in $C$, see \eqref{diff_h}. Observing that $K$ is uniformly bounded and $f'$ is locally Lipschitz, \eqref{diff_r} follows immediately from \eqref{defn_H} and \eqref{diff_h}.

We now proceed to the second step. Integrating \eqref{l2.2} we obtain \begin{align*}
    K(s) & = sI_2 + K(0) + \alpha \int_0^s [\Theta K(u)- K(u) \Theta]du -2\int_0^s f' \Big(\frac{\tr(K(u))}{2}\Big) K(u)du \\
    & \qquad + \int_0^s \int_0^u\Big[\xi'(C(u,v))R(u,v)^T + R(u,v) \xi'(C(v,u))\Big]dvdu\\
    & \qquad + \int_0^s \int_0^u \Big[\Big(\xi''(C(u, v))\odot R(u,v)\Big)C(v,u) + C(u,v)\Big(\xi''(C(v,u))\odot R (u,v)^T\Big)\Big]dvdu.
\end{align*}
Since $K(0) = \wt K(0)$, we have by using the boundedness of $C, K$ and the local Lipschitz property of $f'$ that \begin{align*}
    \|\Delta K(s)\| & \leq \eta\Big[ \int_0^s \|\Delta K (u)\| du + \int_0^s\int_0^v  \Big(\|\Delta C(u,v)\| + \|\Delta R(u,v)\|\Big) dudv\Big]
\end{align*}
and plugging in \eqref{diff_r} above we get \begin{align}\label{diff_k}
    \|\Delta K(s)\| & \leq \eta\Big[ \int_0^s \|\Delta K (u)\| du + \int_0^s\int_u^s\|\Delta C(v,u)\|dvdu \Big].
\end{align}
Similarly, we integrate \eqref{l2.1} to obtain \begin{align*}
    C(s,t) & = K(t) + \alpha \Theta \int_t^s C(u,t)du- \int_t^s f' \Big(\frac{\tr (K(u))}{2}\Big) C(u,t) du + \int_t^s \int_0^t \xi'(C(u,v))R(t,v)^T dvdu\\
    & \qquad + \int_t^s \int_0^t \Big(\xi''(C(u,v))\odot R(u,v)\Big)C(v,t)dvdu + \int_t^s \int_t^u \Big(\xi''(C(u,v))\odot R(u,v)\Big)C(v,t)dvdu. 
\end{align*}
Again, using the boundedness of $C, K$ and the local Lipschitz property of $f'$ we get \begin{align}\label{diff_c}
    \|\Delta C(s,t)\|& \leq \eta\Big[ \|\Delta K(t)\| + \int_t^s \|\Delta C(u,t)\| du + \int_t^s \|\Delta K(u)\|du  + \int_t^s \int_0^t\|\Delta C(u,v)\|dvdu \notag \\
    & \qquad + \int_t^s \int_0^t\|\Delta R(t,v)\|dvdu+\int_t^s \int_0^t\|\Delta R(u,v)\|dvdu  + \int_t^s \int_0^t\|\Delta C(v,t)\|dvdu \notag \\
    & \qquad + \int_t^s \int_t^u\|\Delta C(u,v)\|dvdu + \int_t^s \int_t^u\|\Delta R(u,v)\|dvdu + \int_t^s \int_t^u\|\Delta C(v,t)\|dvdu\Big].
\end{align}

We now come to the final step of the proof. Let \begin{align*}
    \wb C(u) \coloneqq \int_0^u \|\Delta C(u,v)\|dv = \int_0^u \|\Delta C(v,u)\|dv.
\end{align*}
Applying Gr\"onwall's lemma to \eqref{diff_k}, we have \begin{align}\label{gron_k}
    \|\Delta K(s) \| & \leq \eta\int_0^s\int_u^s\|\Delta C(v,u)\|dvdu + \eta^2 \int_0^s e^{\eta(s-w)} \int_0^w\int_u^w\|\Delta C(v,u)\|dvdudw \notag \\
    & \leq ce^{cs} \int_0^s\int_u^s\|\Delta C(v,u)\|dvdu \leq \eta \int_0^s \wb C(u)du.
\end{align}
This, together with \eqref{diff_r} and an application of Fubini's theorem implies that \begin{align}\label{diff_r_bdd}
    \|\Delta R(s,t)\| \leq  \eta\int_0^s \wb C(u)du.
\end{align}
Therefore, by integrating \eqref{diff_c} in $t$ we have using \eqref{gron_k} and Fubini's theorem that \begin{align*}
 & \int_0^s\Big(\|\Delta K(t)\| + \int_t^s \|\Delta K(u)\|du\Big)dt \leq \eta \int_0^s \wb C(u)du,\\
    & \int_0^s \int_t^s \Big(\int_0^t + \int_t^u\Big) \|\Delta C(u,v)\|dvdudt  = \int_0^s \int_t^s \int_0^u\|\Delta C(u,v)\|dvdudt \leq \eta\int_0^s \wb C(u)du\\
    & \int_0^s \int_t^s\Big[\|\Delta C(u,t)\| + \Big(\int_0^t + \int_t^u\Big) \|\Delta C(v,t)\|dv\Big]dudt=\int_0^s \Big[\wb C(u) + 2\int_0^u \wb C(v)dv\Big]du\leq \eta\int_0^s \wb C(u)du,
    \end{align*}
    and additionally, using \eqref{diff_r_bdd}, we have
    \begin{align*}
    \int_0^s \int_t^s \Big[\int_0^t\|\Delta R(t,v)\|dv + \Big(\int_0^t + \int_t^u \Big)\|\Delta R(u,v)\|dv\Big]dudt & \leq \eta \int_0^s \int_0^u \|\Delta R(u,v)\|dvdu \leq \eta\int_0^s \wb C(u)du.
\end{align*}

Assembling the above estimates, we get from \eqref{diff_c} that \begin{align}\label{in:grn}
    \wb C(s) \leq \eta\int_0^s\wb C(u)du,
\end{align}
and hence by Gr\"onwall's lemma, we have that $\wb C(s) = 0$ for all $s \in [0,T]$. Hence, from \eqref{diff_k} and \eqref{diff_r_bdd} it follows that$\Delta K(s) = \Delta R(s,t) = 0$ for all $(s,t) \in \Pi_T$. Finally, from \eqref{diff_c} we have that $\Delta C(s,t) = 0$ as well for all $(s,t) \in \Pi_T$.
\end{proof}

We now complete the proof of Theorem \ref{thm1}.

\begin{proof}[Proof of Theorem \ref{thm1}]
    From Proposition \ref{p2} we have that the sequence $((\E A_N, \E F_N, \E C_N, \E \chi_N))_N$ is precompact with respect to the uniform topology on $[0,T]^2$ and hence admits a limit point, say $(A, F, C, \chi)$. By Proposition \ref{p6}, $(A, F, C, \chi)$ satisfies \eqref{p6.1}-\eqref{p6.4} with the boundary conditions \eqref{p6.cond}, and in turn by Lemma \ref{lem}, we obtain $(C, R = \partial_t\chi)$ associated to this limit point that satisfy the system \eqref{l2.1}-\eqref{l2.3} subject to the conditions \eqref{l2.4}. From Proposition \ref{prop_uniq} we obtain that $(C, R, K)$ is unique solution to the system \eqref{l2.1}-\eqref{l2.3} given the boundary and symmetry conditions \eqref{l2.4}. Thus, we can conclude that $(\E C_N, \E \chi_N)$  converges uniformly in $[0,T]^2$ to $(C,\chi)$, where $\chi(s,t) = \int_0^t R(s,u)du$. By the self-averaging property of $((C_N, \chi_N))_N$ proven in Proposition \ref{p3}, it follows that $((C_N, \chi_N))_N$ converges almost surely to $(C, \chi)$ uniformly on $[0,T]^2$. The uniform convergence of $((C_N,\chi_N))_N$ in $L^q$ for $q\geq 1$ now follows from the uniform moment bounds on $C_N, \chi_N$ established in Proposition \ref{p2}.
\end{proof}

We finish the section by deriving the explicit representation of $H(s,t)$ from the recursive equation \eqref{eqn:int_H}. First, we need to introduce some notation. Let $\B$ denote the set of  finite rooted binary trees. We consider a family of weight matrices $\omega:\B\to \cC(\Pi_T)^{2\times 2}$ defined recursively as follows. For the trivial tree $\emptyset \in\B$, consisting of only the root vertex $\emptyset$, we set $\omega_\emptyset(x,y) = I_2$. For a tree $\T \in \B$  such that $\T_1,\T_2\in \B$ are the left and right children subtrees respectively, we define
\begin{align}\label{wt_const}
    \omega_{\T}(x,y)
    = \int_y^x \int_u^x
    \A[\omega_{\T_1}(v,u)]
    \omega_{\T_2}(u,y)dvdu,
\end{align}
where $\A$ is the operator defined in \eqref{defn_a}. 

Equipped with this notation, we have \begin{lemma}[Binary tree representation of $H$]\label{lem:H}
   The matrix \begin{align}\label{expn:H}
    H(s,t)
    = \sum_{\T\in\B}\omega_{\T}(s,t)
\end{align}
is the unique solution to \eqref{eqn:int_H}. Further, we have that  \begin{align}\label{diff_h}
    \|\Delta H(s,t)\|
    & \leq \eta\int_t^s\int_u^s \|\Delta C(v,u)\|dvdu.
\end{align}
where $\eta$ is a finite constant that depends only on $\alpha, T, \xi$ and the bounds on $C, \wt C$.
\end{lemma}

\begin{proof}
    As earlier, $\eta,\eta'$ shall denote finite positive constants that may depend on $T$, the model parameters, and the bounds on $C, \wt C$ and their values may change from line to line.
    
    First, we show that the series \eqref{expn:H} is absolutely convergent. Clearly, from the uniform boundedness of $C$ and $K$, it follows that $\|\A[M(x,y)]\| \leq \eta\|M(x,y)\|$ for any $M \in \cC(\Pi_T)^{2\times 2}$. Thus, from \eqref{eqn:wt_estimate} we have \begin{align}\label{eqn:bdd_h}
    \|H(s,t) \| & \leq \sum_{n\geq 0}\sum_{\T \in \B:|\T| = n}\|\omega_\T (s,t)\| \leq \sum_{n\geq 0} |\{\T \in \B : |\T| = n\}|\frac{\eta^n (s-t)^{2n}}{(2n)! }  \leq \sum_{n\geq 0} \frac{1}{n+1}\frac{\eta^n T^{2n}}{(n!)^2} < e^{2T \sqrt{\eta}} < \eta',
\end{align}
where we have used that $s,t \in [0,T]$, and the fact that $|\{\T \in \B : |\T| = n\}| = \binom{2n}{n}/(n+1)$ (see \cite[Example 5.3.12]{Sta}). We now verify that the expression of $H(s,t)$ in \eqref{expn:H} satisfies \eqref{eqn:int_H}. Indeed, by the linearity of $\A$ in $H$ we have \begin{align*}
    & I_2 + \int_t^s \int_u^s \A[H(v,u)]H(u,t) dvdu= \omega_\emptyset (s,t) + \sum_{\T_1, \T_2 \in \B}\int_t^s \int_u^s \A[\omega_{\T_1}(v,u)] \omega_{\T_2}(u,t)dvdu \\
    & =\omega_\emptyset (s,t) + \sum_{\T \in \B\setminus \{\emptyset\}} \omega_{\T}(s,t)  = H(s,t).
\end{align*}
The uniqueness follows from a standard Gr\"onwall argument. Indeed,  if $H_1$ and $H_2$ are two solutions to \eqref{eqn:int_H} corresponding to the same operator $\A$, then the difference  $\|H_1 - H_2\|$ satisfies a inequality of the form \eqref{in:grn}, and Gr\"onwall's lemma yields $\|H_1-H_2\| = 0$, implying that $H_1 \equiv H_2$; we skip the details.

We now establish \eqref{diff_h}. Using \eqref{defn_a} we write
 \begin{align*}
     \A[M(v,u)] - \wt \A[M(v,u)] & = e^{-\alpha v\Theta }\Big[\Big(\xi''(C(v,u))-\xi''(\wt C(v,u))\Big)\odot\Big(e^{\alpha v\Theta} M(v,u)e^{-\alpha u\Theta}\Big)\Big]e^{\alpha u\Theta}
 \end{align*}
 and use the facts that $f'$ is locally Lipschitz, $\xi''$ is a polynomial, and $C, \wt C$ are uniformly bounded to obtain that for any $M \in \cC(\Pi_T)^{2\times 2}$, \begin{align*}
     \|\A[M(v,u)] - \wt \A[M(v,u)]\| \leq \eta\|M(v,u)\|\|\Delta C(v,u)\|.
 \end{align*}
Setting $\zeta(v,u) = \eta\|\Delta C(v,u)\|$ in Lemma \ref{lem_app}, \eqref{diff_w} yields that \begin{align*}
    \|\Delta \omega_\T (x,y)\| \leq \eta'|\T|\eta^{|\T|-1}\frac{(x-y)^{2(|\T|-1)}}{(2(|\T|-1))!}\int_y^x\int_u^x \|\Delta C(v,u)\|dvdu.
\end{align*}
Equipped with this estimate we use \eqref{expn:H} to obtain \begin{align*}
    \|\Delta H(s,t)\| & \leq \eta'\sum_{n\geq 1}\frac{n\eta^{n-1}(s-t)^{2(n-1)}}{(2(n-1))!}|\{\T \in \B: |\T| = n\}|\int_t^s \int_u^s \|\Delta C(v,u)\|dvdu \\
    & \leq 4(\eta')^2e^{2\sqrt{\eta}(s-t)}\int_t^s\int_u^s \|\Delta C(v,u)\|dvdu.
\end{align*} 
\end{proof}

\section{Limiting dynamics for the spherical model}\label{sec5}

In this section, we restrict our attention to even mixtures, i.e., $\xi(r) = \sum_{p=1}^n a_{2p}^2r^{2p}/(2p)!$ for some fixed integer $n \geq 1$. Under this assumption, we establish the limiting dynamics for the spherical model mentioned in Theorem \ref{thm2}. 

We consider a sequence of soft-spherical models indexed by a stiffness parameter $\gamma > 0$, where the Langevin dynamics is governed by the constraint potential $f_\gamma(r) = \frac{\gamma}{2}(r-1)^2 + f(r)$, where $f$ is a polynomial that satisfies \eqref{cond} and $f'(1) = 1/2$. For each $\gamma$, we denote the corresponding unique macroscopic solution to \eqref{l2.1}--\eqref{l2.3} satisfying \eqref{l2.4} and $K(0) = I_2$ by $(C_\gamma,R_\gamma,K_\gamma)$. Our first observation is that the macroscopic equations preserve a specific commutative algebra. Let $\V \coloneqq \spn\{I_2, \Theta\}$. Since $\Theta^2 = -I_2$, $\V$ is isomorphic to the complex field $\mathbb{C}$ and is closed under addition, multiplication, and transposition. 

\begin{lemma}[Algebraic Closure] \label{lem:algebra}
    For any $\gamma > 0$, the matrices $C_\gamma(s,t), R_\gamma(s,t) \in \V$ and $K_\gamma(s) = q_\gamma(s) I_2$ for all $s,t \in [0,T]$, where $q_\gamma \in \cC([0,T])$ is non-negative.
\end{lemma}

\begin{proof}
    For any matrix $C\in \V$, we can write \begin{align*}
        C = c_d I_2 + c_a \Theta = \begin{pmatrix} c_d & c_a\\ -c_a & c_d \end{pmatrix}.
    \end{align*}
    Since $\xi$ contains only even powers $\xi''$ is an even function, and thus $\xi''$ acting entrywise on the matrix $C$ leads to $\xi''(C_{21}) = \xi''(-c_a) = \xi''(c_a) = \xi''(C_{12})$. Therefore, $\xi''(C)$ is a symmetric matrix. Consequently, taking the Hadamard product with any $R = r_d I_2 + r_a \Theta \in \V$ yields:
    \begin{align*}
        \xi''(C) \odot R = \begin{pmatrix} \xi''(c_d)r_d & \xi''(c_a) r_a\\ -\xi''(c_a) r_a & \xi''(c_d)r_d \end{pmatrix} = \xi''(c_d)r_d  I_2 + \xi''(c_a) r_a \Theta \in \V.
    \end{align*}
   Note that the initial condition $K_\gamma(0) = I_2 \in \V$ and the drift terms $\alpha \Theta$ and $I_2$ are in $\V$. Further, by the above computation, we have that the non-linear term $(\xi''(C(s,u))\odot R(s,u))C(u,s)$ remains in $\V$. Since $\xi'$ is an odd function, $\xi'(C(s,u)) \in \V$ for $C \in \V$, and hence $\xi'(C(s,u)) R(s,u)^T \in \V$ for $R(s,u)\in \V$. Therefore, Picard iterations applied to \eqref{l2.1}--\eqref{l2.3} preserve the algebra $\V$ at every step. Because $C_\gamma(s,s)$ is symmetric and belongs to $\V$, it must be a multiple of the identity, so $K_\gamma(s) = q_\gamma(s)I_2$.
\end{proof}

To pass to the spherical limit $\gamma \to \infty$, we must bound the solutions uniformly in $\gamma$. In the traditional single-species models considered in \cite{Dem}, this requires pointwise positivity of the function $R(s,t)$, which fails here due to the rotations generated by $\alpha \Theta$. Instead, we bypass positivity by exploiting the physical energy bounds inherited from the pre-limiting system.

\begin{lemma}[Uniform $\gamma$-Bounds] \label{lem:gamma_bounds}
    There exist constants $\eta_C, \eta_R, \eta_\mu, \eta_f < \infty$ and $\gamma_0 = \gamma_0(\eta_C, \eta_f) < \infty$, such that for all $s, t \in [0,T]$ and $\gamma \geq \gamma_0$:
    \begin{align}
        \|C_\gamma(s,t)\| \leq \eta_C, \qquad \|R_\gamma(s,t)\| \leq \eta_R, \qquad \text{and} \qquad |\mu_\gamma(s)| \leq \eta_\mu,
    \end{align}
    where $\mu_\gamma(s) = f'_\gamma(q_\gamma(s))$ is the Lagrange multiplier implementing a soft spherical constraint.
\end{lemma}

\begin{proof}
     Let $x_\gamma(t)$ denote the unique strong solution to \eqref{lan} under potential $f_\gamma$, and define $S_{N,\gamma}(t)=\|x_\gamma(t)\|^2/(2N)$. Applying It\^o's formula to $S_{N,\gamma}$ we observe that the contribution of the confining potential is equal to \begin{align*}
         -2f_\gamma'(S_{N,\gamma})S_{N,\gamma} = -2f'(S_{N,\gamma})S_{N,\gamma} -2\gamma(S_{N,\gamma}-1)S_{N,\gamma}.
     \end{align*}
     Hence, on the region \(\{S_{N,\gamma}\ge 1\}\), we have $-f_\gamma'(S_{N,\gamma})S_{N,\gamma} \leq -f'(S_{N,\gamma})S_{N,\gamma}$ uniformly in $\gamma \geq 0 $. Consequently, the upper bounds on $\sup_{t\leq T} S_{N,\gamma}(t)$ derived via the Lyapunov argument in Proposition \ref{PropUniq} depend only on the initial data and $f$, and thus hold uniformly in $\gamma \geq 0$. Taking the macroscopic limit $N \to \infty$, we obtain that for some constant $\eta_q< \infty$, $\sup_{t\leq T} q_\gamma(t) \leq \eta_q$ uniformly over $\gamma \geq 0$. Since $C_\gamma$ is the macroscopic limit of $\E C_{N,\gamma}$, the uniform bound on $q_\gamma$ together with \eqref{p2.61} yields that $\sup_{s,t\leq T}\|C_\gamma(s,t)\| \leq \eta_q \eqqcolon \eta_C$.
    
    Next, we establish a macroscopic bound on the energy memory term defined as
    \begin{align*}
        r_\gamma(s) \coloneqq \frac{1}{2} \tr \int_0^s \Big[ \xi'(C_\gamma(s,u))R_\gamma(s,u)^T + \Big(\xi''(C_\gamma(s,u))\odot R_\gamma(s,u)\Big)C_\gamma(u,s) \Big] du.
    \end{align*}
    By \eqref{l2.8} and the definition of $A_{N,\gamma}(s,t)$, we have that $r_\gamma(s)$ is the limit of $\frac{1}{2N}\E[G(x_\gamma(s)) \cdot x_\gamma(s)]$. By Euler's homogeneous function theorem for polynomials, 
    \begin{align*}
        G^a(x_\gamma(s)) \cdot x_\gamma^a(s) = \nabla H_N(x_\gamma^a(s)) \cdot x_\gamma^a(s) = \sum_{p=1}^n 2p H_{N,2p}(x_\gamma^a(s)), 
    \end{align*}
    and therefore, we can bound the pure-$2p$ part of the Hamiltonian as
    \begin{align*}
         |H_{N,2p}(x_\gamma^a)| \leq \frac{|a_{2p}|}{(2p)!} \|J\|_\infty^N N^{1-p} \|x_\gamma^a\|^{2p} \leq N \frac{|a_{2p}|}{(2p)!} \|J\|_\infty^N (2 S_{N,\gamma}(s))^p.
    \end{align*}
    Summing over $a \in \{1,2\}$ and applying the uniform-in-$N$ bounds on $\E[\|J\|_\infty^N]$ and $\E [S_{N,\gamma}(s)^p]$, the almost sure convergence of $S_{N,\gamma}(s) \to q_\gamma(s)$ uniformly on $[0,T]$, and $\sup_{t\leq T,\gamma \geq 0} q_\gamma(t) \leq \eta_q$, the limit $N \to \infty$ yields the bound $|r_\gamma(s)| \leq \eta_r$, where $\eta_r$ is strictly independent of $\gamma$ and $R_\gamma$.

    We now establish a uniform lower bound on $q_\gamma(s)$. Taking the trace of \eqref{l2.2} gives:
    \begin{align}\label{eqn:diff_q}
        q_\gamma'(s) = 1 - 2\mu_\gamma(s)q_\gamma(s) + 2r_\gamma(s) = 1 - 2\gamma(q_\gamma(s)-1)q_\gamma(s) - 2f'(q_\gamma(s))q_\gamma(s) + 2r_\gamma(s).
    \end{align}
    Since $0\leq q_\gamma(s) \leq \eta_q$, we have $|2f'(q_\gamma(s))q_\gamma(s)| \leq \eta_{f,q}$, for some constant $\eta_{f,q}$ depending only on $f$ and $\eta_q$. Using our bound on $r_\gamma$, we get:
    \begin{align*}
        q_\gamma'(s) \geq 1 + 2\gamma(1-q_\gamma(s))q_\gamma(s) - \eta_{f,q} - 2\eta_r.
    \end{align*}
    If $q_\gamma(s)$ reaches $1/2$, we have $q_\gamma'(s) \geq 1 + \gamma/2 - \eta_{f,q} - 2\eta_r$. For $\gamma \geq \gamma_0\geq 2(\eta_{f,q} + 2\eta_r)$, we have $q_\gamma'(s) > 0$. Since $q_\gamma(0) = 1$, $q_\gamma(s)$ can never drop below $1/2$. Hence, $q_\gamma(s) \ge 1/2$ uniformly in $\gamma \geq \gamma_0$.

    We bound $R_\gamma$ now. Recall from \eqref{defn_H} that \begin{align}\label{eqn:exp_r}
        R_\gamma(s,t) = \exp\Big(-\int_t^s \mu_\gamma(u)du+\alpha(s-t)\Theta\Big) H_\gamma(s,t).
    \end{align}
    Because $e^{\alpha \Theta}$ is a rotation matrix, its operator 2-norm on $\V$ is exactly 1. As noted in \eqref{defn_a}, the scalar terms $e^{\pm \int_0^\cdot \mu_\gamma(u)du}$ perfectly commute and cancel out inside the operator $\mathcal{A}_\gamma$. Thus, $\|\mathcal{A}_\gamma[H_\gamma]\| \leq \|\xi''(C_\gamma)\| \|H_\gamma\|$. Since $\|C_\gamma\| \leq \eta_q$, the binary tree expansion in Lemma \ref{lem:H} implies $\|H_\gamma(s,t)\| \leq \eta_H$ independent of $\gamma$.
    
    To bound the exponential factor, we rearrange the $q_\gamma'$ equation \eqref{eqn:diff_q} to get
    \begin{align}\label{f_q}
        \mu_\gamma(s) = \frac{1 - q_\gamma'(s) + 2r_\gamma(s)}{2q_\gamma(s)}.
    \end{align}
    Integrating this from $t$ to $s$ yields:
    \begin{align*}
        \int_t^s \mu_\gamma(u)du = -\frac{1}{2}\log\Big(\frac{q_\gamma(s)}{q_\gamma(t)}\Big) + \int_t^s \frac{1 + 2r_\gamma(u)}{2q_\gamma(u)}du.
    \end{align*}
    Since $q_\gamma(u) \geq 1/2$ and $r_\gamma(u) \geq -\eta_r$, choosing $\eta_r \geq 1/2$ if necessary, we have that the integrand $(1+2r_\gamma)/(2q_\gamma)$ is bounded below by $1-2\eta_r$. Thus, \begin{align*}
        \int_t^s \mu_\gamma(u)du \geq -\frac{1}{2} \log(2 \eta_q) + (1-2\eta_r)(s-t).
    \end{align*}
    This lower bound guarantees that $\exp(-\int_t^s \mu_\gamma(u)du) \leq \eta'< \infty$ for some constant $\eta_{q,r}$ depending only on $\eta_q$ and $\eta_r$. Hence, $\|R_\gamma(s,t)\| \leq\eta_{q,r}\eta_H \eqqcolon \eta_R$. 

    Finally, we establish a uniform bound on $\mu_\gamma$. Using \eqref{eqn:diff_q} we obtain that \begin{align*}
        \mu_\gamma'(s) = [\gamma + f''(q_\gamma(s))](1 + 2r_\gamma(s) - 2q_\gamma(s)\mu_\gamma(s)).
    \end{align*}
    We increase $\gamma_0$ if necessary so that $\gamma_0 \geq 1 + \sup_{q\in [1/2, \eta_q]}|f''(q)|$, ensuring that the first term in the product above is positive. Since $q_\gamma \geq 1/2$ and $|r_\gamma| \leq \eta_r$ it follows that $\mu_\gamma'(s)< 0$ when $\mu_\gamma(s) = 2(1+\eta_r)$ and $\mu'_\gamma(s) > 0$ when $\mu_\gamma(s)= -2(1+\eta_r)$. Thus, $|\mu_\gamma(s)| \leq 2(1+\eta_r)$ since  we start at $\mu_\gamma(0) = 1/2$.
\end{proof}

\begin{proof}[Proof of Theorem \ref{thm2}]
    In Lemma \ref{lem:gamma_bounds} we showed that 
    \begin{align*}
    \{ C_\gamma(s,t), R_\gamma(s,t), q_\gamma(s), \mu_\gamma (s): s, t\in [0,T], \gamma \geq \gamma_0\}
    \end{align*}
    is uniformly bounded. Because $(C_\gamma, R_\gamma, q_\gamma)$ satisfy the integro-differential equations \eqref{l2.1}-\eqref{l2.3} and $\xi, f$ are polynomials, we have that 
    \begin{align*}
        \{ \partial_sC_\gamma(s,t), \partial_s R_\gamma(s,t), q'_\gamma(s): s, t\in [0,T], \gamma \geq \gamma_0\}
    \end{align*} is uniformly bounded. Further, from \eqref{l2.6}, \eqref{l2.8} are the uniform boundedness on $C, R$ established above, it follows that \begin{align*}
        \{\partial_t C_\gamma (s,t): s,t\in [0,T], \gamma \geq \gamma_0\}
    \end{align*} 
    is also uniformly bounded.
    
    Now we obtain a uniform bound on $\partial_tH_\gamma$. Differentiating the tree representation \eqref{expn:H} and expanding the recursion in Lemma \ref{lem_app}, we observe that each tree $\T\in\B$ with $|\T|=n\ge1$ contributes an integral over an ordered $2n$-dimensional simplex. The cancellation of the scalar factors in the tree operator $\A$ and the uniform bound on $C_\gamma$ imply that its integrand is bounded by $\eta^n$, for some $\eta<\infty$ independent of $\gamma\ge\gamma_0$. Since $t$ appears only as the lower integration limit, differentiation fixes the smallest integration variable at $t$, and hence 
    \begin{align*}
        \|\partial_t\omega_{\T,\gamma}(s,t)\| \leq \eta^n\frac{(s-t)^{2n-1}}{(2n-1)!}.
    \end{align*}
Thus, as in \eqref{eqn:bdd_h}, the differentiated series converges uniformly on $[0,T]$, justifying termwise differentiation and yielding the desired uniform bound on $\partial_t H_\gamma$. Finally, differentiating \eqref{eqn:exp_r} in $t$, using the boundedness of $H_\gamma, \partial_t H_\gamma, \mu_\gamma$, we have that 
\begin{align*}
        \{\partial_t R_\gamma (s,t): s,t\in [0,T], \gamma \geq \gamma_0\}
    \end{align*} 
    is uniformly bounded.

    We now establish the uniform boundedness of $\{\mu'_\gamma (s), q''_\gamma(s): s\in [0,T], \gamma \geq \gamma_0\}$. Taking a derivative of $r_\gamma$ and using the uniform boundedness of $C_\gamma, R_\gamma, \partial_s C_\gamma, \partial_s R_\gamma$, it is immediate that $\sup_{t\leq T}|r'_\gamma(t)|\leq \eta'_r$ is uniformly over $\gamma \geq \gamma_0$. Differentiating \eqref{eqn:diff_q} we have \begin{align}\label{eq:q_dd}
        q_\gamma''(s) = 2r_\gamma'(s) - 2[\mu_\gamma(s) + q_\gamma(s)(\gamma + f''(q_\gamma(s)))]q'_\gamma(s).
    \end{align}
    Since $q_\gamma'(0) = 1+2r_\gamma(0) - 2q_\gamma(0)\mu_\gamma(0) = 1-2f'(1) = 0$, we have by integrating the above display that \begin{align*}
        q'_\gamma(s)  = 2\int_0^s \exp\Big(-\int_u^s2[\mu_\gamma(v) + q_\gamma(v)(\gamma + f''(q_\gamma(v)))]dv\Big)r'_\gamma(u)du.
    \end{align*}
    Noting that \begin{align*}
        2[\mu_\gamma(s) + q_\gamma(s)(\gamma + f''(q_\gamma(s)))] \geq \gamma - 2\eta_\mu -2\eta_q \sup_{1/2\leq q\leq \eta_q}|f''(q)| \geq \frac{\gamma}{2}
    \end{align*}
    for sufficiently large $\gamma$, we have that \begin{align*}
        |q'_\gamma(s)|\leq 2\eta'_r\int_0^se^{-\gamma(s-u)/2}du \leq \frac{4\eta_r'}{\gamma}.
    \end{align*}
    Thus, from the above display and the identity 
\begin{align*}
        \mu'_\gamma(s) = [\gamma +f''(q_\gamma(s))]q'_\gamma(s),
    \end{align*}
    the uniform boundedness of $\{\mu'_\gamma (s): s\in [0,T], \gamma \geq \gamma_0\}$ follows. Further, from \eqref{eq:q_dd} it follows that $\{q''_\gamma (s): s\in [0,T], \gamma \geq \gamma_0\}$ is also uniformly bounded.
    
    The uniform bounds obtained above imply that
    \begin{align*}
        \{ C_\gamma(s,t), R_\gamma(s,t), q_\gamma(s), q'_\gamma(s), \mu_\gamma(s): s, t\in [0,T], \gamma \geq \gamma_0\}
    \end{align*}
    is an equicontinuous family. The Arzelà-Ascoli theorem guarantees the existence of a sequence $\gamma_k \to \infty$ and a limit $(C, R, q, \wt q, \mu)$ such that \begin{align*}
        (C_{\gamma_k}, R_{\gamma_k}, q_{\gamma_k}, q'_{\gamma_k}, \mu_{\gamma_k}) \to (C, R, q, \wt q, \mu)
    \end{align*} uniformly for $s, t\in [0,T]$.
    
    Since $\mu_{\gamma_k}(s) = \gamma_k(q_{\gamma_k}(s) - 1) + f'(q_{\gamma_k}(s))$ is uniformly bounded by $\eta_\mu$, dividing this identity by $\gamma_k$ forces $\sup_{s \in [0,T]} |q_{\gamma_k}(s) - 1| \to 0$ as $k \to \infty$, and thus $q\equiv 1$ uniformly. From the estimate $|q'_\gamma| \leq 4\eta_r'/\gamma$, it follows that $ q'_{\gamma_k} \to 0$ uniformly, so that $\wt q = q' \equiv 0$, and the spherical constraint $K(s) = C(s,s) = I_2$ holds in the limit. The uniform convergence of $(C_{\gamma_k}, R_{\gamma_k}, q_{\gamma_k}, q'_{\gamma_k}) \to (C,R, 1, 0)$ allows us to substitute $q_{\gamma_k}'(s) \to 0$ and $q_{\gamma_k}(s) \to 1$ into the trace equation \eqref{l2.2} to yield the explicit formula for $\mu(s)$ given in \eqref{hard_mu} (note that $R(s,u)$ and $\xi'(C(u,s))$ commute in $\V$, so their traces are equal).

    To pass the sub-sequential convergence to the full sequence, we must show that the limiting system admits a unique solution. We note that $\mu(s)$ is a locally Lipschitz functional of $C$ and $R$. Substituting \eqref{hard_mu} into the limits of \eqref{l2.1} and \eqref{l2.3} yields a closed system of Volterra integro-differential equations. The generalized Grönwall bound established in Proposition \ref{prop_uniq} applies, implying uniqueness of the limit.
\end{proof}

\section{Stationarity at high temperature}\label{sec6}

In this section, we work with even spin glasses and use $(C,R)$ to denote the solution to \eqref{hard_C}-\eqref{hard_R}, subject to the conditions \eqref{l2.4'}, on $\Pi\coloneqq \{(s,t): 0\leq t\leq s<\infty\}$. The passage from finite to infinite time follows from a compact-interval patching argument. More precisely, for each $T<\infty$, let $(C_T,R_T)$ denote the unique solution to \eqref{hard_C}-\eqref{hard_R} on $\{(s,t): 0\leq t\leq s\leq T\}$. If $0<T_1<T_2<\infty$, then the restriction of $(C_{T_2},R_{T_2})$ to $\{(s,t): 0\leq t\leq s\leq T_1\}$ satisfies the same equations and initial conditions as $(C_{T_1},R_{T_1})$. Uniqueness on $[0,T_1]$ therefore implies that this restriction coincides with $(C_{T_1},R_{T_1})$. Hence, the family $(C_T,R_T)_{T\geq 1}$ is compatible and defines a unique solution $(C,R)$ to \eqref{hard_C}-\eqref{hard_R} on $\R_+$.

\begin{lemma}[Exponential decay of $C$ and $R$]\label{exp_decay_cr}
There exist finite constants $c_0 > 1$, $c_1 > 0$ and $\beta_0 > \beta_1 > 0$ depending only on $\xi$ such that for all $\beta \leq \beta_1$ and all $s\geq t$, \begin{align*}
    \|R(s,t)\| \leq c_0 e^{-\delta_\beta (s-t)}\qquad  \text{and} \qquad \|C(s,t)\| \leq b_\beta e^{-\delta_\beta (s-t)/2}
\end{align*} where \begin{align*}
    \delta_\beta = \frac{1}{4}\Big(1-\frac{\beta}{\beta_0}\Big) \qquad \text{and} \qquad b_\beta = \Big(1 - \frac{256c_0c_1\beta^2}{3(1-(\beta/\beta_0)^2)}\Big)^{-1}.
\end{align*}
\end{lemma}

\begin{proof}
    We begin with the claim that $\sup_{s,t \geq 0}\|C(s,t)\| <\infty$ uniformly over $\beta>0$. Indeed, since $C(s,t) = c_d(s,t) I_2 + c_a(s,t) \Theta$ and $K(s) = q(s) I_2 = I_2$ are the limits of the positive semi-definite matrices $\E C_{N,\gamma}$ and $\E K_{N,\gamma}$ respectively, Cauchy-Schwarz implies that $|c_d(s,t)|^2 \leq q(s)q(t) = 1$, and the off-diagonal terms $c_a(s,t)$ are then bounded by $c_d(s,s)$ and $c_d(t,t)$ using \eqref{p2.61}. 

    We now establish the desired exponential decay of $R$. Let $H(s,t)$ be defined as in \eqref{defn_H}, but with $X(t) \coloneqq \int_0^t[-\mu(u)I_2 + \alpha \Theta] du$, where $\mu(s)$ is defined in \eqref{hard_mu}. From \eqref{hard_R} it follows that $H(s,t)$ satisfies \eqref{eqn:diff_h}, with the operator $\A$ defined as in \eqref{defn_a} (and $\xi''(\cdot)$ replaced by $\beta^2 \xi''(\cdot)$. Since $C(s,t)$ is uniformly bounded, it follows from the computation \eqref{eqn:bdd_h} that $\|H(s,t)\| \leq \eta e^{\eta'\beta(s-t)}$ for some constants $\eta, \eta' > 0$ (which are dependent only on $\xi$) and all $s \geq t$. To deduce exponential decay of $R$, we first obtain a lower bound on $\mu(s)$. Following the Hamiltonian estimate in the proof of Lemma \ref{lem:gamma_bounds}, we have \begin{align*}
        |r_\gamma(s)| \leq \beta \sup_N \E [\|J\|^N_\infty]\sum_{p=1}^n\frac{|a_{2p}|}{(2p-1)!}(2q_\gamma(s))^p,
    \end{align*}
    and since $q_\gamma(s) \to q(s) = 1$ as $\gamma \to 1$ for each fixed $s$, we obtain that for $\beta\leq1$, $\sup_{s\geq 0, \gamma \geq \gamma_0} |r_\gamma(s)| \leq \beta \eta_r$, where $\eta_r$ is a constant that is independent of $\beta$. Thus, taking $\gamma \to \infty$ in \eqref{f_q}, it follows that  for all $s\geq 0$ and $\beta \leq 1/(4\eta_r)$, we have $\mu(s) \geq (1-2\beta\eta_r) /2 \geq 1/4$, and hence \begin{align}\label{eq::1}
        \|e^{X(s) -X(t)}\| = e^{-\int_t^s \mu(u)du} \leq e^{-(s-t)(1-2\beta \eta_r)/2} \leq e^{-(s-t)/4}.
    \end{align}
    We now set $\beta_0 = 1/(4(\eta_r\vee \eta'))$. Then, the above bounds, together with the fact that $\|e^{\alpha \Theta}\| = 1$ yields that \begin{align*}
        \|R(s,t)\| \leq e^{-\int_t^s \mu(u)du}\|H(s,t)\|\leq \eta e^{-(s-t)/4 + \eta' \beta(s-t)}\leq \eta e^{-\delta_\beta(s-t)}
    \end{align*}
    as desired, with $c_0 = \eta\vee 2 > 1$.
    
    We now turn to prove the exponential decay for $C$. Let $X(t)$ be as in the previous paragraph. Integrating \eqref{hard_C} we have for $s\geq t\geq 0$ that \begin{align*}
        C(s,t) & = e^{X(s) - X(t)} + \beta^2 \int_t^s e^{X(s) - X(u)}\Big[\int_0^t \xi'(C(u,v))R(t,v)^T dv\Big]du \\
        & \qquad  + \beta^2 \int_t^s e^{X(s) - X(u)}\Big[\int_0^u \Big(\xi''(C(u,v)) \odot R(u,v)\Big) C(v,t)dv\Big]du.
    \end{align*}
    Since $\xi'$ is an odd function, the uniform bound on $\|C(s,t)\|$ yields that for some constant $c_1>0$ depending only on $\xi, $
   \begin{align*}
       \|C(s,t)\|, \|\xi''(C(s,t))\| \leq c_1 \qquad \text{and}\qquad \|\xi'(C(s,t))\|\leq c_1\|C(s,t)\|.
   \end{align*} Further, taking $\beta < \beta_0$ so that \eqref{eq::1} holds, and using the exponential decay of $R$ from above, we have that 
    \begin{align*}
        \|C(s,t)\| & \leq e^{-(s-t)/4} + c_0c_1\beta^2 \int_t^s e^{-(s-u)/4}\int_0^t \|C(u,v)\|e^{-\delta_\beta(t-v)}dvdu \\
        & \qquad + c_0c_1 \beta^2 \int_t^s e^{-(s-u)/4}\int_0^u\|C(v,t)\|e^{-\delta_\beta(u-v)}dvdu,
    \end{align*}
    Multiplying both sides by $e^{\delta_\beta(s-t)/2}$ we rewrite the above as \begin{align}\label{exp_lem_1}
        e^{\delta_\beta(s-t)/2}\|C(s,t)\| & \leq e^{-(1/2 - \delta_\beta)(s-t)/2} + c_0c_1\beta^2 e^{\delta_\beta(s-t)/2}\int_t^s e^{-(s-u)/4}\int_0^t \|C(u,v)\|e^{-\delta_\beta(t-v)}dvdu \notag \\
        & \qquad + c_0c_1\beta^2 e^{\delta_\beta(s-t)/2}\int_t^s e^{-(s-u)/4}\int_0^t \|C(v,t)\| e^{-\delta_\beta(u-v)}dvdu\notag \\
        & \qquad +c_0c_1\beta^2 e^{\delta_\beta(s-t)/2} \int_t^s e^{-(s-u)/4}\int_t^u \|C(v,t)\| e^{-\delta_\beta(u-v)}dvdu
    \end{align}
    Set \begin{align*}
    \epsilon(s,t) & \coloneqq \sup\{e^{\delta_\beta(u-v)/2}\|C(u,v)\|: u\in [0,s], v\in [0,t],  v\leq u\}\\
    & = \sup\{e^{\delta_\beta(u-v)/2}\|C(v,u)\|: u\in [0,s], v\in [0,t],  v\leq u\}.
    \end{align*}
    Thus, for any $0\leq v\leq t\leq u\leq s$, we have that $\|C(u,v)\| \leq \epsilon (s,t) e^{-\delta_\beta(u-v)/2}$. Noting that $\delta_\beta < 1/4$ for $\beta < \beta_0$, from \eqref{exp_lem_1} we get \begin{align*}
         e^{\delta_\beta(s-t)/2}\|C(s,t)\| & \leq 1  + c_0c_1\beta^2 \epsilon(s,t)\int_t^s e^{-(1/2-\delta_\beta)(s-u)/2}\int_0^t e^{-3\delta_\beta(t-v)/2}dv \\
        & \qquad + c_0c_1\beta^2 \epsilon(s,t)\int_t^s e^{-(1/2-\delta_\beta)(s-u)/2 - \delta_\beta(u-t)/2}du\int_0^te^{-3\delta_\beta(t-v)/2}dv\\
        & \qquad + c_0c_1\beta^2 \epsilon(s,t) \int_t^s e^{-(1/2 - \delta_\beta)(s-u)/2}\int_t^u e^{-\delta_\beta(u-v)/2}dvdu\\
        & \leq 1 + c_0c_1 \beta^2 \epsilon(s,t) \int_t^s e^{-(1/2 - \delta_\beta)(s-u)/2}\Big[\frac{2}{3\delta_\beta}\Big(1+e^{-\delta_\beta(u-t)/2}\Big) + \frac{2}{\delta_\beta} \Big(1-e^{-\delta_\beta(u-t)/2}\Big)\Big]du\\
        & \leq 1 + \frac{8c_0c_1\beta^2}{3\delta_\beta}\varepsilon(s,t) \int_t^s e^{-(1/2-\delta_\beta)(s-u)/2}du \leq 1+ \frac{16c_0c_1\beta^2}{3\delta_\beta(1/2-\delta_\beta)}
\varepsilon(s,t)
    \end{align*}
    Taking supremum over the domain on which $\varepsilon(s,t)$ is defined and plugging the expression for $\delta_\beta$ yields \begin{align*}
        \varepsilon(s,t) \leq 1 + m_\beta \varepsilon(s,t)
    \end{align*} where \begin{align*}
        m_\beta \coloneqq \frac{256 c_0c_1\beta^2}{3(1-(\beta/\beta_0)^2)}.
    \end{align*}
    Thus, for \begin{align*}
        \beta < \beta_1 \coloneqq \frac{\beta_0}{\sqrt{3+ 256 c_0c_1\beta_0^2}} < \beta_0, 
    \end{align*}
    we have $m_\beta <1$, and hence $\|C(s,t)\| \leq b_\beta e^{-\delta_\beta(s-t)/2}$, with $b_\beta = 1/(1-m_\beta)$.
    \end{proof}

    \begin{remark}
    In the proof of \cite[Lemma 3.2]{Dem}, the quantity $\Delta(t,s):=\sup_{0\leq u\leq t,\;0\leq v\leq s}C(u,v)$ is identically $1$, since $0\leq C\leq 1$ and $C(0,0)=1$. Thus the asserted decay estimate for $\Delta$ cannot hold, leaving a gap in the subsequent Gr\"onwall argument. Our weighted supremum avoids this issue.
\end{remark}

For $c >1$ and $0<\lambda<1/4$, let $\S(c,\lambda)$ denote the set of bounded functions $(C,R)\in \cC([0,\infty)^2)^{2\times 2}\times\cC(\Pi)^{2\times 2}$ in $\V$ satisfying \eqref{l2.4'} and
\begin{align}\label{exp_decay}
\|C(s,t)\|\leq c e^{-\lambda(s-t)/2},
\qquad
\|R(s,t)\|\leq c e^{-\lambda(s-t)},
\end{align}
for all $0\leq t\leq s$. We equip $\S(c,\lambda)$ with the norm \begin{align*}
    \|(C,R)\|_\infty\coloneqq \sup_{s,t\geq 0} \|C(s,t)\| + \sup_{s\geq t}\|R(s,t)\|, 
\end{align*}
which makes $\S(c,\lambda)$ is a closed subset of the Banach space $(\cC([0,\infty)^2)^{2\times 2}\times\cC(\Pi)^{2\times 2}, \|\cdot\|_\infty)$, and hence complete.

For convenience, we record here some estimates that will be used frequently in the upcoming proofs. Note that since $\xi$ is an even polynomial, $\xi', \xi''$ are uniformly Lipschitz on the bounded interval $[-c, c]$. Hence, for $(C_i, R_i)_{i=1,2} \in \S(c,\lambda)$ it follows from \eqref{exp_decay} that there exists some constant $\eta = \eta(\xi,c)$ such that\begin{align} \label{nu_bds}
    & \|C_i\| + \|R_i\| + \|\xi'(C_i)\| + \|\xi''(C_i)\| \leq \eta \qquad \text{and}\qquad  \|\xi'(C_1) - \xi'(C_2)\| + \|\xi''(C_1) - \xi''(C_2)\| \leq \eta \|C_1 - C_2\|.
\end{align}

Consider the map $\Xi$ defined on $\S(c,\lambda)$ which maps $(C, R)$ to $(\wt C, \wt R)$, where $(\wt C, \wt R)$ solves the following set of equations for $s\geq t$:\begin{align}
        \partial_s \wt C(s,t) & = (- \mu_{C,R}(s) + \alpha \Theta)\wt C(s,t)  + \beta^2\int_0^t\xi'(C(s,u))R(t,u)^Tdu \notag \\
        & \qquad + \beta^2\int_0^s\Big(\xi''(C(s,u))\odot R(s,u)\Big)C(u,t)du \label{hard_c}\\
    \partial_s \wt R(s,t) & =  (- \mu_{C,R}(s)+ \alpha \Theta )\wt R(s,t) +  \beta^2\int_t^s\Big(\xi''(C(s,u))\odot \wt R(s,u)\Big)\wt R(u,t) du,\label{hard_r}
    \end{align}
subject to the conditions \eqref{l2.4'}, where $\mu_{C,R}$ is  defined as the right-hand side of \eqref{hard_mu}. To see that $\Xi$ is well defined we need to show that the above system admits a unique solution. To that end, fix $(C,R)$ and set
\begin{align}\label{conj}
X_{C,R}(t) \coloneqq \int_0^t\Big(-\mu_{C,R}(u)I_2+\alpha\Theta\Big)du \qquad \text{and}\qquad 
\wt H(s,t) \coloneqq e^{-[X_{C,R}(s)-X_{C,R}(t)]}\wt R(s,t).
\end{align}
From \eqref{hard_r} we obtain that $\wt H$ satisfies \eqref{eqn:int_H}. The absolutely convergent binary-tree representation from Lemma \ref{lem:H} therefore gives the existence and uniqueness of $\wt H$, and hence of $\wt R$. Once $C$ and $R$ are fixed, equation \eqref{hard_c} is a linear nonhomogeneous equation for $\wt C$, so its existence and uniqueness are immediate.

Let $c_0, c_1, \beta_0, \beta_1, \delta_\beta$ be as in Lemma \ref{exp_decay_cr}, and for brevity, write $\delta_1 \coloneqq \delta_{\beta_1}$. We show that the map $\Xi$ admits a fixed point in the space $\cR(c_0, \delta_1)$, which the set of all pairs $(C,R)\in \S(c_0, \delta_1)$ such that for every $\tau \geq 0$, the stationary limits of $C$ and $R$, defined as\begin{align*}
    C_*(\tau) \coloneqq \lim_{t\to \infty}  C(t+\tau , t) \qquad \text{and} \qquad R_*(\tau)\coloneqq \lim_{t\to \infty} R(t+\tau, t),
\end{align*}
exist as continuous functions on $[0,\infty)$. Note that $\cR(c_0, \delta_1)$ is non-empty since the pair $(C_0,R_0)= (e^{-\delta_1(s-t)/2}I_2, e^{-\delta_1(s-t)}I_2)$ belongs to this set.

\begin{prop}\label{contraction}
Let $c_0, c_1, \beta_0, \beta_1, \delta_\beta$ be as in Lemma \ref{exp_decay_cr}. Then there exists $\beta_
2 = \beta_2(\xi) > 0$ such that for all $\beta \leq \beta_2$, we have $\Xi(\S(c_0, \delta_1)) \subseteq \S(c_0, \delta_1)$ and $\Xi(\cR(c_0, \delta_1)) \subseteq \cR(c_0, \delta_1)$.
\end{prop}

\begin{proof}
For brevity, we shall write $\S = \S(c_0, \delta_1)$ and $\cR = \cR(c_0, \delta_1)$. Throughout the proof, we shall denote by $\eta,\eta' > 0$ constants that are dependent only on $\xi$, and whose values may change at every occurrence.

We start by recording the following lower bound on $\mu_{C,R}$ that holds uniformly for all $(C,R)\in\S$, thanks to \eqref{exp_decay}: for all $s\geq 0$,
\begin{align}\label{lbd_mu}
\mu_{C,R}(s) &\geq \frac{1}{2}-\eta(c_0,\xi)^2\beta^2 \int_0^s e^{-3\delta_1(s-u)/2}du \geq\frac{1}{4},
\end{align}
for some constant $\eta(c_0,\xi)$, provided that $\beta$ is sufficiently small. Henceforth, we will assume that $\beta$ is small enough such that \eqref{lbd_mu} holds. 

We first show that $\Xi$ maps $\S$ into itself. To that end, we integrate \eqref{hard_c} over $s\geq t$ and use $\wt C(t,t) = I_2$ to get \begin{align}\label{int_wt_c}
    \wt C(s,t) & = e^{X_{C,R}(s)-X_{C,R}(t)}
        +\beta^2\int_t^s e^{X_{C,R}(s)-X_{C,R}(v)} \Big[\int_0^t \xi'(C(v,u))R(t,u)^Tdu \notag \\
        & \qquad +\int_0^v \Big(\xi''(C(v,u))\odot R(v,u)\Big)C(u,t) du\Big]dv.
\end{align}
With the previous estimate on $\mu_{C,R}$ and using $\delta_1 < 1/4$, we get that \begin{align*}
    \|\wt C(s,t)\| & \leq e^{-(s-t)/4} + \eta\beta^2 \int_t^s e^{-(s-v)/4} \Big[\int_0^t e^{-\delta_1(v-u)/2 -\delta_1(t-u)}du + \int_0^v e^{-\delta_1|t-u|/2 - \delta_1(v-u)} du\Big]dv\\
    & \leq e^{-(s-t)/4} + \frac{8\eta\beta^2}{3\delta_1} \int_t^s e^{-(s-v)/4 - \delta_1(v-t)/2}dv\\
    & \leq \Big(1\vee \frac{32 \eta \beta^2}{3\delta_1(1-2\delta_1)}\Big) e^{-\delta_1(s-t)/2} \leq c_0 e^{-\delta_1(s-t)/2},
\end{align*}
where the final inequality is achieved by reducing $\beta$ if necessary. On the other hand, conjugating $\wt R$ with $X_{C,R}$ as in \eqref{conj}, the bound \eqref{lbd_mu} and the computation in the binary-tree expansion lemma \ref{lem:H} yields
\begin{align*}
    \|\wt R(s,t)\| &\leq e^{-(s-t)/4} \sum_{n\geq0} \frac{1}{n+1} \frac{\eta^n\beta^{2n}(s-t)^{2n}}{(n!)^2}\leq e^{-(1/4-\eta'\beta)(s-t)} \leq c_0e^{-\delta_1(s-t)},
\end{align*}
where the final bound is obtained by choosing $\beta$ small enough if necessary.
 This establishes that $\Xi(\S) \subseteq \S$.

We now show that $\Xi(\cR) \subseteq \cR$. Let $(C, R)\in \cR$ and $(C_*, R_*)$ denote the stationary limits. We want to show that $(\wt C, \wt R) = \Xi(C,R) \in \cR$. Recall from \eqref{conj} that for $t, \tau \geq 0$,  \begin{align}\label{undo_conj}
    \wt R(t+ \tau, t) = \exp\Big(-\int_t^{t + \tau} \mu_{C,R}(u)du+\alpha \tau \Theta\Big) \wt H(t+\tau, t).
\end{align}
We start by identifying the stationary limit of $\wt H$ defined in \eqref{conj}. For $t \geq 0, u\geq v \geq 0$, consider the operators \begin{align*}
    \A_t[M(u,v)]
& = \beta^2 e^{-\alpha(u-v)\Theta} \Big(\xi''(C(t+u, t+v))\odot [e^{\alpha(u-v)\Theta}M(u,v)]\Big),\\
    \A_*[M(u,v)] & = \beta^2 e^{-\alpha(u-v)\Theta} \Big(\xi''(C_*(u-v))\odot[e^{\alpha(u-v)\Theta}M(u,v)]\Big).
\end{align*} 
Recall that $\B$ denotes the set of finite rooted binary trees. Based on $\A_t$ and $\A_*$, we construct a family of weight matrices as follows: for the trivial tree $\emptyset$ consisting of only one vertex, $\omega_{t,\emptyset} = \omega_{*,\emptyset} \equiv I_2$, and $(\omega_{t,\T})_{\T \in \B}, (\omega_{*,\T})_{\T \in \B}$ are defined recursively from $\A_t$ and $\A_*$ using \eqref{wt_const} respectively. It follows from Lemma \ref{lem:H} that \begin{align*}
    \wt H(t+u, t+v) = \sum_{\T \in \B} \omega_{t,\T}(u,v).
\end{align*}
Following the inductive construction of $\omega_{*,\T}$, it is easy to see that $\omega_{*,\T}(u,v)$ is a function of $u-v$. We claim that \begin{align*}
    \wt H_*(t)\coloneqq \sum_{\T\in \B} \omega_{*,\T}(t, 0)
\end{align*}
is the stationary limit of $\wt H$. Indeed, since $C$, and hence $C_*$ is uniformly bounded, it follows by the computation \eqref{eqn:bdd_h} that the series in the above display is absolutely convergent. Further, with \begin{align*}
    \zeta(u,v) = \zeta_t(u,v)\coloneqq \eta\beta^2 \| C(t+u, t+v) - C_*(u-v)\|,
\end{align*} \eqref{diff_a} holds for $\A = \A_t$ and $\wt A = \A_*$, so by the computation in Lemma \ref{lem:H} it follows that \begin{align*}
    \|\wt H(t+u, t + v) - \wt H_*(u-v)\| \leq \eta e^{\eta'\beta(u-v)}\int_v^u\int_y^u \zeta_t(x,y) dxdy.
\end{align*}
Since $C_*$ is the stationary limit of $C$, we have that $\zeta_t(u,v) \to 0$ as $t\to \infty$ for each $u \geq v \geq 0$. Also, by the boundedness of $C$, we have that $\zeta_t$ is uniformly bounded and therefore, by the dominated convergence theorem, we have that for every fixed $u \geq v\geq 0$, \begin{align}\label{tti_h}
    \wt H_* (u-v) = \lim_{t\to \infty}\wt H(t+u, t+v).
\end{align}

We now turn to computing the limit of $\mu_{C,R}(t)$ as $t\to \infty$. Setting $v = s - u$, we have that for each fixed $v\geq 0$, the integrand in \eqref{hard_mu} converges pointwise to \begin{align*}
      \tr\Big[\xi'(C_*(v))R_*(v)^T + \Big(\xi''(C_*(v))\odot R_*(v)\Big)C_*(v)^T\Big]
 \end{align*}
 as $u \to \infty$. Using the exponential decay \eqref{exp_decay} and the bounds \eqref{nu_bds}, it follows that the  above display is bounded by $\eta e^{-3\delta_1v/2}$, so by the dominated convergence theorem we have that $\lim_{t\to \infty} \mu_{C,R}(t) = \mu_*$, where $\mu_*$ is defined in \eqref{mu*}.
 
The boundedness of $C,R$ implies the boundedness of $\mu_{C,R}$, and hence of $\mu_*$, and thus another application of the dominated convergence theorem implies that for each $u \geq v \geq 0$ fixed, \begin{align*}
    \exp\big(-\int_{t+v}^{t+u}\mu_{C,R}(s)ds\big) \to e^{-\mu_*(u-v)}.
\end{align*}
Setting $v= 0$ and $u = \tau$ in \eqref{tti_h} and the above display, it now follows from \eqref{undo_conj} that for every $\tau \geq 0$ fixed, the limit \begin{align}\label{tti_exp_mu}
    \lim_{t \to \infty} \wt R(t+\tau, t) = e^{-(\mu_* - \alpha\Theta)\tau} \wt H_*(\tau) \eqqcolon \wt R_*(\tau)
\end{align}
exists. We now show the same for $\wt C$. For $t, \tau \geq 0$, we have from \eqref{int_wt_c} that\begin{align*}
    \wt C(t+\tau, t) & = \exp\Big(-\int_t^{t+\tau}\mu_{C,R}(u)du+\alpha\tau \Theta\Big) + \beta^2\int_0^\tau \exp\Big(-\int_{t+u}^{t+\tau} \mu_{C,R}(v)dv + \alpha(\tau - u)\Theta\Big) \\
    & \qquad \cdot\Big[\int_{-t}^0\xi'(C(t+u, t+v))R(t,t+v)^T dv\\
    & \qquad \qquad + \int_{-t}^u\Big(\xi''(C(t+u,t+v))\odot R(t+u,t+v)\Big)C(t+v,t)dv\Big]du.
\end{align*}
As we did while establishing the limit of $\mu_{C,R}(t)$, we can use dominated convergence theorem to justify that the inner integrals in the above display converge to \begin{align*}
    \int_u^\infty \xi'(C_*(v))R_*(v-u)^Tdv \qquad \text{and} \qquad \int_0^\infty \Big(\xi''(C_*(v))\odot R_*(v)\Big)C_*(u-v)dv
\end{align*}
respectively. Hence, from \eqref{tti_exp_mu} we have that for every fixed $\tau \geq 0$, we have that the limit \begin{align*}
    \wt C_*(\tau) & \coloneqq \lim_{t \to \infty} \wt C(t+\tau, t)\\
    &   = e^{-(\mu_{C_*,R_*}- \alpha\Theta)\tau} + \beta^2 e^{-(\mu_{C_*,R_*}- \alpha\Theta)\tau}\int_0^\tau e^{(\mu_{C_*,R_*}-\alpha \Theta)u} \Big[\int_u^\infty \xi'(C_*(v))R_*(v-u)^Tdv \\
    & \qquad + \int_0^\infty \Big(\xi''(C_*(v))\odot R_*(v)\Big)C_*(u-v)dv\Big]du
\end{align*}
exists. \eqref{hard_c}-\eqref{hard_r} and the fact that $\Xi(\S) \subset \S$ implies that \begin{align*}
    \sup_{s\geq t} \Big(\|\partial_s \wt C(s,t)\| + \|\partial_s\wt R(s,t)\|\Big) \leq \eta \Big(1 + \alpha + \frac{\beta^2}{\delta_1}\Big)
\end{align*}
for some $\eta = \eta(\xi, c_0)$, and hence $(\wt C_*, \wt R_*)$ are Lipschitz, and in particular, continuous. Thus, we have established that $(\wt C, \wt R) \in \cR$.
\end{proof}

\begin{prop}\label{prop_ctr}
    Let $c_0, c_1, \beta_0, \beta_1, \delta_\beta$ be as in Lemma \ref{exp_decay_cr} and $\beta_2$ be as in Proposition \ref{contraction}. Then, there exists a constant $\beta_* > 0$ such that for all $\beta \leq \beta_*$, $\Xi$ is a contraction on $\S(c_0, \delta_1)$ with respect to the norm $\|\cdot\|_\infty$. Further, the unique solution of \eqref{hard_C}-\eqref{hard_R} is the unique fixed point of $\Xi$ in $\cR(c_0, \delta_1)$. 
\end{prop}

\begin{proof}
We first show that for $\beta$ sufficiently small, the unique solution $(C,R)$ of the system \eqref{hard_C}-\eqref{hard_R} is in $\S(c_0,\delta_1)$. Since $(C,R)$ already satisfies \eqref{l2.4'}, it suffices to show that for $\beta$ sufficiently small, $(C,R)$ satisfies the desired exponential decay rates in \eqref{exp_decay}. Indeed, for \begin{align*}
    \beta \leq \beta_1 \wedge \beta_2 \wedge \Big( \frac{1}{\beta_0^2} + \frac{256 c_0^2 c_1}{3(c_0-1)}\Big)^{-1/2},
\end{align*}
we have that $b_\beta \leq c_0$, and thus $(C,R) \in \S(c_0, \delta_1)$. For the rest of the proof we will use $\eta > 0$ to denote constants which are dependent only on $\xi$. The value of this constant may change at each occurrence. We start with the following estimate. For $(C_1, R_1), (C_2, R_2) \in \S(c_0, \delta_1)$, we claim that \begin{align}\label{est_mu}
    \sup_{s\geq 0}|\mu_{C_1,R_1}(s) - \mu_{C_2,R_2}(s)| \leq \eta \beta^2 \|(C_1, R_1) - (C_2,R_2)\|_\infty.
\end{align}
To see this, we use the expression of $\mu_{C,R}$ in \eqref{hard_mu} to write the difference $\mu_{C_1,R_1}(s) - \mu_{C_2,R_2}(s)$ as a telescoping sum. Using \eqref{nu_bds} as necessary, it follows each term in the telescoping sum contains a factor of the form  $C_1(u,v) - C_2(u,v)$ or $R_1(u,v) - R_2(u,v)$, which can be bounded by $\|(C_1, R_1) - (C_2,R_2)\|_\infty$, and another factor containing $C_i(u,v)$ or $R_i(u,v)$, which decay exponentially in $u-v$. This yields \begin{align*}
    |\mu_{C_1,R_1}(s)-\mu_{C_2,R_2}(s)| &\leq \eta\beta^2\|(C_1, R_1) - (C_2,R_2)\|_\infty \int_0^s\Big(e^{-\delta_1(s-u)/2} +e^{-\delta_1(s-u)} +e^{-3\delta_1(s-u)/2} \Big)du\\
    &\leq \eta\beta^2\|(C_1, R_1) - (C_2,R_2)\|_\infty.
\end{align*}
For simplicity, denote $X_i(s,t)\coloneqq X_{C_i,R_i}(s) - X_{C_i, R_i}(t)$ for $i =1,2$. The uniform bound \eqref{lbd_mu}, together with the inequality $|e^{-x} - e^{-y}| \leq e^{-(x\wedge y)}|x-y|$ for $x, y \geq 0$ yields \begin{align}\label{est_x}
    \|e^{X_1(s,t)} - e^{X_2(s,t)}\| \leq \eta \beta^2 (s-t)e^{-(s-t)/4} \|(C_1, R_1) - (C_2,R_2)\|_\infty,
\end{align} 
where the integral in the definition of $X_{C,R}$ contributes to the linear term $s-t$ and the factor $e^{-(s-t)/4}$ arises from the bound \eqref{lbd_mu}.

Equipped with the above estimates, we show that for all $\beta$ sufficiently small \begin{align}
    \sup_{s\geq t}\|\wt C_1(s,t) - \wt C_2(s,t)\| \leq \frac{1}{3} \|(C_1, R_1) - (C_2,R_2)\|_\infty \label{est_c},\\
    \sup_{s\geq t}\|\wt R_1(s,t) - \wt R_2(s,t)\| \leq \frac{1}{3} \|(C_1, R_1) - (C_2,R_2)\|_\infty \label{est_r},
\end{align}
and hence $\|\Xi(C_1,R_1) - \Xi(C_2,R_2)\|_\infty < \|(C_1, R_1) - (C_2,R_2)\|_\infty$, establishing that $\Xi$ is a contraction. 

We start with the proof of \eqref{est_c}. Using \eqref{int_wt_c}, we write the difference $\wt C_1(s,t) - \wt C_2(s,t)$ as a telescopic sum. The differences of the non-linear terms appearing in the inner integral are controlled by \eqref{nu_bds} as in the derivation of \eqref{est_mu}, and the exponential term $e^{X_i(s,t)}$ can be upper bounded by $e^{-(s-t)/4}$ (see \eqref{eq::1} for details). Together with the estimate \eqref{est_x}, we arrive at \begin{align*}
    \|\wt C_1(s,t) - \wt C_2(s,t)\| & \leq \eta \beta^2  \|(C_1,R_1) -(C_2,R_2)\|_\infty \Big[(s-t) e^{-(s-t)/4} \\
    & \qquad + \int_t^s\Big(\beta^2(s-v) e^{-(s-v)/4 - \delta_1(v-t)/2} + e^{-(s-v)/4}\Big)dv\Big].
\end{align*}
The terms on the right-hand side above can be uniformly bounded for $s\geq t$ by reducing the $\beta$ as necessary, thereby establishing \eqref{est_c}. To prove \eqref{est_r}, we use the conjugation \eqref{conj} and the bounds \eqref{eqn:bdd_h}, \eqref{lbd_mu} and \eqref{est_x} to obtain \begin{align}\|\wt R_1(s,t) - \wt R_2(s,t)\| & \leq \|\wt H_1(s,t)\| \|e^{X_1(s,t)} - e^{X_2(s,t)}\| + \|e^{X_2(s,t)}\|\|\wt H_1(s,t) -\wt H_2(s,t)\|\notag\\
    & \leq \eta \beta^2 (s-t)e^{-(1/4 - \eta\beta)(s-t)} \|(C_1, R_1) - (C_2,R_2)\|_\infty + e^{-(s-t)/4}\|\wt H_1(s,t) -\wt H_2(s,t)\|, \label{est_int_r}
\end{align}
and so it suffices to obtain a bound on $\|\wt H_1(s,t) -\wt H_2(s,t)\|$. Let $\A_i$ denote the operator \eqref{defn_a} corresponding to $C_i$. From \eqref{nu_bds} we have \begin{align*}
    \|\A_i[M(u,v)]\| \leq \eta \|M(u, v)\|, \qquad \|\A_1 [M(u,v)] - \A_2[M(u,v)]\| \leq \eta \|(C_1, R_1) - (C_2,R_2)\|_\infty\|M(u,v)\|.
\end{align*}
Taking $\zeta \equiv \eta \|(C_1,R_1) - (C_2,R_2)\|_\infty$ we apply \eqref{diff_w} term by term to the binary-tree representations of $\wt H_1$ and $\wt H_2$. Summing the resulting estimates exactly as in the proof of Lemma~\ref{lem:H}, we obtain \begin{align*}
    \|\wt H_1(s,t) -\wt H_2(s,t)\| \leq \eta \beta^2(s-t)^2 e^{\eta\beta(s-t)}\|(C_1, R_1) - (C_2,R_2)\|_\infty,
\end{align*}
where the factor $(s-t)^2$ upper bounds the double integral appearing in \eqref{diff_w}. Thus, from \eqref{est_int_r} we get \begin{align*}
    \|\wt R_1(s,t) - \wt R_2(s,t)\| \leq \eta \beta^2 (s-t + (s-t)^2) e^{-(1/4 -\eta\beta)(s-t)}\|(C_1, R_1) - (C_2,R_2)\|_\infty
\end{align*}
and so by choosing $\beta$ sufficiently small we obtain a uniform upper bound of the right-hand side above, and hence establish \eqref{est_r}.

Note that $\cR(c_0,\delta_1)$ is closed with respect to the metric induced by the uniform norm $\|\cdot\|_\infty$, since a uniform limit of functions whose stationary limits are Cauchy as $t\to\infty$ again has a Cauchy stationary limit. Hence, $\cR(c_0,\delta_1)$ is complete as a closed subspace of $\S(c_0,\delta_1)$. Since $\Xi$ is a contraction on $\S(c_0,\delta_1)$, and therefore also on $\cR(c_0,\delta_1)$, its unique fixed point belongs to $\cR(c_0,\delta_1)$. By construction, the unique solution $(C,R)$ of \eqref{hard_C}--\eqref{hard_R} is a fixed point of $\Xi$. Moreover, for sufficiently small $\beta$, we have $(C,R)\in\S(c_0,\delta_1)$. The uniqueness of the fixed point of $\Xi$ therefore implies that $(C,R)\in\cR(c_0,\delta_1)$.
\end{proof}

\begin{proof}[Proof of Theorem \ref{thm3}]
    Let $c_0, \beta_*, \beta_1, \delta_\beta$ be as above. We consider the set $\cR_*(c_0, \delta_1)$ defined as set of pairs $(C,  R) \in (\cC(\R_+)^{2\times 2})^2$ such that $(C,  R)$ satisfies \eqref{l2.4''} and $\| R(t)\|\leq c_0 e^{-\delta_1 t}$ and $\| C(t)\| \leq c_0 e^{-\delta_1 t/2}$ for all $t \geq 0$. In the previous proposition we showed that the  solution $(C,R)$ of \eqref{hard_C}-\eqref{hard_R} is in $\cR(c_0, \delta_1)$. Thus, their stationary limits $(C_*, R_*) \in \cR_*(c_0, \delta_1)$. 

    The map $\Xi$ induces the map $\Xi_*$ on $\cR_*(c_0, \delta_1)$ which sends $(C,R)$ to $(\wt C, \wt R)$, where $(\wt C, \wt R)$ solves the following system: for $t\geq 0$, \begin{align*}
        \wt C'(t) & = (-\mu + \alpha \Theta)\wt C(t) + \beta^2 \int_0^\infty \Big(\xi''(C(u))\odot R(u)\Big)C(t-u)du + \beta^2 \int_t^\infty \xi'(C(u))R(u-t)^T du,\\
        \wt R'(t) & = -(\mu + \alpha \Theta) \wt R(t) +  \beta^2 \int_0^t \Big(\xi''(C(u))\odot \wt R(u)\Big)\wt R(t-u)du,
    \end{align*}
    with \begin{align*}
        \mu = \frac{1}{2} + \frac{\beta^2}{2} \int_0^\infty \tr\Big[\xi'(C(u))R(u)^T + \Big(\xi''(C(u))\odot R(u)\Big)C(u)^T\Big]du,
    \end{align*}
    and $(\wt C, \wt R)$ satisfying \eqref{l2.4''}. Using arguments similar to those presented for $\Xi$, we can show that the map $\Xi_*$ is well defined, is a contraction on $\cR_*(c_0, \delta_1)$ with respect to the metric induced by the supremum norm, and hence admits a unique fixed point in $\cR_*(c_0, \delta_1)$. By design, the stationary limit $(C_*,R_*)$ is a fixed point of $\Xi_*$, and by uniqueness of the fixed point, it follows that $(C_*, R_*)$ is the unique solution to \eqref{C*}-\eqref{R*}. The exponential decay of $(C_*,R_*)$ follows immediately form the definition of $\cR_*(c_0,\delta_1)$.
\end{proof}

\section{The spherical Sherrington-Kirkpatrick model}\label{sec7}

In this section, we restrict our attention to the spherical SK model where $\xi(r)=r^2/2$, and derive explicit formulas for both the coupled and decoupled systems which are valid at all temperatures. Our approach exploits the quadratic structure of the Hamiltonian, which allows us to pass to a co-rotating frame and thereby reduce the coupled, non-reciprocal system to a decoupled one. The analysis is thereby reduced to a scalar system for which explicit formulas are obtained. These formulas allow us to sharpen the results of \cite[Proposition 3.2]{Ben2}, where the critical temperature $\beta=1/2$ marking the transition between a stationary regime and an aging regime was established, and the behavior of the correlation function analyzed for all temperatures. We then exploit the correspondence between the coupled and decoupled systems induced by the rotation to transfer this characterization back to the original model. In this way, we rigorously establish the corresponding physical behavior in the non-reciprocal system and verify the predictions of \cite{Lor}, for example, the aging relation stated in their equation (7).
    
    We write down the system \cite[(1.14)-(1.15)]{Dem} corresponding to the quadratic case $\xi(x) = x^2/2$: for all $s\geq t$,
    \begin{align}
        \partial_sc(s,t) & = - \mu(s) c(s,t)  + \beta^2\int_0^tc(s,u)r(t,u)du + \beta^2\int_0^sr(s,u)c(u,t)du,\label{c_sc}\\
    \partial_s r(s,t) & =  - \mu(s)r(s,t) +  \beta^2\int_t^sr(s,u)r(u,t) du, \label{r_sc}
    \end{align}
    where \begin{align*}
        \mu(s) = \frac{1}{2} + 2\beta^2\int_0^s c(s,u)r(s,u)du,
    \end{align*}
    with \begin{align*}
    \begin{split}
        r(s,t) & = 0 \text{ for } t > s, \\
        r(t,t) & = c(t,t) = 1 \text{ for all } t\geq 0,\\
        c(s,t) & = c(t,s) \text{ for } s \geq t \geq 0.
        \end{split}
    \end{align*}

    Let $(c(s,t), r(s,t))$ be the solution to this system, whose uniqueness is established in \cite[Proposition 1.1]{Dem} for all positive temperatures. A direct computation yields that $(c(s,t)e^{\alpha(s-t)\Theta}, r(s,t)e^{\alpha(s-t)\Theta})$ solves the system \eqref{hard_C}-\eqref{hard_R} and satisfies the conditions \eqref{l2.4'}. By uniqueness, this pair coincides with $(C,R)$ of Theorem \ref{thm2}. For the rest of this section, we focus on the scalar system \eqref{c_sc}-\eqref{r_sc}, since the properties of its solution that we shall establish below readily translate to those of $(C,R)$ of Theorem \ref{thm2} via the identification above. 
    
    For $t\geq 0$, let $\ell(t)$ be the moment generating function of the semi-circular law on $[-2\beta,2\beta]$, i.e., the distribution whose density is given by $f_{\rm sc}(x) = \1_{[-2\beta,2\beta]}(x) \sqrt{4\beta^2-x^2}/(2\pi\beta^2)$. It is well known that the (even) moments of the semi-circular law are the Catalan numbers (see \cite[Section 2.1.1]{And}) and hence we have the representation \begin{align*}
        \ell(t) = \sum_{n=0}^\infty \frac{1}{n+1}\binom{2n}{n} \frac{(\beta t)^{2n}}{(2n)!}, \qquad t\geq 0.
    \end{align*}For $t < 0$, we set $\ell(t) = 0$. Using the upper bound of $4^n$ on the $n$-th Catalan number, it is easy to show that $\ell(t) \leq e^{2\beta t}$. Further, from the recursive relation characterizing Catalan numbers, we obtain by a straightforward computation that \begin{align}\label{bes_ker}
        \ell'(t) = \beta^2 \int_0^t \ell(t-u)\ell(u)du, \qquad t\geq 0.
    \end{align}
    
    For two functions $f, g$, we denote their convolution as $
        (f * g)(t) = \int_0^t f(t-u) g(u)du$.  For $n\geq 1$, we define the $n$-fold convolution inductively: $f^{*1} \equiv f$, and $f^{*(n+1)} \equiv f^{*n}*f$. 
    Define \begin{align}\label{def_ho}
        h(t )= \ell(2t) \qquad \text{and} \qquad \Omega(t) = \sum_{n=1}^\infty h^{*n}(t).    \end{align}
    Note that both $h$ and $\Omega$ are positive functions.

    Our first result provides an explicit representation of $(c(s,t),r(s,t))$, which is the spherical counterpart of the soft-sphere correlation representation in \cite[(3.5)]{Ben2}. Our approach is constructive: motivated by the representation in \cite[(3.5)]{Ben2}, we explicitly construct a candidate solution to the limiting equations and verify directly that it satisfies the system. Since Theorem \ref{thm2} already guarantees uniqueness, this verification identifies the constructed solution with the limiting correlation and response functions.
    
    \begin{prop}[Explicit formula for $(c,r)$]\label{exp_rep}
    For every $\beta > 0$, 
        $\Omega$ defined by the series \eqref{def_ho} is the unique solution to the convolution equation
        \begin{align}\label{renew}
        \Omega(t) = h(t) + (\Omega * h)(t) \quad \text{for }t\geq 0, \qquad \Omega(0) = 1,
    \end{align}
    and admits the representation \begin{align}\label{exp_o}
        \Omega(t) = e^{(1+4\beta^2)t}\Big(1- 4\beta^2 \int_0^t e^{-(1+4\beta^2)u} \ell(2u)du\Big), \qquad t\geq 0.
    \end{align}
        Further, for $s\geq  t \geq 0$, 
    \begin{align}\label{exp_cr}
        c(s,t) = c(t,s) = \frac{1}{\sqrt{\Omega(s)\Omega(t)}}\Big( \ell(s+t) + \int_0^t \Omega(u)\ell(s+t-2u)du\Big), \qquad r(s,t) = \sqrt{\frac{\Omega(t)}{\Omega(s)}}\ell(s-t).
    \end{align}
    \end{prop}
    
    \begin{proof}
        Fix $T<\infty$ and let $\|h\|_T:=\sup_{0\leq t\leq T}h(t)$. Then, \begin{align*}
            h^{*n}(t)& =\int_{\substack{u_1,\ldots,u_{n-1}\geq 0\\ \sum_{i=1}^{n-1}u_i\leq t}}h(u_1)\cdots h(u_{n-1})h\Big(t - \sum_{i=1}^{n-1}u_i\Big)du_1\cdots du_{n-1} \leq  \frac{\|h\|_T^nt^{n-1}}{(n-1)!}
        \end{align*}
for all $0\leq t\leq T$, and hence the series defining $\Omega$ in \eqref{def_ho} converges uniformly on $[0,T]$. Associativity of convolution and local uniform convergence of the series representation gives \begin{align*}
    \Omega(t)=h(t)+\sum_{n=2}^{\infty}h^{*n}(t)=h(t)+(\Omega*h)(t),
\end{align*} and a Gr\"onwall argument establishes uniqueness. Global uniqueness is achieved by patching solutions corresponding to finite time intervals, since the choice of $T$ was arbitrary.

Changing variables in the kernel equation \eqref{bes_ker} we obtain \begin{align*}
    h'(t) = 4\beta^2 \int_0^t h(t-u)h(u)du,
\end{align*}
and differentiating the convolution equation \eqref{renew} we have \begin{align}\label{diff_o}
    \Omega'(t) = h'(t) + \Omega(t) + \Omega*h'(t) =(1+4\beta^2)\Omega(t) - 4\beta^2 h(t),
\end{align}
and \eqref{exp_o} follows from variation of constants and $\Omega(0) = h(0) = 1$.

We now show that the expressions \eqref{exp_cr} satisfy the scalar system \eqref{c_sc}-\eqref{r_sc}. Note that uniqueness of the solution to the system was already established in \cite[Proposition 1.1]{Dem}. First, we observe that the boundary and symmetry conditions of the system are satisfied by the expressions \eqref{exp_cr}, as follows from the definition of $\ell$ and the convolution identity \eqref{renew}. Now, using Tonelli's theorem and \eqref{renew}, we obtain \begin{align*}
     \mu(s) & = \frac{1}{2} + 2\beta^2\int_0^s c(s,u)r(s,u)du\\
     &  = \frac{1}{2} + \frac{2\beta^2}{\Omega(s)}\Big(\int_0^s \ell(s+u)\ell(s-u) du +\int_0^s \int_0^u \Omega(v)\ell(s+u-2v)\ell(s-u)dvdu\Big)\\
     & = \frac{1}{2} + \frac{\ell'(2s)}{\Omega(s)} +\frac{1}{\Omega(s)}\int_0^s\Omega(v)\ell'(2(s-v))dv = \frac{1}{2} + \frac{1}{2\Omega(s)}\Big(h'(s) + (\Omega*h')(s)\Big) = \frac{\Omega'(s)}{2\Omega(s)},
\end{align*}
where the last equality follows from \eqref{diff_o}. Using this expression for $\mu(s)$ together with the convolution identity \eqref{renew},the right-hand side of \eqref{r_sc} becomes
\begin{align*}
    -\mu(s)r(s,t) + \beta^2\int_t^s r(s,u)r(u,t)du = -\frac{\Omega'(s) \Omega^{1/2}(t)}{2\Omega^{3/2}(s)}\ell(s-t) + \sqrt{\frac{\Omega(t)}{\Omega(s)}}\ell'(s-t),
\end{align*}
and this can be verified to be equal to $\partial_sr(s,t)$ by differentiating the expression of $r(s,t)$ in \eqref{exp_cr}. The verification that the expressions in \eqref{exp_cr} also satisfy \eqref{c_sc} proceeds along similar lines; we omit the details.
    \end{proof}
    
We now show that the fluctuation-dissipation relation holds for $(c_*,r_*)$, the stationary limit of $(c(s,t),r(s,t))$, at sufficiently high temperatures, i.e., $r_*(t) = -2c'_*(t)$. In general, when $\xi$ is a polynomial of higher degree, the entrywise non-linearities introduced by $\xi'$ and $\xi''$ are not equivariant under rotation by the phase $e^{\alpha\Theta}$ introduced by the non-reciprocity $\alpha \neq 0$. Consequently, it is not clear whether a fluctuation-dissipation relation continues to hold in this more general setting.

\begin{cor}(Fluctuation-dissipation at high temperature)
    For every $\beta < 1/2$, the stationary limits \begin{align*}
        c_*(\tau) \coloneqq \lim_{t\to \infty} c(t+\tau, t), \qquad r_*(\tau)\coloneqq \lim_{t\to\infty} r(t+\tau, t), \qquad \tau \geq 0,
    \end{align*}
    exist and are given by \begin{align*}
        c_*(\tau) = e^{-(1+4\beta^2)\tau/2} \int_0^\infty e^{-(1+4\beta^2)u} \ell(\tau + 2u) du, \qquad r_*(\tau) = e^{-(1+4\beta^2)\tau/2} \ell(\tau).
    \end{align*}
    In particular, \begin{align*}
        c_*(\tau) \leq (2\beta-1)^{-2}e^{-(2\beta-1)^2\tau/2},
    \end{align*}
    and the fluctuation-dissipation relation $r_*(t) =-2c_*'(t)$ holds for $t\geq 0$.
\end{cor}

\begin{proof}
    For simplicity, denote \begin{align*}
        b(t) = 1- 4\beta^2 \int_0^t e^{-(1+4\beta^2)u} \ell(2u)du
    \end{align*}
    so that from \eqref{exp_o}, we have that $\Omega(t) = e^{(1+4\beta^2)t} b(t)$. A straightforward computation using Fubini's theorem yields that 
    \begin{align}\label{int_l}
        \int_0^\infty e^{-(1+4\beta^2)u} \ell(2u)du = 1, \qquad \beta \leq 1/2.
    \end{align}
    It then follows that $1-4\beta^2 \leq b(t) \leq 1$ for all $t\geq 0$ and $b(t) \to 1-4\beta^2$ as $t \to \infty$. Fixing $\tau \geq 0$ we have from \eqref{exp_cr} we have \begin{align*}
        r(t + \tau, t)  = e^{-(1+4\beta^2)\tau/2} \ell(\tau)\sqrt{\frac{b(t)}{b(t+\tau)}}\to e^{-(1+4\beta^2)\tau/2} \ell(\tau)\eqqcolon r_*(\tau),
    \end{align*}
    as $t\to \infty$. For the correlation function $c$, we change variables in the integral appearing in \eqref{exp_cr} to write \begin{align*}
        c(t+\tau, t) & = \frac{\ell(2t+\tau)}{\sqrt{\Omega(t+\tau)\Omega(t)}} + \int_0^t \frac{\Omega(t-u)\ell(\tau + 2u)}{\sqrt{\Omega(t+\tau)\Omega(t)}}du.
    \end{align*}
    Since $\ell(t) \leq e^{2\beta t}$, the first term above is bounded by $e^{-(1-2\beta)^2(t + \tau/2)}/(1-4\beta^2)$, which vanishes as $t\to\infty$. The same estimate provides an integrable dominating bound on the integrand, while the convergence of $b(t)$ guarantees the pointwise convergence of the integrand to $e^{-(1+4\beta^2)(u+\tau/2)}\ell(\tau + 2u)$. The dominated convergence theorem therefore yields \begin{align*}
        c(t+\tau, t) & \to e^{-(1+4\beta^2)\tau/2} \int_0^\infty e^{-(1+4\beta^2)u} \ell(\tau + 2u) du \eqqcolon c_*(\tau).
    \end{align*}
    The bound on $c_*(t)$ follows from the fact that $\ell(t) \leq e^{2\beta t}$, and the fluctuation-dissipation relation is directly verified from the expressions of $c_*$ and $r_*$ above.
\end{proof}

\begin{remark}
For the SK model, the fluctuation-dissipation relation $r_* = -2c_*'$ can also be deduced directly from \cite[Theorem 1.3]{Dem}, without first computing explicit formulas for $(c_*,r_*)$. Indeed, by Theorem \ref{thm3}, the stationary limits of the functions $(C,R)$, after appropriate phase-correction, satisfy the scalar system \eqref{c_sc}--\eqref{r_sc}, and \cite[Theorem 1.3]{Dem} then implies that $r_*(t)=-2c_*'(t)$. Our approach additionally yields explicit formulas for $(c_*,r_*)$ and applies throughout the full high-temperature regime considered here, whereas we potentially lose some of the high-temperature regime owing to the condition $\beta<\beta_1 \leq 1/2$ in \cite[Theorem 1.3]{Dem}. For the soft-spherical case, an exponential decay, all the way up to criticality, was established in \cite[Proposition 3.2]{Ben2}, however no fluctuation-dissipation relation was established there.
\end{remark}

The next result establishes the aging phenomenon at low temperatures. Since we have already eliminated the non-reciprocal term arising from $\alpha \neq 0$ by working with the scalar system \eqref{c_sc}--\eqref{r_sc}, the result may be viewed as a refinement of \cite[Proposition 3.2]{Ben2}. Moreover, through the identification of $(C,R)$ from Theorem \ref{thm2} with the solution of \eqref{c_sc}--\eqref{r_sc}, as discussed above, we recover the asymptotic relation \cite[(7)]{Lor} obtained via a non-rigorous DMFT treatment; indeed, the additional factor $\cos(\alpha(s-t))$ is reminiscent of the phase shift induced by $e^{\alpha(s-t)\Theta}$. Combined with the preceding result, this shows that, for the non-reciprocally coupled system with purely quadratic Hamiltonian, $\beta = 1/2$ separates an aging phase from a time-translation-invariant regime.

We shall use the following notation in the next results: for two non-negative functions $f, g$ we say that $f\sim g$ if $\lim_{t\to \infty} f(t)/g(t) = 1$.

\begin{cor}(Aging at low temperature) \label{aging}
For every $\beta >1/2$, under the joint limit $t \to \infty$ and $s-t \to \infty$,
\begin{align}\label{eq:aging}
    c(s,t)\sim \Big(1-\frac{1}{2\beta}\Big)\phi(s,t), \qquad \text{where} \qquad \phi(s,t)\coloneqq \Big(\frac{2\sqrt{st}}{s+t}\Big)^{3/2}.
\end{align}
\end{cor}

\begin{proof}
    Let
\begin{align}\label{func_wt}
\wt \ell(t)=e^{-2\beta t} \ell(t), \qquad \wt h(t) =e^{-4\beta t}h(t) = e^{-4\beta t} \ell(2t) \qquad \text{and} \qquad \wt\Omega(t)=e^{-4\beta t}\Omega(t).
\end{align}
By Laplace's method, or alternatively by using the connection between $\ell$ and modified Bessel functions \cite[(10.32.2)]{Nis} and its large-argument asymptotic expansion \cite[(10.40.1)]{Nis}, we obtain
\begin{align*}
\ell(t)\sim \frac{e^{2\beta t}}{2\sqrt{\pi}(\beta t)^{3/2}}, \qquad \text{or equivalently, }\qquad \wt \ell(t) \sim \frac{(\beta t)^{-3/2}}{2\sqrt{\pi}}. 
\end{align*}
This implies that for $s\geq t$, \begin{align}\label{eq:lt1}
    \frac{\ell(s+t)}{\phi(s,t)\sqrt{\ell(2s)\ell(2t)}} = \frac{\wt \ell(s+t)}{\phi(s,t)\sqrt{\wt \ell(2s)\wt \ell(2t)}}  \sim 1,
\end{align}
and \begin{align}\label{tail_h}
    \wt h(t)\sim \frac{1}{2^{5/2}\sqrt{\pi}(\beta t)^{3/2}}.
\end{align}
 Recalling that $\ell(t)$ is the moment generating function of the semi-circular law on $[-2\beta,2\beta]$, a change of variables followed by an application of Tonelli's theorem yields that $\int_0^\infty \wt h(t)dt=1/(2\beta)<1$, where the inequality follows since $\beta>1/2$. As can be seen by a direct computation, $(\wt h, \wt \Omega)$ satisfy the convolution identity \eqref{renew}, and hence taking $f = \wt h$ and $g = \wt \Omega $ in Lemma \ref{lem_a3} gives
\begin{align}
\frac{\Omega(t)}{\ell(2t)}=\frac{\wt\Omega(t)}{\wt h(t)}\sim\frac{4\beta^2}{(2\beta-1)^2},
\end{align}
 and further \begin{align*}
    \int_0^\infty \wt\Omega(t)dt=\frac{1}{2\beta-1}\qquad \text{and} \qquad \wt\Omega(t)\leq c'(1+t)^{-3/2}, \quad t \geq 0,
\end{align*}
for some $c'>0$. Now, taking $\wt \ell$ to be the function $k$ and $g$ to be $\wt \Omega$ in Lemma \ref{lem_a4} we have
\begin{align}\label{eq:lt3}
     \frac{1}{\ell(s+t)}\int_0^t \Omega(u)\ell(s+t-2u)du = \frac{1}{\wt \ell(s+t)}\int_0^t \wt \Omega(u)\wt \ell(s+t-2u)du \to \frac{1}{2\beta -1}
\end{align} as $t \to \infty$ and $s-t \to \infty$. Using \eqref{exp_cr} to write \begin{align*}
    \frac{c(s,t)}{\phi(s,t)} & = \frac{\ell(s+t)}{\phi(s,t)\sqrt{\ell(2s)\ell(2t)}}\Big(1 + \int_0^t \frac{\Omega(u)\ell(s+t-2u)}{\ell(s+t)}du\Big)\Big(\frac{\ell(2s)\ell(2t)}{\Omega(s)\Omega(t)}\Big)^{1/2},
\end{align*}
    \eqref{eq:aging} follows from \eqref{eq:lt1}-\eqref{eq:lt3}.
\end{proof}

We finally establish the behavior of the correlation function at criticality.

\begin{cor}[Behavior at critical temperature]
    At $\beta = 1/2$, \begin{align*}
        c(s,t) \sim \frac{2^{3/2}}{\sqrt{\pi}} \frac{s^{1/4}t^{3/4}}{(s+t)\sqrt{s-t}}
    \end{align*}
    in the joint limit $t \to \infty$ and $s-t \to \infty$.
\end{cor}

\begin{proof}
    Since $\beta = 1/2$, the functions $\wt \ell, \wt h, \wt \Omega$ from \eqref{func_wt} are now given by\begin{align*}
    \wt \ell(t) = e^{-t}\ell(t), \qquad \wt h(t) = e^{-2t}\ell(2t) \qquad \text{and} \qquad \wt \Omega(t) = e^{-2t} \Omega(t).
\end{align*}
From \eqref{int_l}, we have that $\int_0^\infty \wt h(u)du = 1$, and thus, it follows from \eqref{exp_o} that $\wt \Omega(t) = \int_t^\infty \wt h(u)du.$ From \eqref{tail_h} and \cite[Theorem 2.45]{Fos} we have that \begin{align*}
    \wt \Omega(t)\sim t^{-1/2}/\sqrt{\pi}.
\end{align*}
Taking $ k= \wt \ell$ and $g = \wt \Omega$ in Lemma \ref{lem_a5} we have that \begin{align*}
    \Big(\wt \ell(s+t) + \int_0^t \wt \Omega(u)\wt \ell(s+t-2u)du \Big)\sim \frac{2\sqrt{2t}}{\pi(s+t)\sqrt{s-t}}
\end{align*}
as $t \to \infty$ and $s-t \to \infty$. On the other hand, we get \begin{align*}
    \sqrt{\wt \Omega(s) \wt\Omega(t)}\sim \frac{(st)^{-1/4}}{\sqrt{\pi}}.
\end{align*}
Combining the above displays in expression of $c(s,t)$ in \eqref{exp_cr} we have the desired result. 
\end{proof}

\appendix
\section{Appendix}

\subsection{Weight matrices on binary trees}
\begin{lemma}\label{lem_app}
Let $\B$ denote the set of finite rooted binary trees and $\omega:\B\to \cC(\Pi_T)^{2\times 2}$ be a family of weight matrices defined recursively as follows. For the trivial tree $\emptyset \in\B$, consisting of only the root vertex $\emptyset$, we set $\omega_\emptyset(x,y) = I_2$. For a tree $\T \in \B$  such that $\T_1,\T_2\in \B$ are the left and right children subtrees respectively, we define
\begin{align}
    \omega_{\T}(x,y)
    = \int_y^x \int_u^x
    \A[\omega_{\T_1}(v,u)]
    \omega_{\T_2}(u,y)dvdu, \tag{\ref*{wt_const}}
\end{align}
where $\A$ is a bounded linear operator on $\cC(\Pi_T)^{2\times 2}$, i.e., for any $M \in \cC(\Pi_T)^{2\times 2}$, \begin{align*}
    \|\A[M(x,y)]\| \leq \eta\|M(x,y)\| 
\end{align*}
for some constant $\eta> 0$. Then we have that \begin{align}\label{eqn:wt_estimate}
    \|\omega_\T(x,y)\| \leq \eta^{|\T|} \frac{(x-y)^{2|\T|}}{(2|\T|)!},
\end{align}
for every $\T \in \B$, where $|\T|$ denotes the number of internal (i.e., non-leaf) vertices of $\T$.

Suppose that $\wt \A$ is another bounded linear operator (with the same bound as $\A$) on $\cC(\Pi_T)^{2\times 2}$ and $(\wt \omega_T)_{T\in \B}$ is another family of weight matrices with $\wt \omega_\emptyset(x,y) = I_2$ and the other $\wt w_\T$ constructed as in \eqref{wt_const} using the operator $\wt \A$. Assume that there exists a non-negative integrable function $\zeta$ such that \begin{align}\label{diff_a}
    \|\A[M(x,y)] - \wt \A[M(x,y)]\| \leq \|M(x,y)\|\zeta(x,y) 
\end{align}
for every $M$ and $(x,y)\in \Pi_T$. Then, for all $\T$ with size at least one, \begin{align}\label{diff_w}
    \|\omega_\T (x,y) - \wt \omega_\T (x,y)\| \leq |\T|\eta^{|\T|-1}\frac{(x-y)^{2(|\T|-1)}}{(2(|\T|-1))!}\int_y^x\int_u^x \zeta(v,u)dvdu.
\end{align}
\end{lemma}

\begin{proof}
    The proof of \eqref{eqn:wt_estimate} essentially proceeds by induction. The base case $\T = \emptyset$ is easily checked since $|\T| = 0$. Now suppose that $\T$ is a tree with $\T_1$ and $\T_2$ as the left and right subtrees. We have \begin{align*}
    \|\omega_\T(x,y) \| & \leq \int_y^x \int_u^x \|\A[\omega_{\T_1}(v,u)]\|\|\omega_{\T_2}(u,y)\|dvdu \leq \eta\int_y^x \int_u^x\|\omega_{\T_1}(v,u)\|\|\omega_{\T_2}(u,y)\|dvdu,
\end{align*}
where in the last step we used the boundedness of $\A$. Plugging in the estimate \eqref{eqn:wt_estimate} for $\omega_{\T_1}$ and $\omega_{\T_2}$, we have \begin{align*}
    \|\omega_\T(x,y)\| & \leq \frac{\eta^{|\T_1| + |\T_2| + 1}}{(2|\T_1|)! (2|\T_2|)!} \int_y^x \int_u^x (v-u)^{2|\T_1|}(u-y)^{2|\T_2|}dvdu =\eta^{|\T_1| + |\T_2| + 1} \frac{(x-y)^{2(|\T_1| + |\T_2| + 1)}}{(2(|\T_1| + |\T_2| + 1))!},
\end{align*}
and we complete the induction by observing that $|\T| = |\T_1| + |\T_2| + 1$. 

To prove the second part of the lemma, we need some notation. For a tree $\T \in \B\setminus\{\emptyset\}$, let $\T^\circ$ denote the set of internal vertices of the tree. Consider the following quantities recursively defined for $\T \in \B\setminus\{\emptyset\}$ and $z \in \T^\circ$ on $(x,y) \in \Pi_T$. If $\T$ has root $\emptyset$ and left and right subtrees $\T_1$ and $\T_2 \in \B$, then define \begin{align}\label{defn_rho}
    \xi_{\T, z}(x,y) & = \begin{cases}
    \frac{1}{(2|\T_1|)! (2|\T_2|)!}\int_y^x \int_u^x (v-u)^{2|\T_1|}(u-y)^{2|\T_2|}\zeta(v,u) dvdu, & z = \emptyset,\\
    \frac{1}{(2|\T_2|)!}\int_y^x \int_u^x \xi_{\T_1,z}(v,u)(u-y)^{2|\T_2|}dvdu, & z\in \T_1,\\
    \frac{1}{(2|\T_1|)!}\int_y^x \int_u^x (v-u)^{2|\T_1|}\xi_{\T_2,z}(u,y) dvdu, & z\in \T_2.
    \end{cases}
\end{align} 

We shall first prove by induction that for every $\T \in \B\setminus\{\emptyset\}$, \begin{align}\label{w_ind_1}
    \|\Delta \omega_\T(x,y)\| \leq \eta^{|\T|-1}\sum_{z\in \T^\circ}\xi_{\T, z}(x,y),
\end{align}
where $\Delta\omega_\T(x,y) = \omega_\T (x,y) - \wt \omega_\T (x,y)$. We adopt this notation for the differences of other matrices and operators appearing later in the proof.

The base case is easily checked; when $\T$ is the binary tree where both children of the root $\emptyset$ are leaves, we have from \eqref{diff_a} that \begin{align*}
    \|\Delta \omega_\T (x,y)\| \leq \int_y^x\int_u^x \|\A[I_2] - \wt A[I_2]\|dvdu \leq  \int_y^x\int_u^x\zeta(v,u)dvdu = \xi_{\T, \emptyset}(x,y).
\end{align*}
Now suppose that $\T$ has left and right children subtrees as $\T_1$ and $\T_2$ respectively. We write the difference in \eqref{w_ind_1} as a telescoping sum \begin{align*}
     \Delta \omega_\T(x,y) & = \int_y^x \int_u^x \Delta \Big(\A[\omega_{\T_1}(v,u)]\omega_{\T_2}(u,y)\Big)dvdu\\
     & =\int_y^x \int_u^x\Big(\A [\omega_{\T_1}(v,u)] - \wt A[ \omega_{\T_1}(v,u)]\Big)\omega_{\T_2}(u,y)dvdu+ \int_y^x \int_u^x\wt \A [\Delta \omega_{\T_1}(v,u)]\omega_{\T_2}(u,y)dvdu\\
     & \qquad + \int_y^x \int_u^x\wt \A [\wt \omega_{\T_1}(v,u)]\Delta \omega_{\T_2}(u,y)dvdu.
 \end{align*}
 Using the boundedness of $\wt A$, \eqref{eqn:wt_estimate}, \eqref{diff_a} and the induction hypothesis \eqref{w_ind_1} for $\T_1$ and $\T_2$ we have \begin{align*}
    \|\Delta \omega_\T(x,y)\|&  \leq \frac{\eta^{|\T_1| + |\T_2|}}{(2|\T_1|)!(2|T_2|)!}\int_y^x \int_u^x (v-u)^{2|\T_1|} (u-y)^{2|\T_2|}\zeta(v,u)dvdu\\
    & \qquad +\frac{\eta^{|\T_1| + |\T_2|}}{(2|T_2|)!} \sum_{z\in \T_1^\circ} \int_y^x \int_u^x \xi_{\T_1,z}(v,u) (u-y)^{2|\T_2|}dvdu\\
    & \qquad + \frac{\eta^{|\T_1| + |\T_2|}}{(2|T_1|)!} \sum_{z\in \T_2^\circ} \int_y^x \int_u^x (v-u)^{2|\T_1|}\xi_{\T_2,z}(u,y)dvdu.
    \end{align*}
    Noting that $|\T| = |\T_1| +|\T_2| + 1$, the induction is complete from the definition of $\xi_{\T, z}$.

To obtain \eqref{diff_w} from \eqref{w_ind_1}, it remains to control the quantities $\xi_{\T,z}$, which we do now. In particular, we show that for $\T \in \B \setminus\{\emptyset\}$ and $z\in \T^\circ$, there are integers $n_1, n_2, n_3$ depending on $\T$ and $z$ such that $n_1 + n_2 + n_3 = 2(|\T|-1)$ and \begin{align}\label{est_rho}
    \xi_{\T, z}(x,y) = \frac{1}{n_1! n_2! n_3!}\int_y^x \int_u^x(x-v)^{n_1}(v-u)^{n_2}(u-y)^{n_3}\zeta(v,u)dvdu.
\end{align}
The proof is again by induction. When $\T$ is the binary tree where both children of the root $\emptyset$ are leaves, we have \eqref{est_rho} holding true with $n_1 = n_2 = n_3 = 0$. If $\T$ is such that $\emptyset$ is the root and $\T_1$ and $\T_2$ are the left and right children subtrees respectively, there are few cases to consider. When $z =\emptyset$, it follows immediately from \eqref{defn_rho} that $n_1 = 0$, $n_2 = 2|\T_1|$ and $n_3 = 2|\T_2|$. When $z \in \T_1$, we have that for some integers $n_1, n_2, n_3$ with $n_1 + n_2 + n_3 = 2(|\T_1|-1)$, $\xi_{\T_1, z}(x,y)$ is given by \eqref{est_rho}. Using this in \eqref{defn_rho} we have by Fubini's theorem \begin{align*}
    \xi_{\T,z}(x,y) & = \frac{1}{(2|\T_2|)!n_1! n_2! n_3!}\int_y^x \int_u^x \int_u^v \int_w^v (v-r)^{n_1}(r - w)^{n_2}(w-u)^{n_3}(u-y)^{2|\T_2|}\zeta(r,w)drdwdvdu\\
    & = \frac{1}{(2|\T_2|)!n_1! n_2! n_3!} \int_y^x \int_w^x \zeta(r,w)(r-w)^{n_2} \Big[\int_y^w(w-u)^{n_3}(u-y)^{2|\T_2|} du\Big]\Big[\int_r^x(v-r)^{n_1}dv\Big]drdw\\
    & = \frac{1}{(n_1+1)! n_2! (2|\T_2|+ n_3+1)!} \int_y^x \int_w^x (x-r)^{n_1 + 1}(r-w)^{n_2}(w-y)^{2|\T_2| + n_3 + 1}\zeta(r,w)drdw
\end{align*}
and the induction is complete once we observe that the sum of the exponents in the last integral above equals $2(|\T|-1)$. The case $z \in \T_2$ is covered by a similar argument.

With \eqref{est_rho} in hand, the observation that for any $n_1, n_2, n_3 \geq 0$ and $y\leq u\leq v\leq x$, \begin{align*}
    \frac{(x-v)^{n_1}(v-u)^{n_2}(u-y)^{n_3}}{n_1!n_2! n_3!} \leq \frac{(x-y)^{n_1 + n_2 + n_3}}{(n_1+n_2+n_3)!}  
\end{align*}
immediately yields the upper bound \begin{align*}
    \xi_{\T,z}(x,y)\leq \frac{(x-y)^{2(|\T|-1)}}{(2(|\T|-1))!}\int_y^x\int_u^x \zeta(v,u)dvdu,
\end{align*}
and \eqref{diff_w} follows immediately from \eqref{w_ind_1}.
\end{proof}

\subsection{Some technical lemmas}

\begin{claim}\label{claim_a1}
    If $f$ satisfies \eqref{cond}, then for every $\rho_0 > 0$ there exists $\eta = \eta(\rho_0)< \infty$ such that for all $u, t\geq 0$, \begin{align*}
        2tf'(t) - u\sum_{p=1}^m\frac{|a_p|(2t)^{p/2}}{(p-1)!} - 1 \geq \rho_0 t - \eta(1+u)^{m/(2\delta) + 1}. 
    \end{align*}
\end{claim}

\begin{proof}
    Let $\iota = \inf_{t\geq 0}(f'(t) - \rho t^{m/2 + \delta  -1})$. Then, it follows that for any $t\geq 0$, we have \begin{align}\label{a1}
        2tf'(t) \geq 2\rho t^{m/2 + \delta} + 2\iota t.
    \end{align}
    For $1\leq p\leq m$, denote $\iota_p = |a_p|2^{p/2}/(p-1)!$ and $r_p = (m + 2\delta)/p > 1$. By Young's inequality we get that \begin{align}\label{a2}
         \iota_put^{p/2} \leq \frac{\rho}{m}t^{m/2 + \delta} + \frac{1}{r_p'}\Big(\frac{m}{\rho r_p}\Big)^{r_p'/r_p}(\iota_p u)^{r_p'}
    \end{align}
    where $r_p'$ is the H\"older conjugate of $r_p$. 
    
    Note that $r_p' \leq r_m' = m/(2\delta) + 1$, and hence for all $u \geq 0$, $u^{r_p'} \leq (1+ u)^{m/(2\delta) + 1}$. Thus, summing \eqref{a2} over $p = 1, \ldots, m$, we get that \begin{align*}
        u\sum_{p=1}^m \iota_pt^{p/2} \leq \rho t^{m/2 + \delta} + \eta(1+u)^{m/(2\delta) + 1}
    \end{align*}
    for some constant $\eta$. Combining this with \eqref{a1} we get that for all $\rho_0 > 0$\begin{align*}
        2tf'(t) - u \sum_{p=1}^m \iota_pt^{p/2}-1 & \geq \rho t^{m/2 + \delta} + 2\iota t - 1 - \eta(1+u)^{m/{(2\delta)+1}} \\
        & \geq \rho_0 t - \sup_{t\geq 0} \big(1 + (\rho_0-2\iota)t - \rho t^{m/2 + \delta}\big) - \eta(1+u)^{m/{(2\delta)+1}}\\
        & \geq \rho t - \eta(1+u)^{m/{(2\delta)+1}},
    \end{align*}
    where the last line follows by adjusting $\eta$ suitably.
\end{proof}

For the following lemmas, we recall the following notation: for functions $f, g$, we denote their convolution as $(f * g)(t) = \int_0^t f(t-u) g(u)du$.  For $n\geq 1$, we define the $n$-fold convolution inductively by $f^{*1} \equiv f$, and $f^{*(n+1)} \equiv f^{*n}*f$. If $f, g$ are non-negative and $\lim_{t\to \infty} f(t)/g(t) = 1$, we write $f\sim g$.

The following two lemmas have been used in the proof of the aging phenomena in the low-temperature regime for quadratic spin glasses.

\begin{lemma}\label{lem_a3}
Let $f:[0,\infty)\to [0,\infty)$ be a bounded function such that $f(t)\sim \eta t^{-3/2}$ for some $\eta>0$.  If $\tau \coloneqq \int_0^\infty f(t)dt < 1$, then the function  $g(t) = \sum_{n=1}^\infty f^{*n}(t)$ is the unique non-negative bounded solution to the convolution equation \begin{align}
    g(t) = f(t) + (f*g)(t), \qquad t\geq 0, \tag{\ref*{renew}}
\end{align} 
and further, $g$ satisfies 
\begin{align*}
    g(t) \sim \frac{f(t)}{(1-\tau)^2}, \qquad \int_0^\infty g(t)dt = \frac{\tau}{1-\tau}, \qquad \text{and} \qquad g(t) \leq c'(1+t)^{-3/2} \text{ for all } t\geq 0,
\end{align*}
for some constant $c'< \infty$.
\end{lemma}

\begin{proof}
   The proof that $g$ is the unique solution satisfying \eqref{renew} was provided in Proposition \ref{exp_rep}, so we turn to establishing the other properties.
   
   Let $p(t) = f(t)/\tau$ so that $p$ is a bounded probability density. Since $p$ decays as $t^{-3/2}$ at infinity, we have that for every fixed $n\geq 1$, \begin{align}\label{a3.1}
       p^{*n}(t)\sim n p(t),
       \end{align}
   (see \cite[Corollary 4.10]{Fos}). Also, from Kesten's bound \cite[Theorem 4.11]{Fos} we have that for all $\varepsilon > 0$, there exists $c_\varepsilon , t_\varepsilon>0$ such that for all $n\geq 1$, and $t \geq t_\varepsilon$,
   \begin{align}\label{a3.2}
  p^{*n}(t) \leq c_\varepsilon(1+\varepsilon)^n p(t).
\end{align}
 Choose $\varepsilon$ such that $\tau(1+\varepsilon)< 1$. Since \begin{align*}
     \frac{g(t)}{f(t)} = \sum_{n=1}^\infty \tau^{n-1}\frac{p^{*n}(t)}{p(t)},
 \end{align*}
\eqref{a3.1}, \eqref{a3.2} and the dominated convergence theorem yield \begin{align*}
    \frac{g(t)}{f(t)} \to \sum_{n=1}^\infty n\tau^{n-1} = \frac{1}{(1-\tau)^2}
\end{align*}
 as $t \to \infty$. The same Kesten bound also yields finite constants $c', c''>0$, such that 
 \begin{align*}
     g(t)\leq c''f(t)\leq c'(1+t)^{-3/2}
 \end{align*}
 for all $t$ sufficiently large. Since $\tau <1$, the bound $\|f^{*n}\|_\infty\leq \|f\|_\infty \tau^{n-1}$ (here $\|\cdot\|_\infty$ is the supremum norm), implies that $g$ is bounded on all of $[0,\infty)$, and hence by increasing $c'$ if necessary, we have $g(t) \leq c'(1+t)^{-3/2}$ for all $t\geq 0$. Finally, Tonelli's theorem yields that \begin{align*} 
 \int_0^\infty g(t)dt = \sum_{n=1}^\infty \tau^n = \frac{\tau}{(1-\tau)}.
 \end{align*}
\end{proof}

\begin{lemma} \label{lem_a4}
Let $g, k:[0,\infty) \to [0, \infty)$ be locally bounded functions, such that \begin{align*}
        \sigma\coloneqq \int_0^\infty g(t)dt < \infty, \qquad g(t)\leq c'(1+t)^{-3/2}, \qquad k(t) \sim c''t^{-3/2}
    \end{align*}
    for some constants $0<c', c'' <\infty $. Then for any sequence $(s_n)_{n\geq 1}, (t_n)_{n\geq 1}$ such that $t_n \to \infty$ and $s_n - t_n \to \infty$, \begin{align*}
        \frac{1}{k(s_n+t_n)}\int_0^{t_n} g(u)k(s_n + t_n - 2u)du \to \sigma.
        \end{align*}
\end{lemma}

\begin{proof}
    Throughout the proof, $\eta$ shall denote some constant that may depend only on $g, k$ and whose value of $\eta$ may change from line to line. 
    
    Let $T<\infty$ be fixed. From the uniform convergence theorem (see \cite[Theorem 1.5.2]{Bin} and in particular, \cite[(2.18)]{Fos}) we have \begin{align*}
        \sup_{0\leq u\leq T} \Big|\frac{k(s_n +t_n - 2u)}{k(s_n + t_n)} - 1\Big| \to 0. 
    \end{align*}
    As a result, \begin{align}\label{a4.1}
        \frac{1}{k(s_n + t_n)}\int_0^T g(u) k(s_n+t_n - 2u)du \to \int_0^T g(u)du.
    \end{align}
    The asymptotic rate of $k$ yields that for all $n$ sufficiently large and every $0\leq u\leq (s_n + t_n)/4,$ \begin{align*}
        \frac{k(s_n+ t_n - 2u)}{k(s_n + t_n)} \leq 3\Big(\frac{s_n+ t_n}{s_n + t_n-2u}\Big)^{3/2} \leq 6\sqrt{2}. 
    \end{align*}
    Thus, \begin{align}
        \frac{1}{k(s_n + t_n)}\int_T^{t_n \wedge ((s_n + t_n)/4)} g(u)k(s_n + t_n - 2u)du \leq \eta\int_T^\infty g(u)du.
    \end{align}
    Finally, if the interval $[(s_n + t_n)/4, t_n]$ is non-empty, we use the asymptotic bounds on $g$ and $k$ to obtain that for all $n$ sufficiently large, $g(u)/k(s_n+t_n) \leq \eta$ for all $u \in [(s_n + t_n)/4, t_n]$. This gives \begin{align}\label{a4.3}
         \frac{1}{k(s_n + t_n)}\int_{(s_n + t_n)/4}^{t_n} g(u)k(s_n + t_n - 2u)du \leq \eta \int_{(s_n + t_n)/4}^{t_n} k(s_n + t_n - 2u)du \leq \frac{\eta}{2}\int_{s_n - t_n}^\infty k(u)du,
    \end{align}
    which vanishes when $n\to\infty$ since $s_n - t_n \to \infty$. Combining \eqref{a4.1}-\eqref{a4.3}, we therefore get \begin{align*}
        \limsup_{n\to \infty} \Big|\frac{1}{k(s_n + t_n)}\int_0^{t_n} g(u)k(s_n + t_n - 2u)du - \sigma\Big| \leq \eta \int_T^\infty g(u)du.
    \end{align*}
    Letting $T \to \infty$ and using the fact that $g$ is integrable, the claim is established.
\end{proof}

The following lemma was used to analyze the behavior of the correlation function for quadratic spin glasses at critical temperature.

\begin{lemma}\label{lem_a5}
  Let $g, k:[0,\infty) \to [0, \infty)$ be locally bounded functions, such that \begin{align*}
        k(t) \sim c't^{-3/2}, \qquad g(t) \sim c''t^{-1/2},
    \end{align*}
    for some constants $0<c', c'' <\infty $. Then for any sequence $(s_n)_{n\geq 1}, (t_n)_{n\geq 1}$ such that $t_n \to \infty$ and $s_n - t_n \to \infty$, \begin{align*}
       \frac{1}{\varphi(s_n,t_n)}\Big(k(s_n+t_n) + \int_0^{t_n}g(u) k(s_n +t_n -2u)du\Big) \to 1,
        \end{align*}
        where \begin{align*}
    \varphi(s,t) \coloneqq c'c''\int_0^t u^{-1/2} (s+t-2u)^{-3/2}du =\frac{2c'c''\sqrt{t}}{(s+t)\sqrt{s-t}}, \qquad s> t.
\end{align*}
\end{lemma}

\begin{proof}
    Set $\varepsilon\in (0,1)$. Then for some $T = T(\varepsilon) \geq 1$ sufficiently large we have \begin{align*}
    \Big|\frac{k(t)}{c't^{-3/2}}-1\Big| \leq \varepsilon, \qquad \text{and}\qquad \Big|\frac{g(t)}{c'' t^{-1/2}}-1\Big| \leq \varepsilon, \qquad \text{for all }t \geq T.
\end{align*}
For all $n$ sufficiently large, $t_n \geq T$ and $s_n - t_n \geq T$. So for $u \in [T, t_n]$, we have $s_n + t_n - 2u \geq s_n -t_n \geq T$. Hence on $[T, t_n]$, \begin{align*}
    \Big|\frac{g(u)k(s_n+t_n - 2u)}{c'c''u^{-1/2}(s_n+t_n-2u)^{-3/2}}-1\Big| \leq 2\varepsilon + \varepsilon^2 \leq 3\varepsilon, 
\end{align*}
and hence \begin{align*}
    \Big| \int_T^{t_n} g(u) k(s_n + t_n -2u)du - c'c'' \int_T^{t_n} u^{-1/2} (s_n +t_n -2u)^{-3/2}du\Big| \leq 3\varepsilon \varphi(s_n,t_n).
\end{align*}
It remains to control the difference on $[0,T]$. Since $k(v) \sim c' t^{-3/2}$, there exist constants $C$ and $t_0$ such that $k(t)\leq Ct^{-3/2}$ for all $t \geq t_0$. For all sufficiently large $n$, $s_n + t_n - 2T \geq (s_n + t_n)/2 \geq t_0$, and thus using the local boundedness of $g$, we obtain \begin{align*}
    k(s_n + t_n) + \int_0^T g(u)k (s_n+ t_n -2u)du \leq C'(s_n + t_n)^{-3/2},\\
    c'c'' \int_0^T u^{-1/2} (s_n +t_n-2u)^{-3/2}du\leq C' (s_n+t_n)^{-3/2},
\end{align*} 
for some constant $C'$ that depends on $T$ but not on $n$. From the definition of $\varphi$, it follows that \begin{align*}
    \frac{(s_n + t_n)^{-3/2}}{\varphi(s_n,t_n)} \leq \frac{1}{2c'c'' \sqrt{t_n}} \to 0,
\end{align*} and combining with the above two displays yields\begin{align*}
    \limsup_{n\to \infty}\Big|\frac{1}{\varphi(s_n,t_n)}\Big(k(s_n+t_n) + \int_0^{t_n}g(u) k(s_n +t_n -2u)du\Big) - 1\Big| \leq 3\varepsilon.
\end{align*}
Letting $\varepsilon \to 0$ finishes the claim.
\end{proof}

\bibliographystyle{siam}
\bibliography{ref_nr}
\end{document}